\documentclass[11pt,reqno,sumlimits]{amsart}

\usepackage{amssymb, amscd, amsmath, epsfig, mathtools}
\usepackage[utf8]{inputenc}
\usepackage{amsthm}
\usepackage{enumerate}
\usepackage{xcolor}
\usepackage{scalerel}
\usepackage{soul}
\usepackage{tikz-cd}
\usepackage{esint}

\usepackage[margin=1.0in]{geometry}

\newtheorem{theorem}{Theorem}[section]
\newtheorem*{theorem*}{Theorem}
\newtheorem{definition}{Definition}[section]
\newtheorem{corollary}{Corollary}[section]
\newtheorem{lemma}{Lemma}[section]
\newtheorem{proposition}{Proposition}[section]
\theoremstyle{definition}
\newtheorem{remark}{Remark}[section]

\newcommand{\R}{\mathbb R}

\newcommand{\calC}{\mathcal C}
\newcommand{\calL}{\mathcal L}
\newcommand{\dvol}{ d\text{Vol}_{g}}

\begin{document}

\title[Prescribed Scalar Curvatures on Compact Manifolds]{Prescribed Scalar Curvatures on Compact Manifolds Under Conformal Deformation}
\author[J. Xu]{Jie Xu}
\address{
Department of Mathematics and Statistics, Boston University, Boston, MA, USA}
\email{xujie@bu.edu}
\address{
Institute for Theoretical Sciences, Westlake University, Hangzhou, Zhejiang Province, China}
\email{xujie67@westlake.edu.cn}

\date{}							

\begin{abstract} We give sufficient and ``almost" necessary conditions for the prescribed scalar curvature problems within the conformal class of a Riemannian metric $ g $ for both closed manifolds and compact manifolds with boundary, including the interesting cases $ \mathbb{S}^{n} $ or some quotient of $ \mathbb{S}^{n} $, in dimensions $ n \geqslant 3 $, provided that the first eigenvalues of conformal Laplacian (with appropriate boundary conditions if necessary) are positive. When the manifold is not some quotient of $ \mathbb{S}^{n} $, we show that, on one hand, any smooth function that is a positive constant within some open subset of the manifold with arbitrary positive measure, and has no restriction on the rest of the manifold, is a prescribed scalar curvature function of some metric under conformal change; on the other hand, any smooth function $ S $ is almost a prescribed scalar curvature function of Yamabe metric within the conformal class $ [g] $ in the sense that an appropriate perturbation of $ S $ that defers with $ S $ within an arbitrarily small open subset is a prescribed scalar curvature function of Yamabe metric. When the manifold is either $ \mathbb{S}^{n} $ or $ \mathbb{S}^{n} \slash \Gamma $ with Kleinian group $ \Gamma $ we show that any positive function that satisfies a technical analytical condition, called CONDITION B, can be realized as a prescribed scalar curvature functions on these manifolds.
\end{abstract}

\maketitle

\section{Introduction}
The Kazdan-Warner problem for prescribed scalar curvatures asks when a smooth function $ S \in \calC^{\infty}(M) $ on either closed manifolds $ (M, g) $ or compact manifolds $ (\bar{M}, g) $ with smooth boundary $ \partial M $ with dimensions at least three, is the scalar curvature $ S = R_{\tilde{g}} $ of a Yamabe metric $ \tilde{g} $ conformal to $ g $. In this paper, we prove some sufficient conditions for the Kazdan-Warner problem when the first eigenvalue of the conformal Laplacian $ -a\Delta_{g} + R_{g} $ (with appropriate boundary conditions if necessary) is positive, as discussed in Theorem \ref{intro:thm1}, Theorem \ref{intro:thm2} and Theorem \ref{intro:thm3}; these sufficient conditions are also the ``almost" necessary conditions in the sense of Theorem \ref{intro:thm4} and \ref{intro:thm5} below. The results are different for the interesting cases $ M = \mathbb{S}^{n} \slash \Gamma $, where $ \Gamma $ is some group related to the Kleinian group on $ \mathbb{S}^{n} $, including the trivial group; in particular, we show that there are obstructions for the choices of prescribed scalar curvature functions on $ \mathbb{S}^{n} \slash \Gamma $. 

Instead of classical global variational method, we introduce a new systematic procedure including the solvability of local Yamabe-type equations, the construction of sub- and super-solutions, and the monotone iteration schemes to treat this type of problem. One key feature is the smoothness (not piecewise $ \calC^{\infty} $ only) of the super-solutions; in addition, we start with the background metric $ g $ with scalar curvature $ R_{g} $ that is negative somewhere by some conformal change when constructing the sub- and super-solutions; furthermore, there is a difference between prescribing constant and non-constant scalar curvature functions: when the prescribing function is globally constant, we need to introduce the perturbation of the conformal Laplacian \cite{XU4, XU5, XU3}, which is not required for non-constant prescribing functions, due to technical reason.

In general, this type of problem reduces to a semi-linear elliptic PDEs with Robin boundary conditions if necessary. To write this question, we set $ a = \frac{4(n - 1)}{n - 2}, p = \frac{2n}{n - 2} $ with $ n = \dim M $ or $ n = \dim \bar{M} $ with $ n \geqslant 3 $. Throughout this article, we may also assume that the manifold is connected since otherwise our analysis applies equally to each connected components instead. As above, $ R_{g} $ is the scalar curvature of the metric $ g $, $ h_{g} $ is the mean curvature on $ \partial M $, and $ -\Delta_{g} $ is the positive definite Laplace-Beltrami operator. For a given function $ S \in \calC^{\infty}(M) $, the metric $ \tilde{g} : = u^{p -2} g $ has the prescribed scalar curvature $ S $ if and only if some positive, smooth function $ u $ solves
\begin{equation}\label{intro:eqn1}
\Box_{g}u : = -a\Delta_{g} u + R_{g} u = S u^{p-1} \; {\rm on} \; M
\end{equation}
when $ M $ is closed. Similarly, $ S \in \calC^{\infty}(\bar{M}) $ can be realized as a scalar curvature function of some metric $ \tilde{g} = u^{p-2} g $ with minimal boundary, i.e. the mean curvature of $ \tilde{g} $ is zero, if and only if some positive, smooth function $ u $ solves
\begin{equation}\label{intro:eqn2}
\Box_{g} u : = -a\Delta_{g} u + R_{g} u = S u^{p-1} \; {\rm in} \; M, B_{g} u : = \frac{\partial u}{\partial \nu} + \frac{2}{p-2} h_{g} u = 0 \; {\rm on} \; \partial M
\end{equation}
when $ \bar{M} $ is compact with smooth boundary. Here $ \nu $ is the outward unit normal vector along $ \partial M $. In either case, the sign of the Yamabe invariant
\begin{equation}\label{intro:eqn3}
\begin{split}
\lambda(M) & : = \inf_{u \neq 0, u \in H^{1}(M, g)} \frac{\int_{M} \left(a\lvert \nabla_{g} u \rvert^{2} + R_{g} u^{2} \right) \dvol}{\left( \int_{M} u^{p} \dvol \right)^{\frac{2}{p}}} : = \inf_{u \neq 0, u \in H^{1}(M, g)} Q(u); \\
\bar{\lambda}(\bar{M}) & : = \inf_{u \neq 0, u \in H^{1}(M, g)}  \frac{\int_{M} \left(a\lvert \nabla_{g} u \rvert^{2} + R_{g} u^{2} \right) \dvol + \int_{\partial M} \frac{2a}{p-2} h_{g} u^{2} dS}{\left( \int_{M} u^{p} \dvol \right)^{\frac{2}{p}}} : = \inf_{u \neq 0, u \in H^{1}(M, g)} \bar{Q}(u).
\end{split}
\end{equation}
plays an important role in understanding the Yamabe problem, as shown in \cite{XU4, XU5, XU3} when $ S $ is a constant function. In fact, the sign of the Yamabe invariants equals the sign of the first eigenvalue of the conformal Laplacian (with the Robin condition in (\ref{intro:eqn2}) in the case of $ \bar{M} $). Historically the most interesting cases are when $ \lambda(M) $, $ \bar{\lambda}(\bar{M}) $ are positive. For the Kazdan-Warner problem and the Yamabe problem on closed manifolds with $ \lambda(M) < 0 $ and $ S $ negative almost everywhere, see \cite{KW} and \cite{PL}; for the compact manifolds with smooth boundary case with $ \bar{\lambda}(\bar{M}) < 0 $ and $ S < 0 $ everywhere, see \cite{CHS}. In two dimensional case, the prescribed Gauss curvature problem was extensively studied by Kazdan and Warner, see \cite{KW2}. Although the zero eigenvalue of the conformal Laplacian generically does not happen \cite{GHJL}, the necessary and sufficient conditions of the Kazdan-Warner problem with dimensions at least 3 for this case were studied in \cite{KW}, \cite{ESCS}, etc.; on closed Riemann surface, the necessary and sufficient conditions for the case $ \chi(M) = 0 $ was completely given by Kazdan and Warner, see \cite{KW2}. Very recently a comprehensive study for necessary and sufficient conditions of prescribing scalar curvatures on compact manifolds (possibly with boundary) provided that the first eigenvalue of the conformal Laplacian is zero is given in \cite{XU8}.

In contrast to the usual global variational methods, applying our local-to-global iterative method can get the following results: when $ \lambda(M) > 0 $ on a closed manifold $ (M, g) $, except in the case $ M = \mathbb{S}^{n} \slash \Gamma $ for some Kleinian group $ \Gamma $, we show that any smooth real-valued function $ S $ satisfying $ S \equiv c > 0 $ for any constant $ c > 0 $ on any connected open submanifold $ O \subset \bar{O} \subset M $ with smooth boundary $ \partial O $ is the prescribed scalar curvature of some Yamabe metric $ \tilde{g} \in [g] $, where $ [g] $ is the conformal class of $ g $ and thus $ \tilde{g} $ is conformal to $ g $. Analogously, we show that when $ \bar{\lambda}(\bar{M}) > 0 $ on compact manifolds $ (\bar{M}, g) $ with smooth boundary, again except the cases when $ M = \mathbb{S}^{n} \slash \Gamma $, a function $ S $ which is positive constant on any inner connected open submanifold is the prescribed scalar curvature of a Yamabe metric $ \tilde{g} \in [g] $; in addition, $ \partial M $ is minimal with respect to $ \tilde{g} $. In particular, these results include the solution of the Yamabe problem and the Escobar problem, when $ S $ is a positive constant function. 

To summerize the main results, set
\begin{equation}\label{intro:eqn4}
\begin{split}
& \mathcal{A} : = \lbrace S \in \calC^{\infty}(M) : S \equiv \lambda > 0 \; {\rm on} \; \bar{O}, \text{$ \lambda > 0 $ is an arbitrary constant, } \\
& \qquad \text{ $ O \subset \bar{O} \subset M $ is an arbitrary open submanifold with smooth boundary $ \partial O $} \rbrace; \\
& \mathcal{A}' : = \lbrace S \in \calC^{\infty}(\bar{M}) : S \equiv \lambda > 0 \; {\rm on} \; \bar{O}, \text{$ \lambda > 0 $ is an arbitrary constant, } \\
& \qquad \text{ $ O \subset \bar{O} \subset M $ is an arbitrary open interior submanifold with smooth $ \partial O $} \rbrace.
\end{split}
\end{equation}
The first two main results with sufficient conditions of Kazdan-Warner problem are as follows:
\begin{theorem}\label{intro:thm1}
Let $ (M, g) $ be a closed manifold with $ n = \dim M \geqslant 3 $ which is not $ \mathbb{S}^{n} $ for some Kleinian group $ \Gamma $. If $ \lambda(M) > 0 $, then for every $ S \in \mathcal{A} $, the PDE (\ref{intro:eqn1}) admits a positive smooth solution $ u \in \calC^{\infty}(M) $ and $ S $ is the prescribed scalar curvature of the metric $ \tilde{g} = u^{p-2} g $.
\end{theorem}
\begin{theorem}\label{intro:thm2}
Let $ (\bar{M}, g) $ be a compact manifold with smooth boundary $ \partial M $, $ n = \dim \bar{M} \geqslant 3 $. If $ \bar{\lambda}({\bar{M}}) > 0 $, then for every $ S \in \mathcal{A}' $, the PDE (\ref{intro:eqn2}) admits a positive smooth solution $ u \in \calC^{\infty}(\bar{M}) $ and $ S $ is the prescribed scalar curvature of the metric $ \tilde{g} = u^{p-2} g $ with minimal boundary.
\end{theorem}
\medskip

When $ M = \mathbb{S}^{n} \slash \Gamma $ for a Kleinian group $ \Gamma $ (possibly the identity group), we show that there are obstructions for the choices of the prescribed scalar curvature functions. It is clear that not every function on $ \mathbb{S}^{n} $ can be realized as a prescribed scalar curvature function, see \cite{KW}, \cite{BE} and \cite{MaMa}. Geometrically, the speciality of the $ n $-sphere is the ``symmetry". Analytically, the obstructions are related to the fact that $ n $-sphere can be covered by two charts given by stereographic projections; in addition, each stereographic projection is a conformal diffeomorphism with the conformal factors. Thus any local solution of the Yamabe equation with Dirichlet boundary condition on appropriate chosen open subset $ O \subset \mathbb{S}^{n} $ might have some special necessary conditions, since we can localize the Yamabe equation with Dirichlet condition on $ O $ in two different ways, both have the same local expression. This motivates the following definition.
\begin{definition}\label{intro:def1} Let $ (\bar{M}, g) $ be a connected, compact manifold, with or without smooth boundary, and with $ \dim \bar{M} \geqslant 3 $. We say that the manifold $ (\bar{M}, g) $ satisfying {\bf{SCENARIO A}} if there exist at least two points $ \rho, \rho' \in \bar{M} $ and an open subset $ O \subset M \subset \bar{M} $ such that:

(i) The first eigenvalue of conformal Laplacian $ \Box_{g} \varphi = \eta_{1} \varphi $, possibly with Robin condition $ B_{g} \varphi = 0 $ as mentioned in (\ref{intro:eqn2}), is positive; 

(ii) $ \rho \in O $ and $ \rho' \notin \bar{O} $, or vice versa;

(iii) There exists two charts $ (U_{1}, \sigma_{1}) $ and $ (U_{2}, \sigma_{2}) $ such that both maps $ \sigma_{1} $ and $ \sigma_{2} $ are conformal diffeomorphisms; furthermore, $ O \subset U_{1} \cap U_{2} $, $ \rho \in U_{1} $ and $ \rho' \in U_{2} $;

(iv) For $ g_{e} $ the Euclidean metric, we have
\begin{equation}\label{intro:eqn5}
\left( \sigma_{1}^{-1} \right)^{*} g = \left( \sigma_{2}^{-1} \right)^{*} g = v^{2} g_{e}, \sigma_{1}(\rho) = \sigma_{2}(\rho')
\end{equation}
for a positive, smooth function $ v $.
\end{definition}
The SCENARIO A provides a model on which the choices of prescribed scalar curvature functions are restricted. Obviously $ \mathcal{S}^{n} $ is of SCENARIO A. After some analysis, it follows from the result of Schoen and Yau \cite{SY} that any manifold satisfying the SCENARIO A  must be of the form $ \mathbb{S}^{n} \slash \Gamma $ with some Kleinian group $ \Gamma $. For manifolds with the SCENARIO A, we can find obstructions to prescribed scalar curvature based on our iterative scheme in terms of the existence of the local solution of the Yamabe equation. For notation, let $ Q : \mathbb{S}^{n} \slash \Gamma \rightarrow \R $ be a smooth function, let $ Q' $ be the lift of $ Q $ to $ \mathbb{S}^{n} $, and let $ Q'' : \R^{n+1} \backslash \lbrace 0 \rbrace \rightarrow \R $ be defined by $ Q''(x) : = Q'\left(\frac{x}{\lvert x \rvert}\right) $. 
\begin{definition}\label{intro:def2}
Let $ (\mathbb{S}^{n} \slash \Gamma, g) $ satisfying the SCENARIO A. A positive smooth function $ Q : \mathbb{S}^{n} \slash \Gamma \rightarrow \R $satisfies the {\bf{CONDITION B}} if for any pairs of points $ \rho, \rho' $ that satisfies Definition \ref{intro:def1},

(i) either at the pair of points $ \rho $ and $ \rho' $ we have
\begin{equation}\label{intro:eqn7}
\begin{split}
& \nabla^{M} Q(\rho) \neq 0 \Rightarrow \nabla^{M} Q(\rho') \neq 0 \\
\Rightarrow & Q(\rho) = Q(\rho'), \nabla_{\xi_{i}}^{\R^{n+1}} Q'' \bigg|_{\rho} = \nabla_{\xi_{i}}^{\R^{n+1}} Q'' \bigg|_{\rho'}, i = 1, \dotso, n, \nabla_{\tau}^{\R^{n+1}} Q'' \bigg|_{\rho} = - \nabla_{\tau}^{\R^{n+1}} Q'' \bigg|_{\rho'},
\end{split}
\end{equation}
where $ \tau $ is the unit vector in $ \R^{n+1} $ from $ \rho $ to $ \rho' $, and $ \lbrace \xi_{1}, \dotso, \xi_{n}, \tau \rbrace $ is a basis of $ \R^{n+1} $;

(ii) or at antipodal points $ \rho $ and $ \rho' $ we have
\begin{equation*}
\nabla^{M} Q(\rho) = \nabla^{M} Q(\rho') = 0.
\end{equation*}

(iii) or if $ \nabla Q \equiv 0 $ on $ \mathbb{S}^{n} \slash \Gamma $, then $ Q $ is a positive constant.
\end{definition}
We suspect that any function that violates The CONDITION B in Definition \ref{intro:def2} may not be a good candidate for prescribed scalar curvature with respect to some Yamabe metric. However, we can only state our result in a positive way. Denote $ \eta_{1} $ to be the first eigenvalue of the conformal Laplacian on closed manifolds, and $ \eta_{1}' $ to be the first eigenvalue of conformal Laplacian with Robin boundary condition in (\ref{intro:eqn2}) on compact manifolds with non-empty boundary. The next main result for manifolds of the form $ \mathbb{S}^{n} \slash \Gamma $ are stated as follows.
\begin{theorem}\label{intro:thm3}
Let $ (M, g) = (\mathbb{S}^{n} \slash \Gamma, g) $ be a connected, closed Riemannian manifold with $ n \geqslant 3 $ satisfying the SCENARIO A. Assume $ \eta_{1} > 0 $. If $ S \in \calC^{\infty}(M) $ satisfies the CONDITION B, then the function $ S $ can be realized as a prescribed scalar curvature function for some Yamabe metric $ \tilde{g} \in [g] $.
\end{theorem}
\medskip

For necessary conditions of Kazdan-Warner problem, we observe that there is no obstruction of $ S $ in (\ref{intro:eqn4}) on the rest of the manifold $ M \backslash \bar{O} $ or $ \bar{M} \backslash \bar{O} $, respectively, provided that $ M $ or $ \bar{M} $ is not $ \mathbb{S}^{n} $ or some quotient of $ \mathbb{S}^{n} $. The set $ \bar{O} $, on which $ S $ is a constant function, can be arbitrarily small but with positive Lebesgue measure. It follows that all function can ``almost" be realized as a prescribed scalar curvature of some metric under conformal change in the following sense: given any function $ S_{0} \in \calC^{\infty}(M) $ where $ M $ is not some quotient of $ \mathbb{S}^{n} $, we can choose arbitrarily small set $ O \subset M $ and define
\begin{equation*}
S_{0}' = \begin{cases} c, & {\rm in} \; \bar{O} \\ S_{0}, & {\rm in} \; M \backslash \bar{O} \end{cases}
\end{equation*}
with some constant $ c > 0 $. The function
\begin{equation*}
S = S_{0}' * \phi_{\epsilon} \in \calC^{\infty}(M)
\end{equation*}
with the standard mollifier $ \phi_{\epsilon} $ with small enough $ \epsilon $ satisfies the property above. It follows that any smooth function $ S_{0} \in \calC^{\infty}(M) $ agrees pointwise with a scalar curvature on $ M $ under conformal change, except within an arbitrary small subset of $ M $. It is well-known that when $ \lambda(M) $ or $ \bar{\lambda}(\bar{M}) $ is positive, the scalar curvature should be positive some where with respect to the metrics in the conformal class, possibly with minimal boundary if $ \partial M \neq \emptyset $. We define
\begin{equation}\label{intro:eqn6}
\begin{split}
\mathcal{B} & : = \lbrace f \in \calC^{\infty}(M) : f > 0 \; \text{somewhere in M}, \\
& -a\Delta_{g} u + R_{g} u = f u^{p-1} \; \text{admits a real, positive, smooth solution $ u $ with $ \eta_{1} > 0 $}. \rbrace \\
\mathcal{B}' & : = \lbrace f \in \calC^{\infty}(\bar{M}) : f > 0 \; \text{somewhere in the interior M}, \\
& -a\Delta_{g} u + R_{g} u = f u^{p-1} \; {\rm in} \; M, \frac{\partial u}{\partial \nu} + \frac{2}{p-2} h_{g} u = 0 \; {\rm on} \; \partial M, \\
& \text{admits a real, positive, smooth solution $ u $ with $ \eta_{1}' > 0 $}. \rbrace
\end{split}
\end{equation}
We can then state the following ``almost" necessary conditions of Kazdan-Warner problem.
\begin{theorem}\label{intro:thm4}
Let $ (M, g) $ be a connected, closed manifold, $ n =  \dim M \geqslant 3 $. Assume that $ \lambda(M) > 0 $. If $ M $ is not $ \mathbb{S}^{n} \slash \Gamma $ for some Kleinian group $ \Gamma $, then $ \mathcal{A} \subset \mathcal{B} $; in addition, $ \mathcal{A} $ is $ \calC^{0} $-dense in $ \mathcal{B} $.
\end{theorem}
\begin{theorem}\label{intro:thm5}
Let $ (\bar{M}, g) $ be a connected, compact manifold with non-empty smooth boundary, $ n =  \dim \bar{M} \geqslant 3 $. Assume that $ \bar{\lambda}({\bar{M}}) > 0 $. Then $ \mathcal{A}' \subset \mathcal{B}' $; in addition, $ \mathcal{A}' $ is $ \calC^{0} $-dense in $ \mathcal{B}' $.
\end{theorem}
\medskip

As an overview of our iteration scheme, as in \cite{XU2, XU4, XU5, XU3}, we apply local analysis, perturbation methods and the monotone iteration scheme, instead of the classical global variational approach, to construct the solutions of (\ref{intro:eqn1}) and (\ref{intro:eqn2}). The details of comparison between global variational methods used before and the local analysis applied here could be found in \cite{XU4, XU5, XU3}.  We point out that the choice of test functions are critical in both the classical methods and the existence of the local solution of Yamabe equation with Dirichlet boundary condition here. The key is a local variational method in order to seek for a positive solution of a local Yamabe equation $ \Box_{g} u = Q u^{p-1} $ in $ O $ and $ u \equiv 0 $ on $ \partial O $. When the manifold is not locally conformally flat, we apply the result of Wang \cite{WANG} as well as a perturbation of the Yamabe equation to construct the local solution when $ \dim O \geqslant 3 $. The perturbation is introduced to bypass the critical dimension 6 in Aubin's method. When the manifold is locally conformally flat, we apply the result of Bahri and Coron \cite{BC}, and the result of Clapp, Faya and Pistoia \cite{CFP} to construct a local solution in a topologically nontrivial domain $ O $. We extend the local solution by zero as our sub-solution of the Yamabe-type equation. Once we construct a super-solution, we can apply the monotone iteration scheme. We would like to point out that, as a corollary, we can apply the monotone iteration scheme to obtain some positive, smooth solution of the perturbed Yamabe equation $ -\Box_{g} u + \beta u = S u^{p-1} $, possibly with the boundary condition $ B_{g} u = 0 $. Here $ \beta < 0 $ with small enough absolute value. 

In classical methods, the role of the Weyl tensor is crucial. For example, Aubin's conformal normal coordinates construction and the test function estimates relies heavily on the local nonvanishing of the Weyl tensor; Escobar's analysis depends on the vanishing/nonvanishing of the Weyl tensors, both in the interior and on the boundary. In this article, the existence of the local solution is the only place we use vanishing of Weyl tensor to classify manifolds. In other words, the essential difficulty is not coming from the vanishing of Weyl tensor due to our methodology, since we can use two different methods, depending on the locally conformal flatness of the manifold, to get local solutions of the Yamabe equation with Dirichlet condition, although with different domains. It is equivalent to say that we can always construct a local solution of the Yamabe equation on all manifolds. In contrast, the essential difficulty is due to the positivity of the first eigenvalue of the conformal Laplacian, which forces the scalar curvature function of the Yamabe metric to be positive at each point, and hence indicates the comparison between the size of the geodesic ball and the Euclidean ball of the same radius. The Yamabe quotient involves both the scalar curvature $ R_{g} $ (the mean curvature $ h_{g} $ if the boundary is non-empty), and the volume form $ \dvol $. It follows that the sign of $ \eta_{1} $ measures the change of the Yamabe quotient. Essentially the positive first eigenvalue increases the energy of the Yamabe quotient and thus leave us less room to find a good minimizer. We introduce our methodology here to resolve this issue. The Weyl tensor cannot indicate the evaluation of the scalar curvature, neither the size of the geodesic ball.
\medskip

This paper is organized as follows: In \S2, we introduce essential definitions and results we will use in later contexts. We then show the existence of the solutions of the perturbed Yamabe equation locally in Proposition \ref{local:prop2} on some small enough open subset of any compact manifolds with dimensions at least 3. Applying a modification of the test function given in Proposition \ref{local:prop3}, we prove the existence of the solution of the local Yamabe equation in Proposition \ref{local:prop4} when the manifold is not locally conformally flat. When the manifold is locally conformally flat, we apply the result of Theorem \ref{local:thm4} to obtain the local solution in some open subset of the manifold.

In \S3, we first derive the sufficient conditions for the prescribed scalar curvature problem when $ \eta_{1} > 0 $, for $ M $ closed and not locally conformally flat. In Theorem \ref{closed:thm2}, we restrict the scalar curvature $ R_{g} $ to be negative within a certain region. We construct a sub-solution and a super-solution of the Yamabe equation (\ref{intro:eqn1}); then a monotone iteration scheme in Theorem \ref{closed:thm1} is applied. We are able to remove this restriction to obtain a general result in Corollary \ref{closed:cor1}, with the aid of Proposition \ref{closed:thm3}. We then consider the ``almost" necessary conditions of Kazdan-Warner problem on manifolds that are not locally conformally flat. The results are listed in Theorem \ref{closed:thm4} and Corollary \ref{closed:cor3}, they are $ \calL^{r} $-dense result for any $ r \in [1, \infty) $.

In \S4, we first consider the sufficient conditions for the Kazdan-Warner problem on compact manifolds $ (\bar{M}, g) $ with non-empty smooth boundary such that the interior $ M $ is not locally conformally flat. Following the method above, we construct a sub-solution and a super-solution of (\ref{intro:eqn2}), then apply a monotone iteration scheme in Theorem \ref{compact:thm1} to obtain a result of prescribed scalar curvature with minimal boundary in Theorem \ref{compact:thm2}, in which we require $ R_{g} $ to be negative in some open region and the mean curvature $ h_{g} > 0 $ everywhere on $ \partial M $. Corollary \ref{compact:cor1} removes the restriction of the sign of $ R_{g} $, due to Proposition \ref{compact:prop2}, although the positivity of mean curvature is still needed. Finally, a general result with no restrictions of signs of both $ R_{g} $ and $ h_{g} $ is given in Corollary \ref{compact:cor2} by applying Proposition \ref{compact:prop1} and Proposition \ref{compact:prop2}. The ``almost" necessary conditions are stated in Theorem \ref{compact:thm3} and Corollary \ref{compact:cor3} in the second part of \S4. Analogous to the results in \S3, they are $ \calL^{r} $-dense result for any $ r \in [1, \infty) $.

In \S5, we discuss the prescribed scalar curvature functions on $ (\mathbb{S}^{n}, g_{\mathbb{S}^{n}}) $, $ n \geqslant 3 $, as a model of locally conformally flat manifold with many symmetries. The study of the local solution of the Yamabe equation on some open subsets of $ \mathbb{S}^{n} $ reveals an analytic obstruction to the prescribed scalar curvature problem, denoted by the CONDITION A in Definition \ref{sphere:def1}. We showed in Theorem \ref{sphere:thm1} that any positive function that satisfies the CONDITION A can be realized as a prescribed scalar curvature function for some metric $ \tilde{g} \in [g_{\mathbb{S}^{n}}] $. Since our method relies on the construction of a local solution and monotone iteration scheme, the obstruction says that if a positive function does not satisfy the CONDITION A, then the local solution of the Yamabe equation with respect to Dirichlet condition may not exist on open subsets of $ \mathbb{S}^{n} $. We then show that a function satisfying the CONDITION A does not violate the obstructions of Kazdan-Warner \cite{KW2}, of Bourguignon and Ezin \cite{BE}, and of Malchiodi and Mayer \cite{MaMa}. Inspired by the very special geometry on $ \mathbb{S}^{n} $, we define the SCENARIO A on general manifolds in Definition \ref{sphere:def2}. We show later in \S6 that the functions satisfying the CONDITION B, which is defined in Definition \ref{sphere:def3}, on the manifolds with the SCENARIO A can be realized as prescribed scalar curvature functions for Yamabe metrics.

In \S6, we show our main theorems for prescribed scalar curvature functions of Yamabe metrics in Theorem \ref{closed1:thm2} and Theorem \ref{closed1:thm4}, for all closed manifolds and compact manifolds with smooth boundary, respectively, with dimensions at least 3. These results improve the results of sufficient conditions for the Kazdan-Warner problem in \S3 and \S4. We then improve the $ \calL^{r} $-density results in \S3 and \S4 for manifolds that do not have the SCENARIO A to $ \calC^{0} $-density in Theorem \ref{closed1:thm3} and Theorem \ref{closed1:thm5}.
\medskip

{\bf{ACKNOWLEDGEMENT}}: The author would like to thank Prof. Steve Rosenberg and Prof. Gang Tian for their great mentorships in this topic. The author would also like to thank Martin Mayer for many valuable discussions with respect to the Kazdan-Warner problem on $ \mathbb{S}^{n} $.

\section{The Local Analysis}
We begin with a few set-up. Let $ \Omega $ be a connected, bounded, open subset of $ \R^{n} $ with smooth boundary $ \partial \Omega $ equipped with some Riemannian metric $ g $ that can be extended smoothly to $ \bar{\Omega} $. We call $ (\Omega, g) $ a Riemannian domain. Furthermore, let $ (\bar{\Omega}, g) $ be a compact manifold with smooth boundary extended from $ (\Omega, g) $. Throughout this article, we denote $ (M, g) $ to be a closed manifold with $ \dim M \geqslant 3 $, and $ (\bar{M}, g) $ to be a general compact manifold with interior $ M $ and smooth boundary $ \partial M $; we denote the space of smooth functions with compact support by $ \calC_{c}^{\infty} $, smooth functions by $ \calC^{\infty} $, and continuous functions by $ \calC^{0} $.

In this section, we consider the following PDE
\begin{equation}\label{local:eqn1}
-a\Delta_{g} u + R_{g} u = \lambda u^{p-1} \; {\rm in} \; \Omega, u \equiv 0 \; {\rm on} \; \partial \Omega
\end{equation}
with $ \lambda > 0 $ and the scalar curvature $ R_{g} < 0 $ everywhere on $ \Omega $. We call this PDE the local Yamabe equation. Recall that
\begin{equation*}
a = \frac{4(n - 1)}{n - 2}, p = \frac{2n}{n - 2}.
\end{equation*}
We show that for any choice of $ \lambda > 0 $, (\ref{local:eqn1}) admits a smooth solution $ u \in \calC^{\infty}(\Omega) \cap \calC^{0}(\bar{\Omega}) $ such that $ u > 0 $ in a small enough Riemannian domain $ (\Omega, g) $ when $ \Omega $ is not locally conformally flat. Note that the assumption $ R_{g} < 0 $ everywhere is crucial. When the manifold is locally conformally flat, we refer to Theorem \ref{local:thm4} below.

Foremost, we define Sobolev spaces on $ (M, g) $, $ (\bar{M}, g) $ and $ (\Omega, g) $, both in global expressions and local coordinates.
\begin{definition}\label{local:def1} Let $ (\Omega, g) $ be a Riemannian domain. Let $ (M, g) $ be a closed  Riemannian $n$-manifold with volume density $\dvol$. Let $u$ be a real valued function. Let $ \langle v,w \rangle_g$ and $ |v|_g = \langle v,v \rangle_g^{1/2} $ denote the inner product and norm  with respect to $g$. 

(i) 
For $1 \leqslant q < \infty $,
\begin{align*}
\mathcal{L}^{q}(\Omega)\ &{\rm is\ the\ completion\ of}\  \left\{ u \in \calC_c^{\infty}(\Omega) : \Vert u\Vert_{\calL^{q}(\Omega)}^q :=\int_{\Omega} \lvert u \rvert^{q} dx < \infty \right\},\\
\mathcal{L}^{q}(\Omega, g)\ &{\rm is\ the\ completion\ of}\ \left\{ u \in \calC_c^{\infty}(\Omega) : \Vert u\Vert_{\calL^{q}(\Omega, g)}^q :=\int_{\Omega} \left\lvert u \right\rvert^{q} d\text{Vol}_{g} < \infty \right\}, \\
\mathcal{L}^{q}(M, g)\ &{\rm is\ the\ completion\ of}\ \left\{ u \in \calC^{\infty}(M) : \Vert u\Vert_{\calL^{q}(M, g)}^q :=\int_{M} \left\lvert u \right\rvert^{q} d\text{Vol}_{g} < \infty \right\}.
\end{align*}

(ii) For $\nabla u$  the Levi-Civita connection of $g$, 
and for $ u \in \calC^{\infty}(\Omega) $ or $ u \in \calC^{\infty}(M) $,
\begin{equation*}
\lvert \nabla^{k} u \rvert_g^{2} := (\nabla^{\alpha_{1}} \dotso \nabla^{\alpha_{k}}u)( \nabla_{\alpha_{1}} \dotso \nabla_{\alpha_{k}} u).
\end{equation*}
\noindent In particular, $ \lvert \nabla^{0} u \rvert^{2}_g = \lvert u \rvert^{2} $ and $ \lvert \nabla^{1} u \rvert^{2}_g = \lvert \nabla u \rvert_{g}^{2}.$\\

(iii) For $ s \in \mathbb{N}, 1 \leqslant p < \infty $,
\begin{align*}
W^{s, q}(\Omega) &= \left\{ u \in \mathcal{L}^{q}(\Omega) : \lVert u \rVert_{W^{s,q}(\Omega)}^{q} : = \int_{\Omega} \sum_{j=0}^{s} \left\lvert D^{j}u \right\rvert^{q} dx < \infty \right\}, \\
W^{s, q}(\Omega, g) &= \left\{ u \in \mathcal{L}^{q}(\Omega, g) : \lVert u \rVert_{W^{s, q}(\Omega, g)}^{q} = \sum_{j=0}^{s} \int_{\Omega} \left\lvert \nabla^{j} u \right\rvert^{q}_g \dvol < \infty \right\}, \\
W^{s, q}(M, g) &= \left\{ u \in \mathcal{L}^{q}(M, g) : \lVert u \rVert_{W^{s, q}(M, g)}^{q} = \sum_{j=0}^{s} \int_{M} \left\lvert \nabla^{j} u \right\rvert^{q}_g \dvol < \infty \right\}.
\end{align*}
\noindent Here $ \lvert D^{j}u \rvert^{q} := \sum_{\lvert \alpha \rvert = j} \lvert \partial^{\alpha} u \rvert^{q} $ in the weak sense. Similarly, $ W_{0}^{s, q}(\Omega) $ is the completion of $ \calC_{c}^{\infty}(\Omega) $ with respect to the $ W^{s, q} $-norm. 

In particular, $ H^{s}(\Omega) : = W^{s, 2}(\Omega) $ and $ H^{s}(\Omega, g) : = W^{s, 2}(\Omega, g) $, $ H^{s}(M, g) : = W^{s, 2}(M, g) $ are the usual Sobolev spaces. We similarly define $H_{0}^{s}(\Omega), H_{0}^{s}(\Omega,g)$.

(iv) We define the $ W^{s, q} $-type Sobolev space on $ (\bar{M}', g) $ the same as in (iii) when $ s \in \mathbb{N}, 1 \leqslant q < \infty $.
\end{definition}
\medskip

Brezis and Nirenberg considered the PDE
\begin{equation}\label{local:eqn2}
 -\Delta_{e} u - \gamma u = \lambda u^{p-1} \; {\rm in} \; \Omega, u \equiv 0 \; {\rm on} \; \partial \Omega
 \end{equation}
 with Euclidean Laplacian $ -\Delta_{e} $ on a bounded, open subset of $ \R^{n}, n \geqslant 3 $. They showed in \cite{Niren3} that (\ref{local:eqn2}) admits a positive, smooth solution for any $ \gamma > 0 $ when $ n \geqslant 4 $; when $ n = 3 $, the existence of positive solution of (\ref{local:eqn2}) requires $ \Omega $ to be a ball $B_{r}(0) $ on which $ \gamma > \frac{1}{4} \lambda_{1} $, where $ \lambda_{1} $ is the first eigenvalue of Euclidean Laplacian on ball of radius $ r $. This inspires us to consider the local solution of (\ref{local:eqn1}) when $ R_{g} < 0 $ everywhere on $ \Omega $. Locally we treat (\ref{local:eqn1}) as the general second order linear elliptic PDE with Dirichlet boundary condition:
 \begin{equation}\label{local:eqn3}
\begin{split}
Lu & : = -\sum_{i, j} \partial_{i} \left (a_{ij}(x) \partial_{j} u \right) = b(x) u^{p- 1} + f(x, u) \; {\rm in} \; \Omega; \\
u & > 0 \; {\rm in} \; \Omega, u = 0 \; {\rm on} \; \partial \Omega.
\end{split}
\end{equation}
Here $ p - 1 = \frac{n+2}{n -2} $ is the critical exponent with respect to the $ H_{0}^{1} $-solutions of (\ref{local:eqn3}) in the sense of Sobolev embedding. Due to variational method, (\ref{local:eqn3}) is the Euler-Lagrange equation of the functional
\begin{equation}\label{local:eqn4}
J(u) = \int_{\Omega} \left( \frac{1}{2} \sum_{i, j} a_{ij}(x) \partial_{i}u \partial_{j} u - \frac{b(x)}{p} u_{+}^{p} - F(x, u) \right) dx,
\end{equation}
with appropriate choices of $ a_{ij}, b $ and $ F $. Here $ u_{+} = \max \lbrace u, 0 \rbrace $ and $ F(x, u) = \int_{0}^{u} f(x, t)dt $. Set
\begin{equation}\label{local:eqn5}
\begin{split}
A(O) & = \text{essinf}_{x \in O} \frac{\det(a_{ij}(x))}{\lvert b(x) \rvert^{n-2}}, \forall O \subset \Omega; \\
T & = \inf_{u \in H_{0}^{1}(\Omega)}  \frac{\int_{\Omega} \lvert Du \rvert^{2} dx}{\left( \int_{\Omega} \lvert u \rvert^{p} dx \right)^{\frac{2}{p}}}; \\
K & = \inf_{u \neq 0} \sup_{t > 0} J(tu), K_{0} = \frac{1}{n} T^{\frac{n}{2}} \left( A(\Omega) \right)^{\frac{1}{2}}.
\end{split}
\end{equation}
Wang's result \cite{WANG} stated the existence of a positive smooth solution of (\ref{local:eqn3}) with appropriate choices of $ a_{ij}, b $ and $ F $, which is a powerful extension of the results of Brezis and Nirenberg.
\begin{theorem}\label{local:thm1}\cite[Thm.~1.1, Thm.~1.4]{WANG} Let $ \Omega $ be a bounded smooth domain in $ \R^{n}, n \geqslant 3 $. Let $ Lu = -\sum_{i, j} \partial_{i} \left (a_{ij}(x) \partial_{j} u \right) $ be a second order elliptic operator with smooth coefficients in divergence form. Let ${\rm Vol}_g(\Omega)$ and the diameter of $\Omega$ sufficiently small. Let $ b(x) \neq 0 $ be a nonnegative bounded measurable function. Let $ f(x, u) $ be measurable in $ x $ and continuous in $ u $. Assume
\begin{enumerate}[(P1)]
\item There exist $ c_{1}, c_{2} > 0 $ such that $ c_{1} \lvert \xi \rvert^{2} \leqslant \sum_{i, j} a_{ij}(x) \xi_{i} \xi_{j} \leqslant c_{2} \lvert \xi \rvert^{2}, \forall x \in \Omega, \xi \in \R^{n} $;
\item $ \lim_{u \rightarrow + \infty} \frac{f(x, u)}{u^{p-1}} = 0 $ uniformly for $ x \in \Omega $;
\item $ \lim_{u \rightarrow 0} \frac{f(x, u)}{u} < \lambda_{1} $ uniformly for $ x \in \Omega $, where $ \lambda_{1} $ is the first eigenvalue of $ L $;
\item There exists $ \theta \in (0, \frac{1}{2}), M \geqslant 0, \sigma > 0 $, such that $ F(x, u) = \int_{0}^{u} f(x, t)dt \leqslant \theta u f(x, u) $ for any $ u \geqslant M $, $ x \in \Omega(\sigma) = \lbrace x \in \Omega, 0 \leqslant b(x) \leqslant \sigma \rbrace $.
\end{enumerate}
Furthermore, we assume that $ f(x, u) \geqslant 0 $, $ f(x, u) = 0 $ for $ u \leqslant 0 $. We also assume that $ a_{ij}(x) \in \calC^{0}(\bar{\Omega}) $. If
\begin{equation}\label{local:eqn6}
K < K_{0}
\end{equation}
then the Dirichlet problem (\ref{local:eqn3}) possesses a positive solution $ u \in \calC^{\infty}(\Omega) \cap \calC^{0}(\bar{\Omega}) $ which satisfies $ J(u) \leqslant K $.
\end{theorem}
\medskip

This is a slight modification of the original result proved by Wang. The reason we add the smallness of the Riemannian domain is to make sure that the local Yamabe equation has a local expression. In the proof of Theorem \ref{local:thm1}, Wang applied a mountain pass theorem due to Ambrosetti and Rabinowitz \cite[Thm.~2.2]{Niren3}, which dealt with the situation when Palais-Smale condition fails. We list the result below for later analysis.
\begin{theorem}\label{local:thm2}\cite[Thm.~2.2]{Niren3}
Let $ \Phi $ be a $ \calC^{1} $ function on a Banach space $ E $. Suppose there exists a neighborhood $ U $ of $ 0 $ in $ E $ and a constant $ \rho $ such that $ \Phi(u) \geqslant \rho $ for every $ u $ in the boundary of $ U $, 
\begin{equation*}
\Phi(0) < \rho, \Phi(v) < \rho \; \text{for some} \; v \notin U.
\end{equation*}
Set
\begin{equation}\label{local:eqnm1}
T_{0} = \inf_{P \in \mathcal{P}} \max_{w \in P} \Phi(w) \geqslant \rho,
\end{equation}
where $ \mathcal{P} $ denotes the class of continuous paths joining $ 0 $ to $ v $. Then there is a sequence $ \lbrace u_{j} \rbrace $ in $ E $ such that
\begin{equation*}
\Phi(u_{j}) \rightarrow T_{0}, \Phi'(u_{j}) \rightarrow 0, j \rightarrow \infty
\end{equation*}
in $ E^{*} $.
\end{theorem}
To apply Theorem \ref{local:thm1}, we need to check conditions (P1) through (P4). More importantly, we need to verify the key inequality (\ref{local:eqn6}). For the solvability of (\ref{local:eqn1}), the condition $ K < K_{0} $ is exactly equivalent to find out a local test function $ \phi \in \calC_{c}^{\infty}(\Omega) $ such that $ Q(\phi) < \lambda(\mathbb{S}^{n}) $, see \cite{XU3}. Unfortunately, Aubin \cite{PL} shows that no test function works when the dimension is too low or the manifold is locally conformally flat. This forces us to consider the perturbed Yamabe problem locally on $ \Omega $ by perturbing the coefficient of the zeroth order term of the differential operator:
\begin{equation}\label{local:eqn7}
-a\Delta_{g} u + \left( R_{g} + \beta \right) u = \lambda u^{p-1} \; {\rm in} \; \Omega, u \equiv 0 \; {\rm on} \; \partial \Omega.
\end{equation}
Here we also require $ R_{g} < 0 $ everywhere on $ \Omega $. In addition, $ \beta < 0 $ can be any constant. The negative constant $ \beta $ allows us to verify $ K < K_{0} $ with respect to (\ref{local:eqn7}), regardless of the dimension and locally conformal flatness of the manifold. The following technical result, which is related to Yamabe invariant, is necessary to handle the existence of the solution of (\ref{local:eqn7}).
\begin{proposition}\label{local:prop1}\cite[Thm.~A.1]{XU3}
Let $ \Omega = B_{0}(r) $ for small enough $ r $ with $ \dim \Omega \geqslant 3 $. Let $ g $ be some Riemannian metric with local expression $ g = g_{ij} dx^{i} \otimes dx^{j} $. Let $ \beta < 0 $ be any negative constant and $ T $ be the best Sobolev constant
\begin{equation}\label{local:eqn8}
T = \inf_{u \in H_{0}^{1}(\Omega)}  \frac{\int_{\Omega} \lvert Du \rvert^{2} dx}{\left( \int_{\Omega} \lvert u \rvert^{p} dx \right)^{\frac{2}{p}}} = \inf_{u \in H_{0}^{1}(\R^{n})}  \frac{\int_{\R^{n}} \lvert Du \rvert^{2} dx}{\left( \int_{\R^{n}} \lvert u \rvert^{p} dx \right)^{\frac{2}{p}}}.
\end{equation}
The quantity $ Q_{\epsilon, \Omega} $ satisfies
\begin{equation}\label{local:eqn9}
Q_{\epsilon, \Omega} : =\frac{\lVert \nabla_{g} u_{\epsilon, \Omega, \beta} \rVert_{\calL^{2}(\Omega, g)}^{2} + \frac{1}{a} \int_{\Omega} \left(R_{g} + \beta \right) u_{\epsilon, \Omega, \beta}^{2} \sqrt{\det(g)} dx}{\lVert u_{\epsilon, \Omega, \beta} \rVert_{\calL^{p}(\Omega, g)}^{2}} < T.
\end{equation}
with the test function
\begin{equation}\label{local:eqn10}
u_{\epsilon, \Omega, \beta}(x) = \frac{\varphi_{\Omega}(x)}{\left(\epsilon + \lvert x \rvert^{2}\right)^{\frac{n - 2}{2}}}, n \geqslant 3.
\end{equation}
Here the smooth cut-off function $ \varphi_{\Omega}(x) $ is chosen to be any radial bump function supported in $ \Omega $ such that $  \varphi_{\Omega}(x) \equiv 1 $ on $ B_{0} \left( \frac{r}{2} \right) $ when $ n \geqslant 4 $; when $ n = 3 $, we set $ \varphi_{\Omega}(x) =  \cos \left(\frac{\pi \lvert x \rvert}{2} \right) $ with $ \Omega = B_{0}(1) $, the unit ball.
\end{proposition}
\begin{remark}\label{local:re1}
Note that (\ref{local:eqn9}) is no longer a conformal invariant. Proposition \ref{local:prop1}, which is proven in \cite{XU3}, is uniform for all cases when $ n \geqslant 3 $. The Weyl tensor and locally conformal flatness is not required. Only standard normal coordinates with respect to the Riemannian metric $ g $ is used.
\end{remark}
\medskip

We now apply Theorem \ref{local:thm1} and Proposition \ref{local:prop1} to show the existence of a positive smooth solution of the perturbed Yamabe equation (\ref{local:eqn7}) with Dirichlet boundary condition. We refer \cite[Prop.~3.3]{XU3} for a complete proof of the following result. We sketch the proof below to show essential steps which are particularly useful in later context.
\begin{proposition}\label{local:prop2}\cite[Prop.~3.3]{XU3}
Let $ (\Omega, g) $ be Riemannian domain in $\R^n$, $ n \geqslant 3 $, with $C^{\infty} $ boundary, and with ${\rm Vol}_g(\Omega)$ and the Euclidean diameter of $\Omega$ sufficiently small. Let $ \beta < 0 $ be any constant. Assume $ R_{g} < 0 $ within the small enough closed domain $ \bar{\Omega} $. Then for any $ \lambda > 0 $, the Dirichlet problem (\ref{local:eqn7}) has a real, positive, smooth solution $ u \in \calC^{\infty}(\Omega) \cap H_{0}^{1}(\Omega, g) \cap \calC^{0}(\bar{\Omega}) $.
\end{proposition}
\begin{proof} (Sketch) In terms of $ W^{s, q}(\Omega, g) $-norms, (\ref{local:eqn7}) with positive solution is the Euler-Lagrange equation of the functional
\begin{equation}\label{local:eqn11}
J_{\beta}^{*}(u) : = \int_{\Omega} \left( \frac{1}{2} a\sqrt{\det(g)} g^{ij} \partial_{i}u \partial_{j} u - \frac{\sqrt{\det(g)}}{p} \lambda u_{+}^{p} - \int_{0}^{u} \sqrt{\det(g)(x)}(-R_{g}(x) - \beta) t dt \right) dx.
\end{equation}
Note that $ J_{\beta}^{*} u $ is just a special expression of (\ref{local:eqn4}) with
\begin{equation}\label{local:eqn12}
a_{ij}(x) = a\sqrt{\det(g)} g^{ij}(x), b(x) = \lambda \sqrt{\det(g)(x)}, f(x, u) = \sqrt{\det(g)(x)} \left(-R_{g}(x) - \beta \right)u.
\end{equation}
It is checked in \cite[Prop.~3.3]{XU3} that all hypotheses (P1) through (P4) are satisfied. Checking $ K < K_{0} $ in (\ref{local:eqn6}) can be reduced to find a positive function $ u > 0 $ in $ \Omega $, $ u \equiv 0 $ on $ \partial \Omega $ such that $ sup_{t > 0} J_{\beta}^{*}(tu) < K_{0} $. Recall from (\ref{local:eqn5}) that
\begin{equation*}
K_{0} = \frac{1}{n} T^{\frac{n}{2}} \left( A(\Omega) \right)^{\frac{1}{2}}, A(\Omega) = \text{essinf}_{x \in \Omega} \frac{\det(a_{ij}(x))}{\lvert b(x) \rvert^{n-2}}.
\end{equation*}
In the context of (\ref{local:eqn12}), we have
\begin{equation*}
A(\Omega) =  \lambda^{2 - n}a^{n}.
\end{equation*}
It follows that $ K_{0} $ is of the form
\begin{equation}\label{local:eqn13}
K_{0} = K_{0} = \frac{1}{n} T^{\frac{n}{2}} A(\Omega)^{\frac{1}{2}} = \frac{1}{n} \lambda^{\frac{2- n}{2}}a^{\frac{n}{2}} T^{\frac{n}{2}}.
\end{equation}
The inequality $ \sup_{t > 0} J_{\beta}^{*}(tu) < K_{0} $ for a positive function in $ \Omega $ with Dirichlet condition on $ \partial \Omega $ is equivalent to
\begin{equation}\label{local:eqn14}
\begin{split}
& J_{0} : = \frac{\int_{\Omega} a\sqrt{\det(g)} g^{ij} \partial_{i} u \partial_{j} u dx- \int_{\Omega} \sqrt{\det(g)} \left(-R_{g} - \beta \right) u^{2} dx}{\left( \int_{\Omega} \sqrt{\det(g)} \lambda u^{p} dx \right)^{\frac{2}{p}}} < \lambda^{\frac{2-n}{n}} aT \\
\Leftrightarrow &  \frac{\int_{\Omega} \sqrt{\det(g)} g^{ij} \partial_{i} u \partial_{j} u dx- \frac{1}{a} \int_{\Omega}  \sqrt{\det(g)} \left(-R_{g} - \beta \right) u^{2} dx}{\left( \int_{\Omega} \sqrt{\det(g)} u^{p} dx \right)^{\frac{2}{p}}} < T.
\end{split}
\end{equation}
Note that $ \frac{1}{n} J_{0}^{\frac{n}{2}} = \sup_{t > 0} J_{\beta}^{*}(tu) $. This is exactly the inequality that Proposition (\ref{local:prop1}) shows. It follows that the condition $ K < K_{0} $ is satisfied with the test function (\ref{local:eqn10}). It follows that all hypotheses in Theorem \ref{local:thm1} are satisfied. Therefore we conclude that there exists a function $ u \in H_{0}^{1}(\Omega) $, $ u > 0 $ in $ \Omega $ that solves (\ref{local:eqn7}) in the weak sense. Due to the same argument in Brezis and Nirenberg \cite{Niren3}, we conclude that $ u \in \calC^{\infty}(\Omega) \cap \calC^{0}(\bar{\Omega}) $. The boundary regularity can be improved to at least $ \calC^{1, \alpha} $, here the choice of $ \alpha $ is based on the dimension of $ \Omega $.
\end{proof}
\begin{remark}\label{local:re2}
Recall the critical value $ T_{0} $ in mountain pass theorem. For the functional $ J_{\beta}^{*} $, we denote $ T_{0} $ as $ T_{0, \beta} $ since it clearly depends on $ \beta $. In \cite{WANG}, on of the key observations is the solution $ u $ of (\ref{local:eqn7}) satisfies
\begin{equation}\label{local:eqn14a}
J_{\beta}^{*}(u) \leqslant T_{0, \beta} < K_{0}.
\end{equation}
This inequality is the key step for the solvability of the local Yamabe equation (\ref{local:eqn1}).
\end{remark}
\medskip

In order to solve (\ref{local:eqn1}), it is natural to take a limiting argument by letting $ \beta \rightarrow 0^{-} $ in terms of the solutions $ \lbrace u_{\beta} \rbrace_{\beta} $ of (\ref{local:eqn7}) with $ \beta < 0 $. On open subsets of Euclidean space with Euclidean metric, the first eigenvalue goes to infinity when the size of the domain shrinks to a point. On Riemannian domain $ (\Omega, g) $, we need a detailed analysis with a result from Li and Yau \cite{LY}.
\begin{theorem}\label{local:thm3}\cite[Thm. 7]{LY} Let $ (\bar{M}, g) $ be a compact manifold with smooth boundary, let $ r_{inj} $ be the injectivity radius of $ M $, and let $ h_{g} $ be the minimum of the mean curvature of $ \partial M $. Choose $ K \geq 0$  such that  $ Ric_{g} \geqslant -(n - 1)K $. If $ \lambda_{1} $ is the first nonzero eigenvalue
\begin{equation*}
-\Delta_{g} u = \lambda_{1} u \; {\rm in} \; M', u \equiv 0 \; {\rm on} \; \partial M'.
\end{equation*}
Then
\begin{equation}\label{local:eqn15}
\lambda_{1} \geqslant \frac{1}{\gamma} \left( \frac{1}{4(n - 1) r_{inj}^{2}} \left( \log \gamma \right)^{2} - (n - 1)K \right),
\end{equation}
where
\begin{equation}\label{local:eqn16}
\gamma = \max \left\lbrace \exp[{1 + \left( 1 - 4(n - 1)^{2} r_{inj}^{2}K \right)^{\frac{1}{2}}}], \exp[{-2(n - 1)h_{g} r_{inj}} ] \right\rbrace.
\end{equation}
\end{theorem}

\begin{remark}\label{local:re3}
We apply Theorem \ref{local:thm3} above for Riemannian domain $ (\Omega, g) $, for which $ (\bar{\Omega}, g) $ is considered to be a compact manifold with boundary. In Riemannian normal coordinates centered at $ P \in \Omega$, $g$ agrees with the Euclidean metric up to terms of order $O(r^2)$. Thus if $\Omega$ is a $g$-geodesic ball of radius $r$, the mean curvature of $\partial\Omega$ is close $(n-1)/r$, the mean curvature of a Euclidean $r$-ball in $\R^n$.  In (\ref{local:eqn16}), as $r\to 0$,  $K$ can be taken to be unchanged (since $g$ is independent of $r$), $r_{inj}\to 0$, and $h\cdot r_{inj}\to n-1$. Thus $\gamma\to e^2$,  the right hand side of (\ref{local:eqn15}) goes to infinity as $r\to 0$, and $\lambda_1\to\infty. $ If $\Omega$ is a general, small enough Riemannian domain with a small enough injectivity radius, then $\Omega$ sits inside a $g$-geodesic ball $\Omega''$ of small radius.  By the Rayleigh quotient characterization of $\lambda_1$, we have $\lambda_{1,\Omega''} \leqslant \lambda_{1,\Omega} $.  Thus for all Riemannian domains $(\Omega,g)$, $ \lambda_{1}^{-1} \rightarrow 0 $ as the radius of $\Omega$ goes to zero.
\end{remark}

The functional $ J_{\beta}^{*}(u) $ depends on $ \beta $. For any $ \beta < 0 $, Proposition \ref{local:prop1} and \ref{local:prop2} indicate that there exists a $ u_{\epsilon, \Omega, \beta} $, given in (\ref{local:eqn10}), such that
\begin{equation*}
\sup_{t > 0} J_{\beta}^{*}(t u_{\epsilon, \Omega, \beta}) < K_{0}.
\end{equation*}
When $ t $ large enough, the functional $ J^{*}(tu) < 0 $ for any fixed $ u $, thus we choose some large enough $ t = t_{0} $ such that $ J^{*}(t_{0} u_{\epsilon, \Omega, \beta}) < 0 $. By the same argument in \cite[Thm.~2.1]{Niren3} and \cite[Thm.~1.1]{WANG}, we denote $ \mathcal{P} $ to be the collection of all continuous paths from $ 0 $ to $ t_{0} u_{\epsilon, \Omega, \beta} $ and the hypotheses in Theorem \ref{local:thm2} are satisfied. It follows immediately that
\begin{equation*}
T_{0, \beta} \leqslant \sup_{t > 0} J_{\beta}^{*}(t u_{\epsilon, \Omega, \beta}) < K_{0}.
\end{equation*}
Therefore, we conclude that $ J(u_{\beta}) \leqslant T_{0,\beta} < K_{0} $ where $ u_{\beta} $ is the positive solution of (\ref{local:eqn7}) due to Remark \ref{local:re2}. We hope that the critical value $ T_{0, \beta} $ is ideally to be the same for all $ \beta $. Here the Weyl tensor plays a role. The next result shows that when $ (\Omega, g) $ is not locally conformally flat, the inequality (\ref{local:eqn14a}) holds for some fixed $ T_{1} $, uniformly in $ \beta \in [\beta_{0}, 0) $, provided that $ \lvert \beta_{0} \rvert $ is small enough and the domain $ \Omega $ is small enough. We will discuss the locally conformally flat case later.

\begin{proposition}\label{local:prop3}
Let $ (\Omega, g) $ be a Riemannian domain in $\R^n$, $ n \geqslant 3 $, not locally conformally flat, with $C^{\infty} $ boundary, with ${\rm Vol}_g(\Omega)$ and the Euclidean diameter of $\Omega$ sufficiently small. In addition, we assume that the first eigenvalue of Laplace-Beltrami operator $ -\Delta_{g} $ on $ \Omega $ with Dirichlet condition satisfies $ \lambda_{1} \rightarrow \infty $ as $ \Omega $ shrinks. Let $ \lambda > 0 $ be a fixed constant. Assume $ R_{g} < 0 $ within the small enough closed domain $ \bar{\Omega} $. Fix $ \beta_{0} < 0 $ with small enough $ \lvert \beta_{0} \rvert $. Let
\begin{equation}\label{local:eqnt1}
\begin{split}
J_{\beta}^{*}(u) & = \int_{\Omega} \left( \frac{1}{2} a\sqrt{\det(g)} g^{ij} \partial_{i}u \partial_{j} u - \frac{\sqrt{\det(g)}}{p} \lambda u^{p} + \int_{0}^{u} \sqrt{\det(g)(x)}(R_{g}(x) + \beta) t dt \right) dx, \beta < 0; \\
T_{1} & = \sup_{t > 0} J_{\beta_{0}}^{*}(t u_{\epsilon, \Omega, \beta_{0}}).
\end{split}
\end{equation}
Here $ u_{\epsilon, \Omega, \beta_{0}} $ is defined in (\ref{local:eqn10}) with respect to $ J_{\beta_{0}}^{*}(u) $. Then for all $ \beta \in [\beta_{0}, 0) $, we have
\begin{equation}\label{local:eqnt2}
T_{0, \beta} = \inf_{P \in \mathcal{P}_{\beta}} \max_{w \in P} J_{\beta}^{*}(w) \leqslant T_{1} < K_{0}.
\end{equation}
Here $ \mathcal{P}_{\beta} $ is the collection of all continuous paths from $ 0 $ to $ t_{0} v_{\beta, \Omega} $ with respect to $ \beta $ such that $ J_{\beta}^{*}(t_{0} v_{\beta, \Omega}) < 0 $ for some positive function $ v_{\beta, \Omega} \in \calC_{c}^{\infty}(\Omega) $.
\end{proposition}
\begin{proof}
Without loss of generality, we may assume that $ \Omega $ is a geodesic ball of radius $ r $. The assumption $ R_{g} < 0 $ is valid since the manifold is not locally conformally flat, since otherwise we would have $ R_{g} \equiv 0 $ by the conformal invariance of the Yamabe quotient; in particular, the PDE (\ref{local:eqnt4}) below has only trivial solution. By parametrizing the paths in $ t $, we show this by showing that there exists a positive function $ v \in \calC_{c}^{\infty}(\Omega) $ such that
\begin{equation}\label{local:eqnt3}
\sup_{t >0} J_{\beta}^{*}\left(t \left(u_{\epsilon, \Omega, \beta} + v\right) \right) \leqslant \sup_{t > 0} J_{\beta_{0}}^{*}(t u_{\epsilon, \Omega, \beta}), \forall \beta \in [\beta_{0}, 0)
\end{equation}
where $ u_{\epsilon, \Omega, \beta} $ is the test function in (\ref{local:eqn10}) with respect to $ J_{\beta}^{*}(u) $. For simplicity and clarity, we interchangeably use $ u_{\epsilon, \Omega, \beta} = u_{\beta} $ just within this proof. Recall from Proposition \ref{local:prop1} above and \cite{XU3} that the inequality (\ref{local:eqn14a}) is achieved by choosing $ \epsilon $ small enough. Note that $ \epsilon $ is smaller when $ \lvert \beta \rvert $ is smaller. Thus the test function $ u_{\beta} $ is good for $ J_{\beta_{0}}^{*}(u) $. Due to the argument in Proposition \ref{local:prop1} and \cite[Prop.~3.3]{XU3}, (\ref{local:eqnt3}) is equivalent to
\begin{align*}
& \sup_{t >0} J_{\beta}^{*}\left(t \left(u_{\epsilon, \Omega, \beta} + v\right) \right) = \frac{1}{n} \left( \frac{a \int_{\Omega} \nabla_{g} (u_{\beta} + v) \cdot \nabla_{g}(u_{\beta} + v) \dvol + \int_{\Omega} \left( R_{g} + \beta \right) (u_{\beta} + v)^{2} \dvol}{\int_{\Omega} \lambda (u_{\beta} + v)^{p} \dvol} \right)^{\frac{n}{2}} \\
& \qquad \leqslant \frac{1}{n} \left( \frac{a \int_{\Omega} \nabla_{g} u_{\beta} \cdot \nabla_{g}u_{\beta} \dvol + \int_{\Omega} \left( R_{g} + \beta \right) u_{\beta}^{2} \dvol}{\int_{\Omega} \lambda u_{\beta}^{p} \dvol} \right)^{\frac{n}{2}} = \sup_{t > 0} J_{\beta_{0}}^{*}(t u_{\epsilon, \Omega, \beta}) \\
\Leftrightarrow & \frac{a \int_{\Omega} \nabla_{g} (u_{\beta} + v) \cdot \nabla_{g}(u_{\beta} + v) \dvol + \int_{\Omega} \left( R_{g} + \beta \right) (u_{\beta} + v)^{2} \dvol}{\int_{\Omega} (u_{\beta} + v)^{p} \dvol} \\
& \qquad \leqslant \frac{a \int_{\Omega} \nabla_{g} u_{\beta} \cdot \nabla_{g}u_{\beta} \dvol + \int_{\Omega} \left( R_{g} + \beta_{0} \right) u_{0}^{2} \dvol}{\int_{\Omega} u_{\beta}^{p} \dvol} \\
\Rightarrow & \Gamma_{1}: =  \frac{\int_{\Omega} (u_{\beta} + v)^{p} \dvol}{\int_{\Omega} u_{\beta}^{p} \dvol} \geqslant \frac{a \int_{\Omega} \nabla_{g} (u_{\beta} + v) \cdot \nabla_{g}(u_{\beta} + v) \dvol + \int_{\Omega} \left( R_{g} + \beta \right) (u_{\beta} + v)^{2} \dvol}{a \int_{\Omega} \nabla_{g} u_{\beta} \cdot \nabla_{g}u_{\beta} \dvol + \int_{\Omega} \left( R_{g} + \beta_{0} \right) u_{\beta}^{2} \dvol}: = \Gamma_{2}.
\end{align*}
Consider the PDE
\begin{equation}\label{local:eqnt4}
-a\Delta_{g} v = -2R_{g} u_{\beta} \; {\rm in} \; \Omega, v \equiv 0 \; {\rm on} \; \partial \Omega.
\end{equation}
Note that we require $ R_{g} < 0 $ in $ \Omega $. Since $ u_{\beta} > 0 $ in $ \Omega $, it follows that $ v > 0 $ and smooth in $ \Omega $, due to standard elliptic regularity and maximum principle. When we restrict $ \Omega $ to a smaller domain and let $ \beta \rightarrow 0 $, the test function $ u_{\beta}(x) $ becomes $ C_{\beta} u_{\beta}(\delta x) $ with some positive constant $ C_{\beta} $. The key is that the solution $ v $ of (\ref{local:eqnt4}) is invariant under scaling of $ \Omega $, or equivalently, the scaling of the metric $ g $. Therefore, $ \Gamma_{1} $ is invariant under scaling and therefore $ \Gamma_{1} > 1 $ is a fixed positive constant for all small enough geodesic balls $ \Omega $.

Due to Remark \ref{local:re3}, it follows that the first eigenvalue $ \lambda_{1} $ that satisfies $ -\Delta_{g} \phi = \lambda_{1} \phi $ in $ \Omega $, $ \phi = 0 $ on $ \partial \Omega $ increases as the geodesic ball $ \Omega $ shrinks. Since $ v $ solves (\ref{local:eqnt4}), it follows that
\begin{equation*}
\int_{\Omega} a \nabla_{g} v \cdot \nabla_{g} v \dvol = \int_{\Omega} -2R_{g} u_{\beta}v \dvol
\end{equation*}
by pairing $ v $ on both sides of (\ref{local:eqnt4}). We check $ \Gamma_{2} $ that
\begin{align*}
\Gamma_{2} & = \frac{a \int_{\Omega} \nabla_{g} (u_{\beta} + v) \cdot \nabla_{g}(u_{\beta} + v) \dvol + \int_{\Omega} \left( R_{g} + \beta \right) (u_{\beta} + v)^{2} \dvol}{a \int_{\Omega} \nabla_{g} u_{\beta} \cdot \nabla_{g}u_{\beta} \dvol + \int_{\Omega} \left( R_{g} + \beta_{0} \right) u_{\beta}^{2} \dvol} \\
& = 1 + \frac{2a \int_{\Omega} \nabla_{g} u_{\beta} \cdot \nabla_{g} v \dvol + a \int_{\Omega} \nabla_{g} v \cdot \nabla_{g} v \dvol -2\int_{\Omega} R_{g} u_{\beta}v \dvol}{a \int_{\Omega} \nabla_{g} u_{\beta} \cdot \nabla_{g}u_{\beta} \dvol + \int_{\Omega} \left( R_{g} + \beta_{0} \right) u_{\beta}^{2} \dvol} \\
& \qquad + \frac{\int_{\Omega} (R_{g} + \beta) v^{2} \dvol + \left( \beta - \beta_{0} \right) \int_{\Omega} u_{\beta}^{2} \dvol}{a \int_{\Omega} \nabla_{g} u_{\beta} \cdot \nabla_{g}u_{\beta} \dvol + \int_{\Omega} \left( R_{g} + \beta_{0} \right) u_{\beta}^{2} \dvol} \\
& = 1 + \frac{2 \int_{\Omega} u_{\beta} \left( -a\Delta_{g} v\right) \dvol + \int_{\Omega} (R_{g} + \beta) v^{2} \dvol + \left( \beta - \beta_{0} \right) \int_{\Omega} u_{\beta}^{2} \dvol}{a \int_{\Omega} \nabla_{g} u_{\beta} \cdot \nabla_{g}u_{\beta} \dvol + \int_{\Omega} \left( R_{g} + \beta_{0} \right) u_{\beta}^{2} \dvol} \\
\end{align*}
Using (\ref{local:eqnt4}) again, we have
\begin{align*}
\Gamma_{2} & = 1 + \frac{-4 \int_{\Omega} R_{g}u_{\beta}^{2} \dvol + \int_{\Omega} (R_{g} + \beta) v^{2} \dvol + \left( \beta - \beta_{0} \right) \int_{\Omega} u_{\beta}^{2} \dvol}{a \int_{\Omega} \nabla_{g} u_{\beta} \cdot \nabla_{g}u_{\beta} \dvol + \int_{\Omega} \left( R_{g} + \beta_{0} \right) u_{\beta}^{2} \dvol} : = 1 + \Gamma_{3}.
\end{align*}
Recall that $ R_{g} < 0 $ in $ \Omega $, and $ \beta < 0 $, the upper bound of $ \Gamma_{3} $ can be estimated as
\begin{align*}
\Gamma_{3} & = \frac{-4 \int_{\Omega} R_{g}u_{\beta}^{2} \dvol + \int_{\Omega} (R_{g} + \beta) v^{2} \dvol + \left( \beta - \beta_{0} \right) \int_{\Omega} u_{\beta}^{2} \dvol}{a \int_{\Omega} \nabla_{g} u_{\beta} \cdot \nabla_{g}u_{\beta} \dvol + \int_{\Omega} \left( R_{g} + \beta_{0} \right) u_{\beta}^{2} \dvol} \\
& \leqslant \frac{4 \inf_{\Omega} \lvert R_{g} \rvert \lVert u_{\beta} \rVert_{\calL^{2}(\Omega, g)}^{2} + \left( \beta - \beta_{0} \right) \lVert u_{\beta} \rVert_{\calL^{2}(\Omega, g)}^{2}}{a \lVert \nabla_{g} u_{\beta} \rVert_{\calL^{2}(\Omega, g)}^{2} + \left( - \sup_{\Omega} \lvert R_{g} \rvert + \beta_{0} \right) \lVert u_{\beta} \rVert_{\calL^{2}(\Omega, g)}^{2}} \\
& \leqslant \frac{\left( 4 \inf_{\Omega} \lvert R_{g} \rvert+ \left( \beta - \beta_{0} \right) \right) \lambda_{1}^{-1}  \lVert \nabla_{g} u_{\beta} \rVert_{\calL^{2}(\Omega, g)}^{2}}{\left(a  + \lambda_{1}^{-1} \left( - \sup_{\Omega} \lvert R_{g} \rvert + \beta_{0} \right) \right) \lVert \nabla_{g} u_{\beta} \rVert_{\calL^{2}(\Omega, g)}^{2}} \\
& =  \frac{\left( 4 \inf_{\Omega} \lvert R_{g} \rvert+ \left( \beta - \beta_{0} \right) \right) \lambda_{1}^{-1}}{a  + \lambda_{1}^{-1} \left( - \sup_{\Omega} \lvert R_{g} \rvert + \beta_{0} \right)}
\end{align*}
When $ \lambda_{1}^{-1} $ is small enough, i.e. we shrink the geodesic ball $ \Omega $ with small enough radius, the upper bound of $ \Gamma_{3} $ is small enough, provided the smallness of $ \lvert \beta_{0} \rvert $. It follows that when both the size of $ \Omega $ and $ \lvert \beta_{0} \rvert $ are small enough, we have
\begin{equation*}
\Gamma_{2} = 1 + \Gamma_{3} \leqslant \Gamma_{1}.
\end{equation*}
It follows that (\ref{local:eqnt3}) holds. We set
\begin{equation*}
v_{\beta, \Omega} = u_{\beta} + v.
\end{equation*}
By (\ref{local:eqnt3}), we conclude that
\begin{equation*}
T_{0, \beta} \leqslant \inf_{P \in \mathcal{P}_{\beta}} \max_{w \in P} \sup_{t > 0} J_{\beta}^{*}(w) \leqslant \sup_{t > 0} J_{\beta}^{*}(t v_{\beta, \Omega}) \leqslant \sup_{t > 0} J_{\beta_{0}}^{*}(t u_{\beta_{0}}) = T_{1} < K_{0}.
\end{equation*}
\end{proof}
\medskip

The next result shows that the limit does exist and solves (\ref{local:eqn1}) only when $ R_{g} < 0 $ on $ \Omega $, when $ \Omega $ is small enough, not locally conformally flat, and is of good shape. The key steps in the following proposition are inspired by a result of Trudinger and Aubin, see \cite[Prop.~4.4]{PL}.
\begin{proposition}\label{local:prop4}
Let $ (\Omega, g) $ be a Riemannian domain in $\R^n$, $ n \geqslant 3 $, not locally conformally flat, with $C^{\infty} $ boundary, with ${\rm Vol}_g(\Omega)$ and the Euclidean diameter of $\Omega$ sufficiently small. In addition, we assume that the first eigenvalue of Laplace-Beltrami operator $ -\Delta_{g} $ on $ \Omega $ with Dirichlet condition satisfies $ \lambda_{1} \rightarrow \infty $ as $ \Omega $ shrinks. Assume $ R_{g} < 0 $ within the small enough closed domain $ \bar{\Omega} $. Then for any $ \lambda > 0 $, the Dirichlet problem (\ref{local:eqn1}) has a real, positive, smooth solution $ u \in \calC^{\infty}(\Omega) \cap H_{0}^{1}(\Omega, g) \cap \calC^{0}(\bar{\Omega}) $.
\end{proposition}
\begin{proof} Choose the size of $ \Omega $ and $ \lvert \beta_{0} \rvert $ to be small enough so that Proposition \ref{local:prop3} holds. By Proposition \ref{local:prop2}, there exists a collection of smooth functions $ u_{\beta} > 0 $ in $ \Omega $, $ \beta \in [\beta_{0}, 0) $ satisfying
\begin{equation}\label{local:eqn18}
-a\Delta_{g} u_{\beta} + \left( R_{g} + \beta \right) u_{\beta} = \lambda u_{\beta}^{p-1} \; {\rm in} \; \Omega, u_{\beta} = 0 \; {\rm on} \; \partial \Omega.
\end{equation}
Our goal is to show that there exists some $ r > p $, such that
\begin{equation}\label{local:eqn19}
\lVert u_{\beta} \rVert_{\calL^{r}(\Omega, g)} \leqslant C, \forall \beta \in [\beta_{0}, 0).
\end{equation}
Pairing (\ref{local:eqn18}) with $ u_{\beta} $ on both sides, we have
\begin{equation}\label{local:eqn20}
a\lVert \nabla_{g} u_{\beta} \rVert_{\calL^{2}(\Omega, g)}^{2} + \int_{\Omega} \left( R_{g} + \beta \right) u_{\beta}^{2} \dvol = \lambda \lVert u_{\beta} \rVert_{\calL^{p}(\Omega, g)}^{p}.
\end{equation}
Recall the expression of $ J_{\beta}^{*}(u) $ in (\ref{local:eqn11}), $ K_{0} $ in (\ref{local:eqn13}) and the inequality
\begin{equation*}
J(u_{\beta}) \leqslant T_{1} < K_{0}, \forall \beta \in [\beta_{0}, 0)
\end{equation*}
in (\ref{local:eqnt2}), we apply (\ref{local:eqn20}) with $ p = \frac{2n}{n - 2} $,
\begin{align*}
J(u_{\beta}) \leqslant T_{1} \Rightarrow & \frac{a}{2} \lVert \nabla_{g} u_{\beta} \rVert_{\calL^{2}(\Omega, g)}^{2} - \frac{\lambda}{p} \lVert u_{\beta} \rVert_{\calL^{p}(\Omega, g)}^{p} - \frac{1}{2} \int_{\Omega} \left(R_{g} + \beta \right) u_{\beta}^{2} \dvol \leqslant T_{1} \\
\Rightarrow & \frac{a}{2} \lVert \nabla_{g} u_{\beta} \rVert_{\calL^{2}(\Omega, g)}^{2} \leqslant T_{1} + \frac{n - 2}{2n} \left( a\lVert \nabla_{g} u_{\beta} \rVert_{\calL^{2}(\Omega, g)}^{2} + \int_{\Omega} \left( R_{g} + \beta \right) u_{\beta}^{2} \dvol \right) \\
& \qquad + \frac{1}{2} \int_{\Omega} \left(R_{g} + \beta \right) u_{\beta}^{2} \dvol \\
\Rightarrow & \frac{a}{n}  \lVert \nabla_{g} u_{\beta} \rVert_{\calL^{2}(\Omega, g)}^{2} + \frac{1}{n} \int_{\Omega} \left( R_{g} + \beta \right) u_{\beta}^{2} \dvol \leqslant T_{1} \\
\Rightarrow & \frac{a}{n}  \lVert \nabla_{g} u_{\beta} \rVert_{\calL^{2}(\Omega, g)}^{2} + \frac{1}{n} \int_{\Omega} \left( R_{g} + \beta \right) u_{\beta}^{2} \dvol \leqslant \frac{T_{1}}{K_{0}} \cdot \frac{1}{n} \lambda^{\frac{2- n}{2}}a^{\frac{n}{2}} T^{\frac{n}{2}} \\
& a \lVert \nabla_{g} u_{\beta} \rVert_{\calL^{2}(\Omega, g)}^{2} + \int_{\Omega} \left( R_{g} + \beta \right) u_{\beta}^{2} \dvol \leqslant  \frac{T_{1}}{K_{0}} \cdot \lambda^{\frac{2- n}{2}}a^{\frac{n}{2}} T^{\frac{n}{2}}.
\end{align*}
Apply the last inequality above into (\ref{local:eqn20}), we conclude that
\begin{equation}\label{local:eqn21}
\lVert u_{\beta} \rVert_{\calL^{p}(\Omega, g)}^{p} \leqslant \frac{T_{1}}{K_{0}} \cdot  \lambda^{-\frac{n}{2}}a^{\frac{n}{2}} T^{\frac{n}{2}}, \forall \beta \in [\beta_{0}, 0).
\end{equation}
The next derivation is essentially due to Trudinger and Aubin. Let $ \delta > 0 $ be some constant. Pairing (\ref{local:eqn18}) with $ u_{\beta}^{1 + 2\delta} $ and denote $ w_{\beta} = u_{\beta}^{1 + \delta} $, it follows that
\begin{align*}
& \int_{\Omega} a \nabla_{g} u_{\beta} \cdot \nabla_{g} \left(u_{\beta}^{1 + 2\delta} \right) \dvol + \int_{\Omega} \left(R_{g} + \beta \right) u_{\beta}^{2 + 2\delta} \dvol = \lambda \int_{\Omega} u_{\beta}^{p + 2\delta} \dvol; \\
\Rightarrow & \frac{1 + 2\delta}{(1 + \delta )^{2}} \int_{\Omega} a \lvert \nabla_{g} w_{\beta} \rvert^{2} \dvol = \lambda \int_{\Omega} w_{\beta}^{2} u_{\beta}^{p-2} \dvol - \int_{\Omega} \left(R_{g} + \beta \right) w_{\beta}^{2} \dvol.
\end{align*}
Note that $ \frac{p-2}{p} = \frac{2}{n} $. According to the sharp Sobolev inequality on closed manifolds \cite[Thm.~2.3, Thm.~3.3]{PL}, it follows that for any $ \epsilon > 0 $,
\begin{align*}
\lVert w_{\beta} \rVert_{\calL^{p}(\Omega, g)}^{2} & \leqslant (1 + \epsilon) \frac{1}{T} \lVert \nabla_{g} w_{\beta} \rVert_{\calL^{2}(\Omega, g)}^{2} + C_{\epsilon}' \lVert w_{\beta} \rVert_{\calL^{2}(\Omega, g)}^{2} \\
& = (1 + \epsilon) \frac{1}{aT} \cdot \frac{(1 + \delta )^{2}}{1 + 2\delta} \left( \frac{1 + 2\delta}{(1 + \delta )^{2}} \int_{\Omega} a \lvert \nabla_{g} w_{\beta} \rvert^{2} \dvol \right) + C_{\epsilon}' \lVert w_{\beta} \rVert_{\calL^{2}(\Omega, g)}^{2} \\
& = (1 + \epsilon) \frac{\lambda}{aT} \cdot \frac{(1 + \delta )^{2}}{1 + 2\delta} \int_{\Omega} w_{\beta}^{2} u_{\beta}^{p - 2} \dvol + C_{\epsilon}  \lVert w_{\beta} \rVert_{\calL^{2}(\Omega, g)}^{2} \\
& \leqslant (1 + \epsilon) \frac{\lambda}{aT} \cdot \frac{(1 + \delta )^{2}}{1 + 2\delta} \lVert w_{\beta} \rVert_{\calL^{p}(\Omega, g)}^{2} \lVert u_{\beta} \rVert_{\calL^{p}(\Omega, g)}^{p - 2} + C_{\epsilon}  \lVert w_{\beta} \rVert_{\calL^{2}(\Omega, g)}^{2} \\
& = (1 + \epsilon) \frac{\lambda}{aT} \cdot \frac{(1 + \delta )^{2}}{1 + 2\delta} \lVert w_{\beta} \rVert_{\calL^{p}(\Omega, g)}^{2} \left( \lVert u_{\beta} \rVert_{\calL^{p}(\Omega, g)}^{p} \right)^{\frac{p - 2}{p}} + C_{\epsilon}  \lVert w_{\beta} \rVert_{\calL^{2}(\Omega, g)}^{2} \\
& \leqslant (1 + \epsilon) \frac{\lambda}{aT} \cdot \frac{(1 + \delta )^{2}}{1 + 2\delta} \lVert w_{\beta} \rVert_{\calL^{p}(\Omega, g)}^{2} \cdot \left( \frac{T_{1}}{K_{0}} \lambda^{-\frac{n}{2}}a^{\frac{n}{2}} T^{\frac{n}{2}} \right)^{\frac{2}{n}} + C_{\epsilon}  \lVert w_{\beta} \rVert_{\calL^{2}(\Omega, g)}^{2} \\
& = (1 + \epsilon) \cdot \frac{(1 + \delta )^{2}}{1 + 2\delta} \cdot \left( \frac{T_{1}}{K_{0}} \right)^{\frac{2}{n}} \lVert w_{\beta} \rVert_{\calL^{p}(\Omega, g)}^{2} + C_{\epsilon}  \lVert w_{\beta} \rVert_{\calL^{2}(\Omega, g)}^{2}.
\end{align*}
By (\ref{local:eqnt2}), the ratio $ \frac{T_{1}}{K_{0}} < 1 $ uniformly in $ \beta $. It follows that we can choose small enough $ \epsilon $ and $ \delta $ such that
\begin{equation*}
(1 + \epsilon) \cdot \frac{(1 + \delta )^{2}}{1 + 2\delta} \cdot \left( \frac{T_{1}}{K_{0}} \right)^{\frac{2}{n}} < 1, \forall \beta \in [\beta_{0}, 0).
\end{equation*}
It implies that
\begin{equation*}
\lVert w_{\beta} \rVert_{\calL^{p}(\Omega, g)}^{2} \leqslant K_{1}, \forall \beta \in [\beta_{0}, 0).
\end{equation*}
Note that $ \lVert w_{\beta} \rVert_{\calL^{p}(\Omega, g)} = \lVert u_{\beta} \rVert_{\calL^{p(1 + \delta)}(\Omega, g)}^{1 + \delta} $, it follows that (\ref{local:eqn19}) holds. Since $ r > p $, by bootstrapping method with Sobolev embedding, we conclude that
\begin{equation}\label{local:eqn22}
\lVert u_{\beta} \rVert_{\calC^{2, \alpha}(\Omega)} \leqslant \mathcal{K}', \forall \beta \in [\beta_{0}, 0].
\end{equation}
Due to Arzela-Ascoli, we pass the limit by letting $ \beta \rightarrow 0^{-} $ on both sides of (\ref{local:eqn19}), it follows that
\begin{equation*}
-a\Delta_{g} u + R_{g} u = \lambda u^{p-1} \; {\rm in} \; \Omega, u = 0 \; {\rm on} \; \partial \Omega.
\end{equation*}
Here $ u = \lim_{\beta \rightarrow 0^{-}} u_{\beta} $ pointwise. Another bootstrapping method shows that $ u \in \calC^{\infty}(\Omega) $ and $ u \geqslant 0 $ since $ u $ is the limit of a positive sequence. Lastly we show that $ u > 0 $. Due to Kazdan and Warner \cite{KW}, $ \eta_{1} > 0 $ implies that
\begin{equation*}
\lambda_{\beta} : = \inf_{u \geqslant 0, u \in H^{1}(M, g)} \frac{a \int_{M} \lvert \nabla_{g} u \rvert^{2} \dvol + \int_{M} \left(R_{g} + \beta \right) u^{2} \dvol}{\left( \int_{M} u^{p} \dvol \right)^{\frac{2}{p}}} > 0
\end{equation*}
provided that $ \lvert \beta \rvert < 0 $ with small enough absolute value. Making $ \lvert \beta_{0} \rvert $ smaller if necessary, we apply (\ref{local:eqn20}) again,
\begin{equation*}
0 < \lambda_{\beta} \leqslant \frac{a \int_{\Omega} \lvert \nabla_{g} u_{\beta} \rvert^{2} \dvol + \int_{\Omega} \left(R_{g} + \beta \right) u_{\beta}^{2} \dvol}{\left( \int_{\Omega} u_{\beta}^{p} \dvol \right)^{\frac{2}{p}}} = \lambda \lVert u_{\beta} \rVert_{\calL^{p}(\Omega, g)}^{\frac{p-2}{p}}, \forall \beta \in [\beta_{0}, 0)
\end{equation*}
Since $ u = \lim_{\beta \rightarrow 0^{-}} u_{\beta} $ in the classical sense, it follows that $ \lVert u \rVert_{\calL^{p}(\Omega, g)} > 0 $. By the maximum principle, it follows that $ u > 0 $ in $ \Omega $. The boundary regularity of $ u $ follows exactly the same as in Brezis and Nirenberg \cite{Niren3}.
\end{proof}
\medskip

When the Riemannian domain is locally conformally flat, i.e. under conformal change the metric is flat, we need the results of Clapp, Faya, and Pistoia \cite{CFP} to discuss the Riamannian domain $ \Omega $ comes from some locally conformally flat manifolds. Given some subgroup $ G \subset O(n+1) $, we need to introduce $ G $-invariance on domains, on functions, and on differential operators.
\begin{definition}\label{local:def2} Let $ G $ be a subgroup of $ O(n) $. Let $ \Omega \subset \R^{n} $ be an open subset. Let $ Q(x) $ be a function defined on $ \bar{\Omega} $. Let $ L $ be a differential operator acting on $ \calC_{c}^{\infty}(\Omega) $. 

(i) We say that the domain $ \Omega $ is $ G $-invariant if for every element $ x \in \Omega $ and every element $ U \in G $, the point $ Ux \in \Omega $ also;

(ii) We say that the function $ Q $ is $ G $-invariant of for every element $ U \in G $, we have $ Q(Ux) = x, \forall x \in \bar{\Omega} $;

(iii) We say that the differential operator $ L $ is $ G $-invariant if for every $ U \in G $,
\begin{equation*}
(Lf) \circ U = L(f \circ U), \forall f \in \calC_{c}^{\infty}(\Omega).
\end{equation*}
\end{definition}
\begin{remark}\label{local:re4}
It is well-known that the Euclidean Laplacian $ -\Delta_{e} $ acting on $ H_{0}^{1}(\Omega) $ is $ G $-invariant if and only if the domain $ \Omega $ is $ G $-invariant.
\end{remark}
Followed by Definition \ref{local:def2}, the Dirichlet problem of the nonlinear Poisson equation with critical exponent is given as below.
\begin{theorem}\label{local:thm4}\cite[Thm.~1.1]{CFP}
Let $ G $ be a subgroup of $ O(n) $. Let $ \Omega $ be an open subset of $ \R^{n} $ that is $ G $-invariant. Let $ \rho $ be a point in $ \Omega $ such that $ U\rho = \rho, \forall U \in G $. Assume that $ Q \in \calC^{2}(\bar{\Omega}) $ that is $ G $-invariant and $ \min_{x \in \bar{\Omega}} Q(x) > 0 $. For each $ \epsilon > 0 $, denote the space $ \Omega_{\epsilon} $ by
\begin{equation*}
\Omega_{\epsilon} = \lbrace x \in \Omega | \lvert x - \rho \rvert > \epsilon \rbrace.
\end{equation*}
Consider the Dirichlet problem
\begin{equation}\label{local:eqn23}
-\Delta_{e}u = Q(x) u^{p-1} \; {\rm in} \; \Omega_{\epsilon}, u \equiv 0 \; {\rm in} \; \partial \Omega_{\epsilon}, u > 0 \; {\rm in} \; \Omega_{\epsilon}.
\end{equation}
If $ \nabla Q(\rho) \neq 0 $, then there exists an $ \epsilon_{0} $ such that for each $ \epsilon \in (0, \epsilon_{0}) $ the Dirichlet problem (\ref{local:eqn23}) has a $ G $-invariant smooth solution.
\end{theorem}
\begin{remark}\label{local:re5} 
(i) The subgroup $ G $ can be chosen as the trivial subgroup of $ O(n) $.

(ii) The existence of the solution of (\ref{local:eqn23}) is due to the method of calculus of variation. The solution $ u $ is exactly one critical point of
\begin{equation*}
J(u) = \frac{1}{2} \int_{\Omega_{1}} \lvert \nabla u \rvert^{2} dx - \frac{1}{p} \int_{\Omega_{1}} Qu^{p} dx.
\end{equation*}
The details of the proof shows that the solution of (\ref{local:eqn23}) has a one-to-one correspondence to the following two quantities
\begin{equation*}
\eta_{0} = -\frac{\nabla Q(\rho)}{\lvert \nabla Q(\rho) \rvert}, d_{0} = \left( \frac{(n - 2)\beta}{2^{n-2} \gamma} \cdot \frac{Q(\rho)}{\lvert \nabla Q(\rho) \rvert} \right)^{\frac{1}{n-1}}, n \geqslant 4
\end{equation*}
with some positive constants $ \beta $ and $ \gamma $, and
\begin{equation*}
\eta_{0} = \left( \frac{\alpha - \sqrt{\alpha^{2} + \gamma^{2} \left\lvert \frac{\nabla Q(\rho)}{Q(\rho) }\right\rvert^{2} }}{\gamma \left\lvert \frac{\nabla Q(\rho)}{Q(\rho) }\right\rvert^{2} }\right) \cdot \frac{\nabla Q(\rho)}{Q(\rho) }, d_{0} = \sqrt{\frac{\beta}{(1 + \lvert \eta_{0} \rvert^{2}) \cdot \left(\alpha - \gamma \left\langle \frac{\nabla Q(\rho)}{Q(\rho) }, \eta_{0} \right\rangle \right)}}
\end{equation*}
with some positive constants $ \alpha, \beta, \gamma $. 

In other words, the evaluation of $ \frac{\nabla Q(\rho)}{\lvert \nabla Q(\rho) \rvert} $ and $ \frac{Q(\rho)}{\lvert \nabla Q(\rho) \rvert} $ determines the solution $ u $ of (\ref{local:eqn23}), and vice versa.
\end{remark}
\medskip

Let $ \tilde{g} = \phi^{p-2} g $ for some conformal factor $ \phi > 0 $ either on $ (M, g) $ or $ (\bar{M}, g) $, the conformal invariance of the conformal Laplacian is in the following sense:
\begin{equation}\label{local:eqn24}
\tilde{g} = \phi^{p-2} g \Rightarrow -a\Delta_{\tilde{g}} + R_{\tilde{g}} = \phi^{-\frac{n+2}{n-2}} \left(-a\Delta_{g} + R_{g} \right) \phi \Leftrightarrow \Box_{\tilde{g}} = \phi^{1 - p} \Box_{g} \phi.
\end{equation}
By conformal invariance of the conformal Laplacian, we apply Theorem \ref{local:thm4} to obtain the existence of a positive, smooth solution of the Yamabe equation with Dirichlet condition on some open subset of the closed manifold $ M $ or in interior of some compact manifold $ \bar{M} $, provided that the manifolds are locally conformally flat.  
\begin{proposition}\label{local:prop5}
Let $ (\Omega, g) $ be a Riemannian domain in $\R^n$, $ n \geqslant 3 $, with $C^{\infty} $ boundary. Let the metric $ g $ be locally conformally flat on some open subset $ \Omega' \supset \bar{\Omega} $. For any point $ \rho \in \Omega $ and any positive constant $ \epsilon $, denote the region $ \Omega_{\epsilon} $ to be
\begin{equation*}
\Omega_{\epsilon} = \lbrace x \in \Omega | \lvert x - \rho \rvert > \epsilon \rbrace.
\end{equation*}
Assume that $ Q \in \calC^{2}(\bar{\Omega}) $, $ \min_{x \in \bar{\Omega}} Q(x) > 0 $ and $ \nabla Q(\rho) \neq 0 $. Then there exists some $ \epsilon_{0} $ such that for every $ \epsilon \in (0, \epsilon_{0}) $ the Dirichlet problem
\begin{equation}\label{local:eqn25}
-a\Delta_{g}u + R_{g} u = Qu^{p-1} \; {\rm in} \; \Omega_{\epsilon}, u = 0 \; {\rm on} \; \partial \Omega_{\epsilon}
\end{equation}
has a real, positive, smooth solution $ u \in \calC^{\infty}(\Omega_{\epsilon}) \cap H_{0}^{1}(\Omega_{\epsilon}, g) \cap \calC^{0}(\bar{\Omega_{\epsilon}}) $.
\end{proposition}
\begin{proof} Since $ g $ is locally conformally flat within $ \Omega_{\epsilon} \subset \Omega $, there exists a conformal change $ \tilde{g} = \psi^{p-2} g $ with $ \psi > 0 $ on $ \Omega_{\epsilon} $ such that the Riemannian curvature tensor $ \text{Rm}_{\tilde{g}} \equiv 0 $ on $ \Omega $. We may denote $ \tilde{g} : = e $ locally in $ \Omega_{\epsilon} $ since $ \tilde{g} $ is locally Euclidean. By Theorem \ref{local:thm4}, there exists some $ \epsilon_{0} $ such that for each $ \epsilon \in (0, \epsilon) $ we can find some smooth function $ u_{0} > 0 $ in $ \Omega_{\epsilon} $ that solves
\begin{equation}\label{local:eqn26} 
-a\Delta_{e} u_{0} = Q u_{0}^{p-1} \; {\rm in} \; \Omega_{\epsilon}, u = 0 \; {\rm on} \; \partial \Omega_{\epsilon}.
\end{equation}
Due to the conformal invariance of conformal Laplacian (\ref{local:eqn24}), we conclude that the function
\begin{equation*}
u : = \psi^{-1} u_{0}
\end{equation*}
solves (\ref{local:eqn25}) in the same region $ \Omega_{\epsilon} $. Clearly $ u > 0 $ since both $ \psi $ and $ u_{0} $ are positive.
\end{proof}
\begin{remark}\label{local:re6}
The result of Proposition \ref{local:prop5} implies a local solution of the Yamabe equation within some topologically nontrivial open subset of either the closed manifold $ M $ or the interior of the compact manifolds $ \bar{M} $, provided that the manifolds have vanishing Weyl tensors.
\end{remark}
\medskip

\section{The Global Analysis on Closed Manifolds, Not Locally Conformally Flat Case}
In this section, we apply monotone iteration scheme to construct the solution of the PDE
\begin{equation}\label{closed:eqn1}
-a\Delta_{g} u + R_{g} u = Su^{p-1} \; {\rm on} \; M
\end{equation}
with $ S \in \mathcal{A}_{1} $ on closed manifolds $ M $ that are not locally conformally flat. Here the set $ \mathcal{A}_{1} $ is defined to be
\begin{equation}\label{closed:eqnr1}
\begin{split}
& \mathcal{A}_{1} : = \lbrace S \in \calC^{\infty}(M) : S \equiv \lambda > 0 \; {\rm on} \; \bar{O}, \text{$ c > 0 $ is an arbitrary constant,} \\
& \qquad \text{ $ O \subset \bar{O} \subset M $ is any open submanifold with smooth $ \partial O $ such that }\\
& \qquad \text{the metric $ g $ is not locally conformally flat in $ O $.} \rbrace.
\end{split}
\end{equation}
We seek for solutions of (\ref{closed:eqn1}) when $ \eta_{1} $, the first eigenvalue of conformal Laplacian $ \Box_{g} $, is positive, since a necessary condition of positive first eigenvalue of conformal Laplacian is that the scalar curvature is positive somewhere by \cite{KW2}. We will show that the solution of (\ref{closed:eqn1}) is positive and smooth. It follows that $ S $ is the prescribed scalar curvature with respect to the metric $ \tilde{g} = u^{p-2} g $. Throughout this section, we denote $ (M, g) $ as closed manifolds, not locally conformally flat, with $ n = \dim M \geqslant 3 $.

A variation of monotone iteration scheme introduced by Kazdan and Warner \cite{KW} is crucial in this argument.
\begin{theorem}\label{closed:thm1}\cite[Lemma.~2.6]{KW}
Let $ (M, g) $ be a closed Riemannian manifold. Assume $ S_{1}, S_{2} \in \calC^{\infty}(M) $. If there exist functions $ u_{-}, u_{+} \in H^{1}(M, g) \cap \calC^{0}(M) $ such that
\begin{equation}\label{closed:eqn2}
\begin{split}
-a\Delta_{g} u_{+} + S_{1} u_{+} & \geqslant S_{2} u_{+}^{p-1} \; {\rm on} \; M; \\
-a\Delta_{g} u_{-} + S_{1} u_{-} & \leqslant S_{2} u_{-}^{p-1} \; {\rm on} \; M
\end{split}
\end{equation}
in the weak sense, with $ 0 \leqslant u_{-} \leqslant u_{+} $ and $ u_{-} \not\equiv 0 $, then there exists a function $ u \in \calC^{\infty}(M) $ satisfying (\ref{closed:eqn1}). Moreover, $ u_{-} \leqslant u \leqslant u_{+} $.
\end{theorem}
\begin{remark}\label{closed:re1}
In original proof of Theorem \ref{closed:eqn1}, Kazdan and Warner requires $ u_{-}, u_{+} \in H^{2, q}(M, g) $ with some $ q > n $. We showed in \cite[\S2]{XU3} that these hypotheses can be reduced in the theorem above. A weak maximum principle, which is essentially due to Stampacchia, was applied.
\end{remark}
\medskip

The next two lemmas concern the construction of sub-solution and super-solution of (\ref{closed:eqn1}). We point out that the sub-solution is merely $ \calC^{0}(M) \cap H^{1}(M, g) $, but the super-solution is $ \calC^{\infty}(M) $. We start with a special case by adding a restriction with respect to the sign of $ R_{g} $.
\begin{lemma}\label{closed:lemma1}
Let $ (M, g) $ be a closed Riemannian manifold that is not locally conformally flat.  Let $ S \in \mathcal{A}_{1} $ be a smooth function satisfies $ S = \lambda > 0 $ for some positive constant $ \lambda $ on some open subset $ O \subset M $. Assume further that $ R_{g} < 0 $ somewhere in $ O $. Then there exists a function $ u_{-} \in H^{1}(M, g) \cap \calC^{0}(M) $ such that
\begin{equation}\label{closed:eqn2a}
-a\Delta_{g} u_{-} + R_{g} u_{-} \leqslant S u_{-}^{p-1} \; {\rm on} \; M
\end{equation}
in the weak sense. In addition, $ u_{-} \geqslant 0 $ and $ u_{-} \not\equiv 0 $.
\end{lemma}
\begin{proof}
Without loss of generality, we may assume that $ O $ is connected since otherwise we can work on some open subset of any connected component of $ O $ such that $ R_{g} < 0 $ somewhere within it. Choose a small enough connected, open, proper subset $ \Omega \subset O \subset M $ such that the hypotheses of Proposition \ref{local:prop4} are satisfied. In addition, $ S = \lambda > 0 $ on $ \Omega $ due to the assumption. We shrink $ \Omega $ if necessary so that $ R_{g} < 0 $ on $ \Omega $. By Proposition \ref{local:prop4}, the following PDE
\begin{equation}\label{closed:eqn3}
-a\Delta_{g} u_{1} + R_{g} u_{1} = \lambda u_{1}^{p-1} \; {\rm in} \; \Omega, u_{1} \equiv 0 \; {\rm on} \; \partial \Omega
\end{equation}
admits a positive solution $ u_{1} \in \calC^{\infty}(\Omega) \cap \calC^{0}(\bar{\Omega}) $. Denote
\begin{equation}\label{closed:eqn4}
u_{-} : = \begin{cases} u_{1} & \; {\rm in} \; \Omega \\ 0 & \; {\rm in} \; M \backslash \bar{\Omega} \end{cases}.
\end{equation}
It follows from the same argument in Theorem 4.3 of \cite{XU3} that $ u_{-} $ satisfies
\begin{equation*}
-a\Delta_{g} u_{-} + R_{g} u_{-} \leqslant \lambda u_{-}^{p-1} = S u_{-}^{p-1} \; {\rm on} \; M
\end{equation*}
in the weak sense. But $ u_{-} \equiv 0 $ outside $ \Omega $, it follows that (\ref{closed:eqn2a}) holds with $ u_{-} $ given in (\ref{closed:eqn4}) in the weak sense. Since $ u_{-} $ is exactly an extension of a smooth function in $ \Omega $ and a continuous function on $ \bar{\Omega} $ by zero, we conclude that $ u_{-} \in H^{1}(M, g) \cap \calC^{0}(M) $. Due to the construction of $ u_{-} $ in (\ref{closed:eqn4}), it is straightforward to see that $ u_{-} \geqslant 0 $ on $ M $, which is not identically zero.
\end{proof}
\medskip

We now turn to the construction of a super-solution of (\ref{closed:eqn1}). There is a trivial case, say $ S = \lambda > 0 $ for some constant $ \lambda $. Due to the Yamabe problem \cite{XU3}, we can always apply a conformal change so that the background metric $ g $ has $ R_{g} = \lambda $. Thus the super-solution must satisfy
\begin{equation*}
-a\Delta_{g} u_{+} + \lambda u_{+} \geqslant \lambda u_{+}^{p-1}.
\end{equation*}
Clearly the constant function $ u_{+} = 1 $ is a super-solution of (\ref{closed:eqn1}). In this case, we do not need the nontrivial sub-solution constructed in Proposition \ref{closed:lemma1}, i.e. we can take $ u_{-} = 0 $ on $ M $. Taking $ u_{+} = 1 $ in the monotone iteration scheme, every $ u_{k} \equiv 1 $ in the iteration steps. Therefore the limit is $ u \equiv 1 $ on $ M $, which solves the PDE (\ref{closed:eqn1}). For globally constant prescribing scalar curvature functions, we do not need the following gluing procedure in Lemma \ref{closed:lemma2}. In fact, Lemma \ref{closed:lemma2} does not work for this case. The point is that the strict inequality in (\ref{closed:eqn7}) is actually an equality. Therefore the quantity $ \beta $ we defined in (\ref{closed:eqn8}) is zero, hence there is no positive $ \gamma $ such that (\ref{closed:eqn9}) holds. In a word, we have no room to compensate possible negative terms when we apply the gluing procedure. Instead we need to introduce a perturbation method to resolve the issue of prescribing globally constant functions, which were widely used in proving the Escobar problem and the Yamabe problem in \cite{XU4, XU5, XU3}. In particular, we apply a conformal change so that the background metric $ g $ we start with admits the scalar curvature $ R_{g} $ that is negative somewhere. Although not the same, but this is similar to the conformal normal coordinates Aubin did, see \cite{PL}. In order to construct a super-solution of (\ref{closed:eqn1}) for not globally constant function $ S $, we need to glue $ u_{1} $ from (\ref{closed:eqn3}) with the first eigenfunction of $ \Box_{g} $ appropriately, as the following lemma will show. 
\begin{lemma}\label{closed:lemma2}
Let $ (M, g) $ be a closed Riemannian manifold that is not locally conformally flat. Let $ \eta_{1} $, the first eigenvalue of $ \Box_{g} $, be positive. Let $ S \in \mathcal{A}_{1} $ be a non-constant smooth function satisfies $ S = \lambda > 0 $ for some positive constant $ \lambda $ on some open subset $ O \subset M $. Assume that $ R_{g} < 0 $ somewhere in $ O $. In addition, we assume that the metric $ g $ is not locally conformally flat within an open subset of the region $ \lbrace x \in O | R_{g}(x) < 0 \rbrace $. Then there exists a function $ u_{+}  \in \calC^{\infty}(M) $ such that
\begin{equation}\label{closed:eqn5}
-a\Delta_{g} u_{+} + R_{g} u_{+} \geqslant S u_{+}^{p-1} \; {\rm on} \; M
\end{equation}
in the weak sense. In addition, $ u_{+} \geqslant u_{-} \geqslant 0 $.
\end{lemma}
\begin{proof}
Without loss of generality, we may assume that $ O $ is connected since otherwise we can work on some open subset of any connected component of $ O $ such that $ R_{g} < 0 $ somewhere within it. 

By assumption, the metric $ g $ is not locally in an open subset of $ \lbrace x \in O | R_{g}(x) < 0 \rbrace $. Thus we choose $ \Omega $ to be the same as in Lemma \ref{closed:lemma1}, or shrinking $ \Omega $ if necessary. It follows that $ u_{1} > 0 $ solves (\ref{closed:eqn3}) as in Lemma \ref{closed:lemma1}. With the hypothesis $ \eta_{1} > 0 $, there exists a smooth eigenfunction $ \varphi > 0 $ on $ M $ such that
\begin{equation}\label{closed:eqn6}
-a\Delta_{g} \varphi + R_{g} \varphi = \eta_{1} \varphi \; {\rm on} \; M.
\end{equation}
Any scaling $ \theta \varphi $ also solves (\ref{closed:eqn6}). We choose $ \theta $ small enough so that
\begin{equation*}
\eta_{1} \min_{M} \left(\theta \varphi \right) > 2^{p-2} \max_{M} S \max_{M} \left( \theta \varphi \right)^{p-1} \Leftrightarrow \eta_{1} \min_{M} \varphi > 2^{p-2} \theta^{p-2} \cdot \max_{M} S \max_{M} \varphi^{p-1}.
\end{equation*}
Denote $ \phi : = \theta \varphi $, it follows that
\begin{equation}\label{closed:eqn7}
-a\Delta_{g} \phi + R_{g} \phi = \eta_{1} \phi > 2^{p-2} \max_{M} S \max_{M} \phi^{p-1} \geqslant \max_{M} S \max_{M} \phi^{p-1} \geqslant S \phi^{p-1} \; \text{pointwise on} \; M.
\end{equation}
Note that $ S = \lambda > 0 $ within $ \Omega $, the inequality (\ref{closed:eqn7}) becomes
\begin{equation}\label{closed:eqn8}
-a\Delta_{g} \phi + R_{g} \phi = \eta_{1} \phi > 2^{p-2} \lambda \max_{\Omega} \phi^{p-1} \geqslant 2^{p-2} \lambda \phi^{p-1} > \lambda \phi^{p-1} \; \text{pointwise on} \; \Omega.
\end{equation}
Set
\begin{equation}\label{closed:eqn9}
\beta = \max_{\Omega} \left( \eta_{1} \phi - 2^{p-2} \lambda \phi^{p-1} \right) \geqslant \eta_{1} \phi - 2^{p-2} \lambda \phi^{p-1} \; \text{pointwise on} \; \Omega.
\end{equation}
The following arguments are essentially the same as in Theorem 4.3 of \cite{XU3}. We just sketch the key steps and constructions below. Pick up $ \gamma \ll 1 $ such that
\begin{equation}\label{closed:eqn10}
0 < 20 \lambda \gamma + 2\gamma \cdot \sup_{M} \lvert R_{g} \rvert \gamma < \frac{\beta}{2}, 31 \lambda (\phi + \gamma)^{p-2} \gamma < \frac{\beta}{2}. 
\end{equation}
The choice of $ \gamma $ is dimensional specific. Set
\begin{align*}
V & = \lbrace x \in \Omega: u_{1}(x) > \phi(x) \rbrace, V' = \lbrace x \in \Omega: u_{1}(x) < \phi(x) \rbrace, D = \lbrace x \in \Omega : u_{1}(x) = \phi(x) \rbrace, \\
D' & = \lbrace x \in \Omega : \lvert u_{1}(x) - \phi(x) \rvert < \gamma \rbrace, D'' = \left\lbrace x \in \Omega : \lvert u_{1}(x) - \phi(x) \rvert > \frac{\gamma}{2} \right\rbrace.
\end{align*}
If $ \phi \geqslant u_{1} $ pointwise, then $ \phi $ is a super-solution. If not, a good candidate of super-solution will be $ \max \lbrace u_{1}, \phi \rbrace $ in $ \Omega $ and $ \varphi $ outside $ \Omega $, this is an $ H^{1} \cap \calC^{0} $-function. Let $ \nu $ be the outward normal derivative of $ \partial V $ along $ D $. If $ \frac{\partial u_{1}}{\partial \nu} = - \frac{\partial \phi}{\partial \nu} $ on $ D $ then the super-solution has been constructed. However, this is in general not the case. If not, then $ \frac{\partial u_{1} - \partial \phi}{\partial \nu} \neq 0 $, which follows that $ 0 $ is a regular point of the function $ u_{1} - \varphi $ and hence $ D $ is a smooth submanifold of $ \Omega $. Define
\begin{equation}\label{closed:eqn11}
\Omega_{1} = V \cap D'', \Omega_{2} = V' \cap D'', \Omega_{3} = D'.
\end{equation}
Clearly $ \Omega_{i}, i = 1, 2, 3 $ are open sets, $ \bigcup_{i} \Omega_{i} = \Omega $ and $ \Omega_{1} \cap \Omega_{2} = \emptyset $. Note that $ D, D' $ will never intersect $ \partial \Omega $. Without loss of generality, we may assume that all $ \Omega_{i}, i = 1, 2, 3 $ are connected, since if not, then we do local analysis in all components where $ u_{1} - \phi $ changes sign and combine results together. For any $ 0 < \epsilon \ll \frac{\delta}{4} $, set
\begin{align*}
D_{1} & = \lbrace x \in V: \text{dist}(x, \partial D') < \epsilon \rbrace, D_{1}' = \lbrace x \in V': \text{dist}(x, \partial D') < \epsilon \rbrace, \\
D_{2} & = \lbrace x \in V: \text{dist}(x, \partial D'') < \epsilon \rbrace, D_{2}' = \lbrace x \in V': \text{dist}(x, \partial D'') < \epsilon \rbrace, \\
E_{1} & = \partial D_{1} \cap \Omega_{1} \cap \Omega_{3}, E_{2} = \partial D_{2} \cap \Omega_{1}, E_{1}' = \partial D_{1}' \cap \Omega_{2} \cap \Omega_{3}, E_{2}' = \partial D_{2}' \cap \Omega_{2}; \\
F_{1} & = \left(\Omega_{1} \cap \Omega_{3} \right) \backslash \left( D_{1} \cup D_{2} \right), F_{2} =  \left(\Omega_{2} \cap \Omega_{3} \right) \backslash \left( D_{1}' \cup D_{2}' \right).
\end{align*}
In conclusion, $ D_{1}, D_{2}, E_{1}, E_{2}, F_{1} $ are subsets of $ V $, in which $ u_{1} > \phi $; $ D_{1}', D_{2}', E_{1}', E_{2}', F_{2} $ are subsets of $ V' $, in which $ u_{1} < \phi $. Note that $ \partial F_{1} = E_{1} \cup E_{2} $, $ \partial F_{2} = E_{1}' \cup E_{2}' $. We now construct appropriate partition of unity subordinate to $ \Omega_{i}, i = 1, 2, 3 $. Choose a local coordinate system $ \lbrace x_{i} \rbrace $ on $ \Omega $, we choose $ v_{1} $ and $ v_{2} $ such that
\begin{equation}\label{closed:eqn12}
\begin{split}
&  -\sum_{i, j} g^{ij} (u_{1} - \phi - \gamma) \frac{\partial^{2} v_{1}}{\partial x_{i} \partial x_{j}} - 2 \sum_{i, j} g^{ij} \left( \frac{\partial u_{1}}{\partial x_{j}} - \frac{\partial \phi}{\partial x_{j}} \right) \frac{\partial v_{1}}{\partial x_{i}} \\
& \qquad - \sum_{i, j, k} (u_{1} - \phi - \gamma) g^{ij} \Gamma_{ij}^{k} \frac{\partial v_{1}}{\partial x_{k}}  = 0 \; {\rm in} \; F_{1}; \\
& v_{1} = 0 \; {\rm on} \; E_{2}; v_{1}  = 1 \; {\rm on} \; E_{1}.
\end{split}
\end{equation}
\begin{equation}\label{closed:eqn13}
\begin{split}
&  -\sum_{i, j} g^{ij} \frac{\partial^{2} v_{2}}{\partial x_{i} \partial x_{j}} - \sum_{i, j, k} g^{ij} \Gamma_{ij}^{k} \frac{\partial v_{2}}{\partial x_{k}}  = 0 \; {\rm in} \; F_{2}; \\
& v_{2} = 0 \; {\rm on} \; E_{2}'; v_{2}  = 1 \; {\rm on} \; E_{1}'.
\end{split}
\end{equation}
Since $ E_{1} \cap E_{2} = E_{1}' \cap E_{2}' = \emptyset $. The linear PDEs in (\ref{closed:eqn13}) and (\ref{closed:eqn14}) exist unique solutions, respectively. In addition, both $ v_{1}, v_{2} \in [0, 1] $ are smooth, due to maximum principle and standard elliptic regularity. We then construct $ \chi_{1} $ and $ \chi_{2} $ as smooth functions
\begin{equation}\label{closed:eqn14}
\begin{split}
v_{1}'' & : = \min \lbrace 1, \max \lbrace v_{1}, 0 \rbrace \rbrace, v_{1}'(x) : = v_{1}'' * \psi_{\epsilon}(x), \forall x \in \overline{\Omega_{1} \cap \Omega_{3}}; \\
\chi_{1}(x) & : = \begin{cases} 1, & x \in \Omega_{1} \backslash \bar{\Omega}_{3} \\ v_{1}'(x), & x \in \overline{\Omega_{1} \cap \Omega_{3}} \\ 0, & x \in \Omega \backslash \bar{\Omega}_{1} \end{cases}. 
\end{split}
\end{equation}
\begin{equation}\label{closed:eqn15}
\begin{split}
v_{2}'' & : = \min \lbrace 1, \max \lbrace v_{2}, 0 \rbrace \rbrace, v_{2}'(x) : = v_{2}'' * \psi_{\epsilon}(x), \forall x \in \overline{\Omega_{2} \cap \Omega_{3}}; \\
\chi_{2}(x) & : = \begin{cases} 1, & x \in \Omega_{2} \backslash \bar{\Omega}_{3} \\ v_{1}'(x), & x \in \overline{\Omega_{2} \cap \Omega_{3}} \\ 0, & x \in \Omega \backslash \bar{\Omega}_{2} \end{cases}. 
\end{split}
\end{equation}
The last cut-off function $ \chi_{3} $ is then defined by
\begin{equation*}
1 = \chi_{1}(x) + \chi_{2}(x) + \chi_{3}(x), \forall x \in \Omega.
\end{equation*}
We now define
\begin{equation}\label{closed:eqn16}
\bar{u} = \chi_{1} u_{1} + \chi_{2} \phi + \chi_{3} \left( \phi + \gamma \right).
\end{equation}
Due to the same argument in Theorem 4.3 of \cite{XU3}, we conclude that $ \bar{u} \in \calC^{\infty}(\Omega) $ satisfies
\begin{equation*}
-a\Delta_{g} \bar{u} + R_{g} \bar{u} \geqslant \lambda \bar{u}^{p-1} = S \bar{u}^{p-1}
\end{equation*}
in $ \Omega $ pointwise. The argument consists of the analysis of $ \bar{u} $ on five different sets $ \left( \Omega_{1} \backslash \bar{\Omega}_{3} \right) $, $ \left( \Omega_{2} \backslash \bar{\Omega}_{3} \right) $, $ \left( \Omega_{3} \backslash \left( \overline{\Omega_{1} \cup \Omega_{2}} \right) \right) $, $ \left( \Omega_{1} \cap \Omega_{3} \right) $ and $\left( \Omega_{2} \cup \Omega_{3} \right) $, respectively, whose union is $ \Omega $. By the definition of $ \bar{u} $, it is immediate that $ \bar{u} \geqslant u_{1} $. Define
\begin{equation}\label{closed:eqn17}
u_{+} : = \begin{cases} \bar{u}, & \; {\rm in} \; \Omega; \\ \phi, & \; {\rm in} \; \bar{M} \backslash \Omega. \end{cases}
\end{equation}
It follows that $ u_{+} \in \calC^{\infty}(M) $ since $ \bar{u} = \phi $ near $ \partial \Omega $. By (\ref{closed:eqn7}) we conclude that
\begin{equation}\label{closed:eqn18}
-a\Delta_{g} u_{+} +R_{g} u_{+} \geqslant S u_{+}^{p-1} \; {\rm on} \; M.
\end{equation}
Critically, $ 0 \leqslant u_{-} \leqslant u_{+} $ on $ \bar{M} $ due to the construction of $ u_{+} $ in (\ref{closed:eqn16}).
\end{proof}
\medskip

With the aids of Lemma \ref{closed:lemma1} and \ref{closed:lemma2}, we can construct the solution of (\ref{closed:eqn1}) with $ S \in \mathcal{A}_{1} $.
\begin{theorem}\label{closed:thm2}
Let $ (M, g) $ be a closed Riemannian manifold that is not locally conformally flat. Let $ \eta_{1} $, the first eigenvalue of $ \Box_{g} $, be positive. Let $ S \in \mathcal{A}_{1} $ be a smooth function satisfies $ S = \lambda > 0 $ for some positive constant $ \lambda $ on some open subset $ O \subset M $. Assume that $ R_{g} < 0 $ somewhere in $ O $.  In addition, we assume that the metric $ g $ is not locally conformally flat within an open subset of the region $ \lbrace x \in O | R_{g}(x) < 0 \rbrace $. Then there exists a positive function $ u \in \calC^{\infty}(M) $ that solves
\begin{equation*}
-a\Delta_{g} u + R_{g} u = Su^{p-1} \; {\rm on} \; M.
\end{equation*}
Equivalently, $ S $ is the prescribed scalar curvature with respect to the metric $ \tilde{g} = u^{p-2} g $.
\end{theorem}
\begin{proof} When $ S $ is a globally positive constant, it is the trivial case which is the result of the Yamabe problem. When $ S $ is not a globally constant, we apply Lemma \ref{closed:lemma1} and \ref{closed:lemma2}, it follows that there exist $ u_{-} \in H^{1}(M, g) \cap \calC^{0}(M) $, $ u_{+} \in \calC^{\infty}(M) $, $ 0 \leqslant u_{-} \leqslant u_{+} $ with $ u_{-} \not\equiv 0 $ such that the inequalities in (\ref{closed:eqn2}) hold with $ S_{1} = R_{g} $ and $ S_{2} = S $. Applying Theorem \ref{closed:thm1}, there exists a function $ u \in \calC^{\infty}(M) $ that solves (\ref{closed:eqn1}). In addition, $ u \geqslant u_{-} \geqslant 0 $ with $ u \not\equiv 0 $. Due to standard maximum principle as discussed in \cite[\S2]{PL}, it follows that $ u > 0 $ on $ M $.
\end{proof}
\medskip

In Theorem \ref{closed:thm2}, we have two restrictions: (i) $ R_{g} < 0 $ somewhere; (ii) the metric $ g $ must not be locally conformally flat within some subset of the region on which $ R_{g} < 0 $. The following result helps us to remove the restrictions mentioned above. 
\begin{proposition}\label{closed:thm3}
Let $ (M, g) $ be a closed manifold with $ n = \dim M \geqslant 3 $. Let $ P $ be any fixed point of $ M $. Assume that $ \eta_{1} $, the first eigenvalue of $ \Box_{g} $, is positive. Then there exists some smooth function $ H \in \calC^{\infty}(M) $, which is negative at $ P $, such that $ H $ is realized as the scalar curvature function with respect to some conformal change of the original metric $ g $.
\end{proposition}
\begin{proof} Pick up a normal ball $ B_{r}(P) $ with small enough radius $ r $ centered at $ P $. Choose a function $ f $ with
\begin{equation*}
f(P) = -C < -1, f \leqslant 0, \lvert f \rvert \leqslant C, f \in \calC^{\infty}(M), f \equiv 0 \; {\rm on} \; M \backslash B_{r}(P).
\end{equation*}
The rest of the proof is essentially the same as in Theorem 4.5 of \cite{XU3}, we omit the details here.
\end{proof}
\medskip

The most general case for prescribed scalar curvature so far is as follows:
\begin{corollary}\label{closed:cor1}
Let $ (M, g) $ be a closed Riemannian manifold that is not locally conformally flat. Let $ \eta_{1} $, the first eigenvalue of $ \Box_{g} $, be positive. Let $ S \in \mathcal{A}_{1} $ be a smooth function satisfies $ S = \lambda > 0 $ for some positive constant $ \lambda $ on some open subset $ O \subset M $. Then there exists a positive function $ u \in \calC^{\infty}(M) $ that solves
\begin{equation*}
-a\Delta_{g} u + R_{g} u = Su^{p-1} \; {\rm on} \; M.
\end{equation*}
Equivalently, $ S $ is the prescribed scalar curvature function with respect to the metric $ \tilde{g} = u^{p-2} g $.
\end{corollary}
\begin{proof} Pick up a point $ P \in O $ at which the Weyl tensor does not vanish. By Proposition \ref{closed:thm3}, there exists a smooth function $ v > 0 $ such that the metric $ \bar{g} = v^{p-2} g $ admits a scalar curvature $ S_{\bar{g}} $ with $ S_{\bar{g}}(P) < 0 $. Since the Weyl tensor is invariant under conformal change, the Weyl tensor with respect to $ \bar{g} $ does not vanish also. We now consider the Riemannian manifold $ (M, \bar{g}) $. Choose a small enough neighborhood of $ P $, say $ \Omega $, such that the results of Lemma \ref{closed:lemma1} and \ref{closed:lemma2} hold with respect to $ \Box_{\bar{g}} $ and $ S_{\bar{g}} $. By the same argument as in Theorem \ref{closed:thm2}, there exists a positive function $ w \in \calC^{\infty}(M) $ such that
\begin{equation*}
-a\Delta_{\bar{g}} w + S_{\bar{g}} w = S w^{p-1}
\end{equation*}
holds on $ M $. Denote
\begin{equation*}
\tilde{g} = w^{p-2} \bar{g} = (vw)^{p-2} g : = u^{p-2} g.
\end{equation*}
The positive, smooth function $ u $ and the metric $ \tilde{g} $ are as desired.
\end{proof}
\begin{remark}\label{closed:re2}
Note that in Corollary \ref{closed:cor1}, we still restricts the choice of the prescribed scalar curvature function $ S $ by restricting the region $ O $ covers at least one point in $ M $ at which the Weyl tensor does not vanish. However, we have no restriction on the scalar curvature $ R_{g} $.
\end{remark}
\medskip

Corollary \ref{closed:cor1} gives us a sufficient condition of Kazdan-Warner problem. We try to find some necessary conditions. Note that when $ \eta_{1} > 0 $, one requirement of prescribed scalar curvature $ S $ is that $ S > 0 $ somewhere, see \cite{XU3}. Denote
\begin{equation}\label{closed:eqn19}
\begin{split}
\mathcal{B} & : = \lbrace f \in \calC^{\infty}(M) : f > 0 \; \text{somewhere in M}, \\
& -a\Delta_{g} u + R_{g} u = f u^{p-1} \; \text{admits a real, positive, smooth solution $ u $ with $ \eta_{1} > 0 $}. \rbrace
\end{split}
\end{equation}
The set $ \mathcal{B} $ is the collection of functions that are positive somewhere and are prescribed scalar curvatures of some metrics under conformal change. Clearly
\begin{equation*}
\mathcal{A}_{1} \subset \mathcal{B}.
\end{equation*}
Furthermore, we can conclude as follows.
\begin{theorem}\label{closed:thm4}
Let $ (M, g) $ be a closed manifold that is not locally conformally flat, $ n = \dim M \geqslant 3 $. Then the set $ \mathcal{A}_{1} $ is $ \calL^{r} $-dense in the set $ \mathcal{B} $, for all $ r \in [1, \infty) $.
\end{theorem}
\begin{proof}
Pick up any function $ f \in \mathcal{B} $. If $ f $ is a globally positive constant function, we are done as $ f \in \mathcal{A}_{1} $. Assume from now on that $ f $ is not a constant function. Pick up a small enough open geodesic ball $ O \subset M $ with radius $ \delta \ll 1 $ such that the Weyl tensor does not vanish at the center of the ball. The concentric geodesic ball $ O' $ with radius $ \frac{\delta}{2} $ is contained in $ O $. Take any constant $ \lambda > 0 $. Define the function
\begin{equation*}
f^{*} = \begin{cases} \lambda, & \; {\rm in} \; O' \\ f, & \; {\rm outside} \; O' \end{cases}.
\end{equation*}
We then choose $ 0 < \epsilon \ll \delta $ and mollify $ f^{*} $ with the standard mollifier $ \phi_{\epsilon} $, it follows that
\begin{equation*}
S : = f^{*} * \phi_{\epsilon}
\end{equation*}
can be realized as a prescribed scalar curvature function for some conformal change of $ g $, due to Corollary \ref{closed:cor1}. It is immediate that for each $ r \in [1, \infty) $, we can choose $ \delta $ to be small enough such that
\begin{equation*}
\lVert S - f \rVert_{\calL^{r}(M, g)} 
\end{equation*}
can be arbitrarily small, as desired.
\end{proof}
\begin{remark}\label{closed:re3}
In Theorem \ref{closed:thm4}, we obtain the $ \calL^{r} $-closeness between two sets $ \mathcal{A}_{1} $ and $ \mathcal{B} $. We want a result of $ \calC^{0} $-closeness but it is impossible due to the restriction that the open subset $ O $ required in $ \mathcal{A}_{1} $ must cover at least one point at which the Weyl tensor does not vanish. We will see in later section that the restriction of the choice of the open set $ O $ can be released in most cases. Hence we can enlarge the set $ \mathcal{A}_{1} $ and eventually obtain the $ \calC^{0} $-dense in $ \mathcal{B} $, except very few special manifolds.
\end{remark}
\medskip

Analogous to Kazdan and Warner's ``Trichotomy Theorem", we hope that there is no restriction for the choice of the prescribed scalar curvature function. Unfortunately, we require the pointwise conformal change and thus it is impossible. However, for general smooth function $ f \in \calC^{\infty}(M) $, we have the following density result.
\begin{corollary}\label{closed:cor3}
Let $ (M, g) $ be a closed Riemannian manifold that is not locally conformally flat. Let $ \eta_{1} $, the first eigenvalue of $ \Box_{g} $, be positive. For any given function $ S' \in \calC^{\infty}(M) $, the following statements hold:

(i) Let $ O \subset M $ be a connected, open subset of $ M $. Assume that the Weyl tensor does not vanish at least at one point in $ O $. Then there exists a function $ S_{1} \in \mathcal{A}_{1} $ as the prescribed scalar curvature of some metric $ \tilde{g} $ under conformal change such that $ S_{1} = S' $ outside $ O $;

(ii) For any $ \epsilon > 0 $, there exists a function $ S_{2} \in \mathcal{A}_{1} $ as the prescribed scalar curvature function of some metric $ \tilde{g} $ under conformal change such that $ \lVert S_{2} - S' \rVert_{\calL^{r}(M, g)} < \epsilon, \forall 1 \leqslant r < \infty $.
\end{corollary}
\begin{proof} It is trivial if $ S' $ is a globally positive constant function. Assume not.

For (i), we may assume, without loss of generality, that $ R_{g} < 0 $ on some subset of $ O $, due to Proposition \ref{closed:thm3}. We pick a proper subset $ O' \subset O $ with $ \text{dist}(\partial O', \partial O) > 4\gamma $ and $ \text{diam}(O') \ll 4\gamma $. Denote
\begin{equation*}
S_{1}' = \begin{cases} \lambda, & {\rm in} O' \\ S', & {\rm in} \; M \backslash \bar{O} \end{cases}
\end{equation*}
Mollifying $ S_{1}' $ by $ \phi_{\delta} $ with $ \delta \ll \gamma $, it follows that
\begin{equation*}
S_{1} = S_{1}' * \phi_{\delta}
\end{equation*}
satisfies $ S_{1} \in \mathcal{A}_{1} $. By Theorem \ref{closed:thm2}, $ S_{1} $ is a prescribed scalar curvature with some metric $ \tilde{g} $ under conformal change. Furthermore, $ S_{1} = S' $ outside $ O $.

For (ii), choose any $ \epsilon > 0 $ and $ r \in [1, \infty) $. We pick up a positive constant $ c > 0 $. We can choose a small enough connected, open subset $ O \subset M $ in which the Weyl tensor does not vanish identically and define
\begin{equation*}
S_{2}' = \begin{cases} c, & {\rm in} \; O \\ S', & {\rm in} \; M \backslash O \end{cases}
\end{equation*}
The function
\begin{equation*}
S_{2} = S_{2}' * \phi_{\delta} \in \calC^{\infty}(M)
\end{equation*}
with small enough $ \delta $ satisfies
\begin{equation*}
S_{2} \in \mathcal{A}_{1}, \lVert S_{2} - S' \rVert_{\calL^{r}(M, g)} < \epsilon.
\end{equation*}
As desired.
\end{proof}
\begin{remark}\label{closed:re4}
The result in Corollary \ref{closed:cor3} indicates that every function is ``almost" a prescribed scalar curvature function under pointwise conformal change, in the sense of $ \calL^{r} $-closeness for all $ r \in [1, \infty) $, since we can take $ \epsilon $ to be arbitrarily small by shrinking $ O $ if necessary.
\end{remark}
\medskip

\section{The Global Analysis on Compact Manifolds with Smooth Boundary, Not Locally Conformally Flat Case}
In this section, we discuss the prescribed scalar curvature problem on compact manifolds $ (\bar{M}, g) $ with non-empty smooth boundary $ \partial M $. Precisely speaking, for a given smooth function $ S \in \calC^{\infty}(\bar{M}) $, we look for some metric $ \tilde{g} = u^{p-2} g $ with positive, smooth function $ u $ on $ \bar{M} $ such that $ S $ is the scalar curvature of $ \tilde{g} $; in addition, the boundary is minimal with respect to $ \tilde{g} $. Equivalently, we seek for a positive solution $ u \in \calC^{\infty}(\bar{M}) $ such that
\begin{equation}\label{compact:eqn1}
\Box_{g}u : = -a\Delta_{g} u + R_{g} u = Su^{p-1} \; {\rm in} \; M, B_{g} u : = \frac{\partial u}{\partial \nu} + \frac{2}{p-2} h_{g} u = 0 \; {\rm on} \; \partial M.
\end{equation} 
Here $ R_{g} $ and $ h_{g} $ are scalar and mean curvature with respect to $ g $, respectively. We assume that $ \eta_{1}' $, the first eigenvalue of conformal Laplacian $ \Box_{g} $ with homogeneous Robin condition $ B_{g} = 0 $, is positive. The case when the first eigenvalue is negative is fairly easy, see e.g. \cite{CHS}. The local result in Proposition \ref{local:prop4} plays a central role here. Throughout this section, we denote $ (\bar{M}, g) $ be a compact Riemannian manifold with $ \dim(\bar{M}) \geqslant 3 $ and non-empty smooth boundary $ \partial M $. We denote $ M $ to be the interior of $ \bar{M} $. In this section, we also require that the Weyl tensor does not vanish identically in the interior $ M $ with respect to $ g $.
\medskip

Analogous to the set $ \mathcal{A}_{1} $ for closed manifolds, we define the following set
\begin{equation}\label{compact:eqn2}
\begin{split}
& \mathcal{A}_{1}' : = \lbrace S \in \calC^{\infty}(\bar{M}) : S \equiv c > 0 \; {\rm on} \; \bar{O}, \text{$ c > 0 $ is an arbitrary constant, } \\
& \qquad \text{ $ O \subset \bar{O} \subset M $ is an arbitrary open interior submanifold with smooth $ \partial O $ such that} \\
& \qquad \text{the metric $ g $ is not locally conformally flat within $ O $.} \rbrace.
\end{split}
\end{equation}
We show that (\ref{compact:eqn1}) admits a smooth, positive solution when $ S \in \mathcal{A}_{1}' $, provided that $ (\bar{M}, g) $ is not locally conformally flat.

We introduce a monotone iteration scheme with boundary conditions, similar to the method on closed manifolds. We need to restrict $ h_{g} > 0 $ everywhere on $ \partial M $.
\begin{theorem}\label{compact:thm1}\cite[Thm.~4.1]{XU5}
Let $ (\bar{M}, g) $ be a compact manifold with smooth boundary $ \partial M $. Let $ \nu $ be the unit outward normal vector along $ \partial M $ and $ q > \dim \bar{M} $. Let $ h_{g} > 0 $ everywhere on $ \partial M $. Let $ S \in \calC^{\infty}(\bar{M}) $ be a given function. Suppose that there exist $ u_{-} \in \calC^{0}(\bar{M}) \cap H^{1}(M, g) $ and $ u_{+} \in W^{2, q}(M, g) \cap \calC^{0}(\bar{M}) $, $ 0 \leqslant u_{-} \leqslant u_{+} $, $ u_{-} \not\equiv 0 $ on $ \bar{M} $ such that
\begin{equation}\label{compact:eqn3}
\begin{split}
-a\Delta_{g} u_{-} + R_{g} u_{-} - S u_{-}^{p-1} & \leqslant 0 \; {\rm in} \; M, \frac{\partial u_{-}}{\partial \nu} + \frac{2}{p-2} h_{g} u_{-} \leqslant 0 \; {\rm on} \; \partial M \\
-a\Delta_{g} u_{+} + R_{g} u_{+} - S u_{+}^{p-1} & \geqslant 0 \; {\rm in} \; M, \frac{\partial u_{+}}{\partial \nu} + \frac{2}{p-2} h_{g} u_{+} \geqslant 0 \; {\rm on} \; \partial M
\end{split}
\end{equation}
holds weakly. Then there exists a real, positive solution $ u \in \calC^{\infty}(M) \cap \calC^{1, \alpha}(\bar{M}) $ of
\begin{equation}\label{compact:eqn4}
\Box_{g} u = -a\Delta_{g} u + R_{g} u = S u^{p-1}  \; {\rm in} \; M, B_{g} u =  \frac{\partial u}{\partial \nu} + \frac{2}{p-2} h_{g} u = 0 \; {\rm on} \; \partial M.
\end{equation}
\end{theorem}
\begin{proof} The proof is almost the same as in Theorem 4.1 of \cite{XU5}, the only difference is set $ \zeta = 0 $ and replace $ \lambda $ by $ S $ in \cite[Thm.~4.1]{XU5}. We omit the details here. We would like to mention that at here and all arguments below, $ \frac{\partial u_{-}}{\partial \nu} $ is well-defined on $ \partial M $.
\end{proof}
\medskip

Mimicking the results in Lemma \ref{closed:lemma1} and \ref{closed:lemma2}, we construct the sub-solution and super-solution of (\ref{compact:eqn1}). Given the function $ S \in \mathcal{A}_{1}' $ where $ S = \lambda > 0 $ on $ O \subset M \subset \bar{M} $. We impose the restrictions that $ h_{g} > 0 $ everywhere on $ \partial M $ and $ R_{g} < 0 $ somewhere in $ O $ first.
\begin{lemma}\label{compact:lemma1}
Let $ (\bar{M}, g) $ be a compact manifold, not locally conformally flat, with smooth boundary $ \partial M $. Let $ \nu $ be the outward unit normal vector along $ \partial M $. Let $ S \in \mathcal{A}_{1}' $ be a smooth function such that $ S = \lambda > 0 $ on some inner, open subset $ O \subset M \subset \bar{M} $ for some constant $ \lambda $. Assume further that $ h_{g} > 0 $ everywhere on $ \partial M $ and $ R_{g} < 0 $ somewhere in $ O $. Then there exists a function $ u_{-} \in H^{1}(M, g) \cap \calC^{0}(\bar{M}) $ such that
\begin{equation}\label{compact:eqn5}
 -a\Delta_{g} u_{-} + R_{g} u_{-} \leqslant Su_{-}^{p-1} \; {\rm in} \; M,  \frac{\partial u_{-}}{\partial \nu} + \frac{2}{p-2} h_{g} u_{-} \leqslant 0 \; {\rm on} \; \partial M
 \end{equation}
 in the weak sense. In addition, $ u_{-} \geqslant 0 $ and $ u_{-} \not\equiv 0 $.
\end{lemma}
\begin{proof} Without loss of generality, we may assume that $ O \subset M $ is connected since otherwise we can work on some open subset of any connected component of $ O $ such that $ R_{g} < 0 $ somewhere within it. We choose a small enough open, proper, connected subset $ \Omega \subset O \subset M \subset \bar{M} $ in which the hypothesis of Proposition \ref{local:prop4} are satisfied. In particular, $ R_{g} < 0 $ everywhere on $ \Omega $. Furthermore, $ S = \lambda > 0 $ in $ \Omega $. It follows from the same argument as in Lemma \ref{closed:lemma1} that a smooth, positive function $ u_{2} $ solves
\begin{equation}\label{compact:eqn6}
-a\Delta_{g} u_{2} + R_{g} u_{2} = \lambda u_{2}^{p-1} \; {\rm in} \;O, u_{2} \equiv 0 \; {\rm on} \; \partial O.
\end{equation}
It follows that a function $ u_{-} \in H^{1}(M, g) \cap \calC^{0}(\bar{M}) $ with the expression
\begin{equation}\label{compact:eqn7}
u_{-} : = \begin{cases} u_{2} & \; {\rm in} \; \Omega \\ 0 & \; {\rm in} \bar{M} \backslash \bar{O} \end{cases}
\end{equation}
satisfies
\begin{equation*}
-a\Delta_{g} u_{-} + R_{g} u_{-} \leqslant Su_{-}^{p-1} \; {\rm in} \; M
\end{equation*}
weakly. Since $ u_{-} \equiv 0 $ on $ \partial M $, the boundary condition holds trivially. Lastly we see that $ u_{-} \geqslant 0 $ with $ u_{-} \not\equiv 0 $.
\end{proof}
\medskip

The construction of super-solution is quite similar to Lemma \ref{closed:lemma2} since the only hard part in analysis is within the same set $ \Omega $ as in Lemma \ref{compact:lemma1}. Analogous to the previous section for closed manifolds, we work on the interesting case $ \eta_{1}' > 0 $. Recall that $ \eta_{1}' $ is the first eigenvalue of conformal Laplacian with homogeneous Robin condition. It is known from \cite{ESC} that the first eigenvalue $ \eta_{1}' $ admits a positive, smooth eigenfunction.
\begin{lemma}\label{compact:lemma2}
Let $ (\bar{M}, g) $ be a compact manifold with smooth boundary $ \partial M $. Let $ \nu $ be the outward unit normal vector along $ \partial M $. Suppose $ \eta_{1}' > 0 $ be the first eigenvalue of $ \Box_{g} $ with respect to $ B_{g} \varphi = 0 $, where $ \varphi > 0 $ is the associated eigenfunction. Let $ S \in \mathcal{A}_{1}' $ be a  non-constant smooth function such that $ S = \lambda > 0 $ on some inner, open subset $ O \subset M \subset \bar{M} $ for some constant $ \lambda $. Assume that the metric $ g $ is not locally conformally flat within the region $ \lbrace x \in O | R_{g}(x) < 0 \rbrace $. Assume further that $ h_{g} > 0 $ everywhere on $ \partial M $ and $ R_{g} < 0 $ somewhere in $ O $. Then there exists a function $ u_{+} \in \calC^{\infty}(\bar{M}) $ such that
\begin{equation}\label{compact:eqn8}
 -a\Delta_{g} u_{+} + R_{g} u_{+} \geqslant Su_{+}^{p-1} \; {\rm in} \; M,  \frac{\partial u_{+}}{\partial \nu} + \frac{2}{p-2} h_{g} u_{+} \geqslant 0 \; {\rm on} \; \partial M
 \end{equation}
 in the usual sense. In addition, $ u_{+} \geqslant u_{-} \geqslant 0 $ pointwise.
\end{lemma}
\begin{proof}
Without loss of generality, we may assume that $ O $ is connected since otherwise we can work on some open subset of any connected component of $ O $ such that $ R_{g} < 0 $ somewhere within it. It is known from \cite{ESC} that there exists a positive function $ \varphi \in \calC^{\infty}(\bar{M}) $ such that
\begin{equation}\label{compact:eqn9}
-a\Delta_{g} \varphi + R_{g} \varphi = \eta_{1}' \varphi \; {\rm in} \; M, \frac{\partial \varphi}{\partial \nu} + \frac{2}{p-2} h_{g} \varphi = 0 \; {\rm on} \; \partial M.
\end{equation}
Since any positive scale of $ \varphi $ solves (\ref{compact:eqn9}), we apply the same argument as in Lemma \ref{closed:lemma2}, and conclude that there exists a function $ \phi : = \delta \varphi $ for some small enough $ \delta > 0 $ such that
\begin{equation}\label{compact:eqn10}
-a\Delta_{g} \phi + R_{g} \phi \geqslant 2^{p-2} \max_{\bar{M}} S \phi^{p-1} \geqslant S \phi^{p-1} \; {\rm in} \; M, \frac{\partial \phi}{\partial \nu} + \frac{2}{p-2} h_{g} \phi \geqslant 0 \; {\rm on} \; \partial M.
\end{equation}
In particular we choose an open set $ \Omega \subset O $ which satisfies the requirements in Lemma \ref{compact:lemma1}; in addition, within $ \Omega $ the metric is not locally conformally flat. Within $ \Omega $, (\ref{compact:eqn10}) is of the expression
\begin{equation}\label{compact:eqn11}
-a\Delta_{g} \phi + R_{g} \phi \geqslant 2^{p-2} \lambda \phi^{p-1} > \lambda \phi^{p-1} \; {\rm in} \; \Omega \; \text{pointwise}.
\end{equation}
With the same choices of constants, sets, and partitions of unity in Lemma \ref{closed:lemma2} and \cite[Thm.~4.3]{XU3}, we conclude that there exists a function $ f \in \calC^{\infty}(\Omega) $ such that
\begin{equation}\label{compact:eqn12}
-a\Delta_{g} f + R_{g} f \geqslant \lambda f^{p-1} \; {\rm in} \; \Omega, f \geqslant u_{2} > 0 \; {\rm in} \; \Omega, f \equiv \phi \; {\rm on} \; \partial \Omega.
\end{equation}
Here we apply the local results in (\ref{compact:eqn6}) as well as the inequality (\ref{compact:eqn11}). Note that $ f \equiv \phi $ near $ \partial \Omega $. Hence we define
\begin{equation}\label{compact:eqn13}
u_{+} : = \begin{cases} f, & {\rm in} \; \Omega \\ \phi, & {\rm in} \; \bar{M} \backslash \Omega \end{cases}.
\end{equation}
Immediately $ u_{+} \in \calC^{\infty}(\bar{M}) $. Since $ u_{+} = \phi $ on $ \partial M $, it follows that
\begin{equation*}
-a\Delta_{g} u_{+} + R_{g} u_{+} \geqslant S u_{+}^{p-1} \; {\rm in} \; M, \frac{\partial u_{+}}{\partial \nu} + \frac{2}{p-2} h_{g} u_{+} \geqslant 0 \; {\rm on} \; \partial M \; {\rm pointwise}.
\end{equation*}
In addition, we conclude from (\ref{compact:eqn12}), (\ref{compact:eqn13}) and the construction of $ u_{-} $ in (\ref{compact:eqn7}) that
\begin{equation*}
u_{+} \geqslant u_{-} \geqslant 0 \: {\rm on} \; \bar{M}.
\end{equation*}
\end{proof}
\medskip

We can now prove the first result with respect to prescribed scalar curvature on $ (\bar{M}, g) $ with positive Yamabe invariant, under the restrictions of Lemma \ref{compact:lemma1} and \ref{compact:lemma2}.
\begin{theorem}\label{compact:thm2}
Let $ (\bar{M}, g) $ be a compact manifold, not locally conformally flat, with smooth boundary $ \partial M $. Let $ \nu $ be the outward unit normal vector along $ \partial M $. Suppose $ \eta_{1}' > 0 $ be the first eigenvalue of $ \Box_{g} $ with respect to homogeneous Robin condition in (\ref{compact:eqn1}). Let $ S \in \mathcal{A}_{1}' $ be a smooth function such that $ S = \lambda > 0 $ on some inner, open subset $ O \subset M \subset \bar{M} $ for some constant $ \lambda $. Assume further that $ h_{g} > 0 $ everywhere on $ \partial M $ and $ R_{g} < 0 $ somewhere in $ O $. Assume further that $ h_{g} > 0 $ everywhere on $ \partial M $ and $ R_{g} < 0 $ somewhere in $ O $. Then there exists a conformal change $ \tilde{g} = u^{p-2} g $ such that the metric $ \tilde{g} $ admits the prescribed scalar curvature $ S $, where the smooth function $ u > 0 $ on $ \bar{M} $ satisfies
\begin{equation}\label{compact:eqn14}
-a\Delta_{g} u + R_{g} u = S u^{p-1} \; {\rm in} \; M, \frac{\partial u}{\partial \nu} + \frac{2}{p-2} h_{g} u = 0 \; {\rm on} \; \partial M.
\end{equation}
\end{theorem}
\begin{proof} The problem is reduced to the Escobar problem if $ S $ is a globally positive constant function. Assume not. According to Lemma \ref{compact:lemma1} and \ref{compact:lemma2}, the hypothesis in Theorem \ref{compact:thm1} and inequalities in (\ref{compact:eqn3}) hold with the choice of $ S $ here. Applying Theorem \ref{compact:thm1} directly, it follows that there exists some $ u \in \calC^{\infty}(\bar{M}) $ with $ u_{-} \leqslant u \leqslant u_{+} $ that solves (\ref{compact:eqn14}). Due to the same maximum principle argument as in \cite{XU4}, it follows that $ u > 0 $ on $ \bar{M} $. The regularity of $ u $ can also be achieved due to \cite{CHE}.
\end{proof}
\medskip

We now turn to the general cases, i.e. removing the restrictions that (i) $ h_{g} > 0 $ everywhere on $ \partial M $; (ii) $ R_{g} < 0 $ within a specific region; (iii) the Weyl tensor of metric $ g $ does not vanish identically within $ \lbrace x \in O | R_{g} < 0 \rbrace $. We apply the following two results from \cite{XU4} to release the three restrictions above.
\begin{proposition}\label{compact:prop1}\cite[Thm.~5.3]{XU4}
Let $ (\bar{M}, g) $ be a compact manifold with smooth boundary $ \partial M $. There exists a Yamabe metric $ \tilde{g} \in [g] $ with mean curvature $ h_{\tilde{g}} > 0 $ everywhere on $ \partial M $.
\end{proposition}
\begin{proposition}\label{compact:prop2}\cite[Thm.~5.7]{XU4}\cite[Thm.~4.5]{XU3}
Let $ (\bar{M}, g) $ be a compact manifold with smooth boundary $ \partial M $. Let $ P $ be a fixed point in interior $ M $. There exists a Yamabe metric $ \tilde{g} \in [g] $ with scalar curvature $ R_{\tilde{g}} $ such that $ R_{\tilde{g}}(P) < 0 $. In addition, the sign of $ h_{g} $ and the sign of $ h_{\tilde{g}} $ are the same pointwise on $ \partial M $.
\end{proposition}
\begin{proof}
We construct a function $ F \in \calC^{\infty}(\bar{M}) $ such that (i) $ \int_{M} F \dvol = 0 $; (ii) $ F $ is negative at the point $ P $; (iii) $ \lVert F \rVert_{H^{s - 2}}(M, g) $ is small enough, with $ s= \frac{n}{2} + 1 $ when $ n $ is even, and $ s = \frac{n + 1}{2} $ when $ n $ is odd. This construction is almost the same as in \cite[Thm.~5.7]{XU4}, just note that we have
\begin{equation}\label{compact:eqn15}
\text{Vol}_{g}(B_{r}(P)) = \left(1 - \frac{R_{g}(P)}{n + 2} r^{2} + O(r^{4}) \right) \text{Vol}_{e}(B_{r}(P))
\end{equation}
and thus
\begin{equation*}
\text{Vol}_{g}(B_{r}(P)) = O(r^{n})
\end{equation*}
when $ r \ll 1 $. Here $ \text{Vol}_{g}(B_{r}(P)) $ is the volume of the geodesic ball $ B_{r}(P) $ with radius $ r $, and $ \text{Vol}_{e}(B_{r}(P)) $ is the volume of the Euclidean ball with same radius and center. For more details, we refer to \cite[Thm.~4.5]{XU3}. The rest of the proof is exactly the same as in \cite[Thm.~5.7]{XU4}.
\end{proof}
\begin{remark}\label{compact:re1}
Note that the formula (\ref{compact:eqn15}) is the essence of the curvature, which tells us the deviation of the volume of a geodesic ball from the associated Euclidean ball. The Weyl tensor cannot provide this type of information. As we mentioned in the introduction, the Weyl tensor is only used to distinguish the locally conformally flat case, for which we apply a different method to obtain the local solution. The essential difficulty is due to the positivity of the first eigenvalue of the conformal Laplacian, which, roughly speaking, shrinks the volumes of the manifold and hence lower the energy in terms of the Yamabe quotient.
\end{remark}
\medskip

Given a function $ S \in \mathcal{A}_{1}' $ where $ S = \lambda > 0 $ in $ O $. In Theorem \ref{compact:thm2}, we require (i) $ R_{g} < 0 $ somewhere in $ O $; (ii) the Weyl tensor with respect to $ g $ does not vanish at least at one point at which $ R_{g} $ is negative. Fix a point $ P \in O $ at which the Weyl tensor does not vanish, due to Proposition \ref{compact:prop2}, any metric is conformal to a metric with scalar curvature that is negative at $ P $ and hence within a neighborhood of $ P $. In addition, the Weyl tensor does not vanish after conformal change due to conformal invariance. Therefore we can remove the restriction of the positivity of the scalar curvature somewhere in $ O $.
\begin{corollary}\label{compact:cor1}
Let $ (\bar{M}, g) $ be a compact manifold, not locally conformally flat, with smooth boundary $ \partial M $. Let $ \nu $ be the outward unit normal vector along $ \partial M $. Suppose $ \eta_{1}' > 0 $ be the first eigenvalue of $ \Box_{g} $ with respect to homogeneous Robin condition in (\ref{compact:eqn1}). Let $ S \in \mathcal{A}_{1}' $ be a smooth function such that $ S = \lambda > 0 $ on some inner, open subset $ O \subset M \subset \bar{M} $ for some constant $ \lambda $. Assume further that $ h_{g} > 0 $ everywhere along $ \partial M $. Then there exists a conformal change $ \tilde{g} = u^{p-2} g $ such that the metric $ \tilde{g} $ admits the prescribed scalar curvature $ S $, where the smooth function $ u > 0 $ on $ \bar{M} $ satisfies
\begin{equation}\label{compact:eqn16}
-a\Delta_{g} u + R_{g} u = S u^{p-1} \; {\rm in} \; M, \frac{\partial u}{\partial \nu} + \frac{2}{p-2} h_{g} u = 0 \; {\rm on} \; \partial M.
\end{equation}
\end{corollary}
\begin{proof} Fix the function $ S \in \mathcal{A}_{1}' $ for which is a positive constant within some subset $ O \subset M $. Pick up a point $ P \in O \subset M $. Starting at the metric $ g $, there exists a positive function $ v \in \calC^{\infty}(\bar{M}) $ such that the metric $ g_{1} = v^{p-2} g $ admits the scalar curvature $ R_{g_{1}} < 0 $ on $ P $. We may assume, without loss of generality, that the Weyl tensor of $ g $ at $ P $ does not vanish. Since if not, we can apply Proposition \ref{compact:prop2} again to make $ R_{g_{1}} $ at $ P $. Note that the sign of $ h_{g_{1}} $ is the same as the sign of $ h_{g} $, pointwise. 

Choosing a small enough neighborhood $ \Omega $ around $ P $ with $ \Omega \subset O $, we apply the result of Theorem \ref{compact:thm2}, it follows that there exists a positive function $ w \in \calC^{\infty}(\bar{M}) $ such that
\begin{equation*}
-a\Delta_{g_{1}} w + R_{g_{1}} w = S w^{p-1} \; {\rm in} \; M, \frac{\partial w}{\partial \nu} + \frac{2}{p-2} h_{g_{1}} w = 0 \: {\rm on} \; \partial M.
\end{equation*}
Hence the scalar curvature of the metric $ \tilde{g} = w^{p-2} g_{1} $ is exactly the given function $ S $. It follows that
\begin{equation*}
\tilde{g} = w^{p-2} g_{1} = w^{p-2} v^{p-2} g.
\end{equation*}
Denote $ u = wv $ pointwise on $ \bar{M} $, we conclude that $ \tilde{g} = u^{p-2} g $ admits the prescribed scalar curvature $ S \in \mathcal{A}_{1}' $. Equivalently, (\ref{compact:eqn16}) holds with $ u $.
\end{proof}
\medskip

The iteration scheme introduced in Theorem \ref{compact:thm1} requires the positivity of the mean curvature along $ \partial M $. With the aid of both Proposition \ref{compact:prop1} and \ref{compact:prop2}, we can remove this restriction eventually.
\begin{corollary}\label{compact:cor2}
Let $ (\bar{M}, g) $ be a compact manifold, not locally conformally flat, with smooth boundary $ \partial M $. Let $ \nu $ be the outward unit normal vector along $ \partial M $. Suppose $ \eta_{1}' > 0 $ be the first eigenvalue of $ \Box_{g} $ with respect to homogeneous Robin condition in (\ref{compact:eqn1}). Let $ S \in \mathcal{A}_{1}' $ be a smooth function such that $ S = \lambda > 0 $ on some inner, open subset $ O \subset M \subset \bar{M} $ for some constant $ \lambda $. Then there exists a conformal change $ \tilde{g} = u^{p-2} g $ such that the metric $ \tilde{g} $ admits the prescribed scalar curvature $ S $, where the smooth function $ u > 0 $ on $ \bar{M} $ satisfies
\begin{equation}\label{compact:eqn17}
-a\Delta_{g} u + R_{g} u = S u^{p-1} \; {\rm in} \; M, \frac{\partial u}{\partial \nu} + \frac{2}{p-2} h_{g} u = 0 \; {\rm on} \; \partial M.
\end{equation}
\end{corollary}
\begin{proof} Given the function $ S \in \mathcal{A}_{1}' $ with $ S = \lambda > 0 $ in $ O \subset M $. If $ h_{g} > 0 $ everywhere on $ \partial M $ already, the problem is reduced to the situation of Corollary \ref{compact:cor1}.

If not, we apply Proposition \ref{compact:prop1}, which follows that there exists a metric $ g_{1} = v_{1}^{p-2} g $ such that $ h_{g_{1}} > 0 $ everywhere on $ \partial M $. Here $ v_{1} \in \calC^{\infty}(\bar{M}) $ is positive on $ \bar{M} $. If $ R_{g_{1}} < 0 $ somewhere on $ O $, and the metric $ g_{1} $ is not locally conformally flat within $ \lbrace x \in O | R_{g_{1}}(x) < 0 \rbrace $, then we apply Theorem \ref{compact:thm2} and conclude that the metric $ \tilde{g} = v_{2}^{p-2} g_{1} $ admits the prescribed scalar curvature $ S $ with some positive function $ v_{2} \in \calC^{\infty}(\bar{M}) $. The choice of $ u : = v_{1} v_{2} $ with $ \tilde{g} = u^{p-2} g $.

If either requirement of $ R_{g_{1}} $ above does not satisfy, we apply Proposition \ref{compact:prop2}. Hence there exists some positive function $ v_{3} \in \calC^{\infty}(M) $ such that $ g_{2} = v_{3}^{p-2} g_{1} $ admits a scalar curvature $ R_{g_{2}} > 0 $ somewhere in $ O $; in addition, $ h_{g_{2}} > 0 $ everywhere on $ \partial M $. We then apply Theorem \ref{compact:thm2} again.
\end{proof}
\medskip

In the second half of \S3, we discussed whether we can approximate a prescribed scalar curvature of some metric under conformal change by some function in $ \mathcal{A}_{1} $ in the $ \calL^{r} $-sense. Analogous results hold for prescribed scalar curvature on compact manifolds with boundary. Recall that we still discuss the case when $ \eta_{1}' $, the first eigenvalue of conformal Laplacian with homogeneous Robin condition, is positive; the compact manifold $ (\bar{M}, g) $ has smooth boundary $ \partial M $, with dimension $ \dim \bar{M} \geqslant 3 $; the outward unit normal vector along $ \partial M $ is denoted to be $ \nu $.

Corollary \ref{compact:cor2} is a sufficient condition of Kazdan-Warner problem on $ (\bar{M}, g) $ with respect to the positive first eigenvalue of conformal Laplacian. Conversely, if $ S $ is the prescribed scalar curvature of some metric $ \tilde{g} $ after conformal change with minimal boundary, i.e. $ R_{\tilde{g}} = S $ on $ M $ and $ h_{\tilde{g}} = 0 $, we consider some necessary conditions of Kazdan-Warner problem. On closed manifold case, Kazdan and Warner \cite{KW} stated that any prescribed scalar curvature with $ \eta_{1} > 0 $ must be positive somewhere. There is a similar obstruction of $ S $ as a prescribed scalar curvature with minimal boundary. \begin{proposition}\label{compact:prop3}
Let $ (\bar{M}, g) $ be a compact manifold with smooth boundary $ \partial M $ associated with outward unit normal vector $ \nu $ along the boundary. Let $ S $ be a smooth function such that
\begin{equation}\label{compact:eqn18}
-a\Delta_{g} u + R_{g} u = S u^{p-1} \; {\rm in} \; M, \frac{\partial u}{\partial \nu} + \frac{2}{p-2} h_{g} u = 0 \; {\rm on} \; \partial M
\end{equation}
admits a positive, smooth solution $ u $ on $ \bar{M} $. Assume that $ \eta_{1}' > 0 $, then the function $ S $ must be positive somewhere.
\end{proposition}
\begin{proof}
Note from \cite[Prop.~2.1]{XU4} that the sign of first eigenvalue of conformal Laplacian with homogeneous Robin condition is a conformal invariant, thus the first eigenvalue $ \eta_{2}' $ with respect to the operator $ \Box_{\tilde{g}} $ is also positive, where $ \tilde{g} = u^{p-2} g $ with $ u $ a positive, smooth solution of (\ref{compact:eqn18}). Note that $ h_{\tilde{g}} = 0 $, thus $ \eta_{2}' $ is characterized by
\begin{equation*}
0 < \eta_{2}' = \inf_{\phi \geqslant 0} \frac{\int_{M} a\lvert \nabla_{\tilde{g}} \phi^{2} \rvert d\text{Vol}_{\tilde{g}} + \int_{M} S \phi^{2} d\text{Vol}_{\tilde{g}}}{\int_{M} u^{2}  d\text{Vol}_{\tilde{g}}}.
\end{equation*}
Note that $ h_{\tilde{g}} = 0 $ thus the boundary term disappears. Taking the test function $ \phi \equiv 1 $, which satisfies the Robin condition
\begin{equation*}
B_{g} \phi = \frac{\partial 1}{\partial \nu} + \frac{2}{p-2} h_{\tilde{g}} \cdot 1 = 0,
\end{equation*}
we conclude that
\begin{equation*}
0 < \eta_{2}' \leqslant C \int_{M} S d\text{Vol}_{\tilde{g}}
\end{equation*}
for some positive constant $ C $. It follows that $ S > 0 $ somewhere.
\end{proof}
Due to the result of Proposition \ref{compact:prop3}, we have to consider the collection of functions that are positive somewhere. Denote
\begin{equation}\label{compact:eqn19}
\begin{split}
\mathcal{B}' & : = \lbrace f \in \calC^{\infty}(\bar{M}) : f > 0 \; \text{somewhere in the interior M}, \\
& -a\Delta_{g} u + R_{g} u = f u^{p-1} \; {\rm in} \; M, \frac{\partial u}{\partial \nu} + \frac{2}{p-2} h_{g} u = 0 \; {\rm on} \; \partial M, \\
& \text{admits a real, positive, smooth solution $ u $ with $ \eta_{1}' > 0 $}. \rbrace
\end{split}
\end{equation}
The set $ \mathcal{B}' $ is the collection of functions that are positive somewhere and are prescribed scalar curvatures of some metrics under conformal change on $ (\bar{M}, g) $. Analogous to Theorem \ref{closed:thm4} and Corollary \ref{closed:cor3}, we conclude the ``almost" necessary conditions of Kazdan-Warner problem as follows.
\begin{theorem}\label{compact:thm3} The set $ \mathcal{A}_{1}' \subset \mathcal{B}' $ is $ \calL^{r} $-dense in the set $ \mathcal{B}' $ for all $ r \in [1, \infty) $.
\end{theorem}
\begin{proof} Clearly $ \mathcal{A}_{1}' \subset \mathcal{B}' $. The rest of the proof is exactly the same as in Theorem \ref{closed:thm4} since we only consider the difference within a small enough inner subset $ \Omega \subset M $.
\end{proof}
\begin{remark}\label{compact:re2}
Again we mentioned here that the choice of open subsets in $ \mathcal{A}_{1}' $ is restricted. We will discuss in later section for how to remove this restriction in most cases and thus we would be able to achieve $ \calC^{0} $-dense for most situations.
\end{remark}
\medskip

\begin{corollary}\label{compact:cor3}
Let $ (\bar{M}, g) $ be a compact Riemannian manifold, not locally conformally flat, with smooth boundary $ \partial M $. Let $ \eta_{1}' $, the first eigenvalue of $ \Box_{g} $ with homogeneous Robin condition, be positive. For any given function $ S' \in \calC^{\infty}(\bar{M}) $, the following statements hold:

(i) For any inner, connected, open subset $ O \subset M $ in which the metric $ g $ is not locally conformally flat, there exists a function $ S_{1} \in \mathcal{A}_{1}' $ as the prescribed scalar curvature of some metric $ \tilde{g} $ under conformal change such that $ S_{1} = S' $ outside $ O $;

(ii) For any $ \epsilon > 0 $, there exists a function $ S_{2} \in \mathcal{A}_{1}' $ as the prescribed scalar curvature of some metric $ \tilde{g} $ under conformal change such that $ \lVert S_{2} - S' \rVert_{\calL^{r}(M, g)} < \epsilon, \forall 1 \leqslant r < \infty $.
\end{corollary}
\begin{proof} Again the proof of these statements are exactly the same as in Corollary \ref{closed:cor3} since we only mollify $ S' $ within an inner open subset of $ \bar{M} $.
\end{proof}
\begin{remark}\label{compact:re3}
The result of Corollary \ref{compact:cor3} states that any smooth function on $ \bar{M} $ is ``almost" a prescribed scalar curvature of some metric under conformal change in the sense of the $ \calL^{r} $-closeness, meanwhile the boundary is minimal after conformal change.
\end{remark}
\medskip

\section{Prescribed Scalar Curvature on $ n $-Sphere and Its Quotient}
In this section, we first discuss the most special case, the $ n $-sphere $ (\mathbb{S}^{n}, g_{\mathbb{S}^{n}} ) $ with $ n \geqslant 3 $. We will first give some results in terms of the functions that can be realized as prescribed scalar curvature functions of some Yamabe metric $ \tilde{g} \in [g_{\mathbb{S}^{n}}] $. We then will discuss the speciality of local solutions of the Yamabe equation on $ n $-sphere, and give some criterion which give some forbidden cases on locally conformally flat manifolds. It turns that all manifolds with this type of analytical criterion is of the form $ \mathbb{S}^{n} \slash \Gamma $ with some (possibly trivial) Kleinian group $ \Gamma $.

We start with some basic facts. On $ n $-sphere, we could introduce the local charts in terms of stereographic projections by removing one of two antipodal points, respectively. Without loss of generality, we discuss the stereographic projections on $ \mathbb{S}^{n} \backslash \lbrace N \rbrace $ and $ \mathbb{S}^{n} \backslash \lbrace S \rbrace $, respectively. Here the north pole is denoted by $ N = (0, \dotso, 0, 1) $ and the south pole is $ S = (0, \dotso, 0, -1) $. Two stereographic projections $ \sigma_{1} $ and $ \sigma_{2} $ can be defined as
\begin{equation}\label{sphere:eqn1}
\begin{split}
\sigma_{1} & : \mathbb{S}^{n} \backslash {N} \rightarrow \mathbb{R}^{n}, (\xi, \tau) = (\xi_{1}, \dotso, \xi_{n}, \tau) \mapsto (x_{1}, \dotso, x_{n}) = x; \\
\sigma_{2} & : \mathbb{S}^{n} \backslash {S} \rightarrow \mathbb{R}^{n}, (\xi, \tau) = (\xi_{1}, \dotso, \xi_{n}, \tau) \mapsto (y_{1}, \dotso, y_{n}) = y.
\end{split}
\end{equation}
Here $ (\xi, \tau) $ is the global coordinate in ambient space $ \R^{n+1} $. The $ \tau $-axis passes through both north and south poles. We can check easily that
\begin{equation*}
\sigma_{1}(\xi, \tau) = \frac{\xi}{1 - \tau}, \sigma_{2}(\xi, \tau) = \frac{\xi}{1 + \tau}.
\end{equation*}
It follows that
\begin{equation*}
\sigma_{1}^{-1}(x) = \left(\frac{2x}{1 + \lvert x \rvert^{2}}, \frac{\lvert x \rvert^{2} - 1}{\lvert x \rvert^{2} + 1}\right), \sigma_{2}^{-1}(y) = \left(\frac{2y}{1 + \lvert y \rvert^{2}}, \frac{1 - \lvert y \rvert^{2}}{\lvert y \rvert^{2} + 1}\right).
\end{equation*}
Since both $ \sigma_{1}^{-1} $ and $ \sigma_{2}^{-1} $ are global parametrizations of $ \mathbb{S}^{n} \backslash {N} $, $ \mathbb{S}^{n} \backslash {S} $, respectively, we can compute that
\begin{equation}\label{sphere:eqn2}
\begin{split}
\left(\sigma_{1}^{-1} \right)^{*} g_{\mathbb{S}^{n}} & = \left( \frac{2}{1 + \lvert x \rvert^{2}} \right)^{2} g_{e} : = \Phi^{p-2} g_{e}; \\
\left( \sigma_{2}^{-1} \right)^{*}  g_{\mathbb{S}^{n}} & = \left( \frac{2}{1 + \lvert y \rvert^{2}} \right)^{2} g_{e} = \Phi^{p-2} g_{e}.
\end{split}
\end{equation}
Note that the local expression of $ g_{\mathbb{S}^{n}} $ after stereographic projections are conformal changes of the Euclidean metric. $ \sigma_{1} $ and $ \sigma_{2} $ are typical examples of conformal diffeomorphisms. We are only interested in chart maps of manifolds that are conformal diffeomorphisms, which is defined as follows:
\begin{definition}\label{sphere:def0}
Let $ (\bar{M}, g) $ be a compact manifold with or without boundary. Let $ (U, \psi) $ be a chart of $ \bar{M} $, we say that $ \psi $ is a conformal diffeomorphism if $ \left(\psi^{-1} \right)^{*} g $ is conformal to the Euclidean metric $ g_{e} $.
\end{definition}
Taking a neighborhood $ \Omega_{0} $ of the south pole $ S $ within $ \mathbb{S}^{n} $. For each $ \epsilon > 0 $, set
\begin{equation*}
\Omega_{0, \epsilon} = \lbrace x \in \Omega_{0} | d(x, N) > \epsilon \rbrace.
\end{equation*}
We choose $ \epsilon $ small enough so that the ball $ B_{\epsilon}(N) \subset \Omega_{0} $. We consider the following Dirichlet problem
\begin{equation}\label{sphere:eqn3}
\begin{split}
-a\Delta_{g_{\mathbb{S}^{n}}} u + n(n  -1) u & = Q u^{p-1} \; {\rm in} \; \Omega_{0, \epsilon}; \\
u & \equiv 0 \; {\rm on} \; \partial \Omega_{0, \epsilon}; \\
u & > 0 \; {\rm in} \; \Omega_{0, \epsilon}.
\end{split}
\end{equation}
Here $ Q \in \calC^{2}(\Omega_{0}) $ and $ \min_{x \in \bar{\Omega}_{0}} Q(x) > 0 $. We also require that the north pole $ N $ is not contained in the open set $ \Omega_{0} $. Without loss of generality, we may assume that $ \nabla Q(S) \neq 0 $ since otherwise we consider the local solution at another point at which $ \nabla Q $ does not vanish. After change of variables, we may still treat that point as the south pole. We would like to mention that the extremal case is $ \nabla Q \equiv 0 $ on $ \mathbb{S}^{n} $. For this case, the problem is reduced to the Yamabe problem.

Note that $ n(n - 1) $ is the scalar curvature on $ \mathbb{S}^{n} $ with respect to the standard metric $ g_{\mathbb{S}^{n}} $. We choose the stereographic projection map $ \sigma_{1} $ to localize the PDE, it follows that the south pole becomes the origin of $ \R^{n} $. Denote $ \Omega_{1, \epsilon} = \sigma_{1}(\Omega_{0, \epsilon}) $. By the conformal invariance (\ref{local:eqn24}) and the parametrization of $ g_{\mathbb{S}^{n}} $ in terms of stereographic projection, the Yamabe equation in (\ref{sphere:eqn3}) is equivalent to
\begin{equation*}
Q u^{P-1} = -a\Delta_{g_{\mathbb{S}^{n}}} u + n(n - 1) u = \Phi^{1 - p} (-a\Delta_{e}) (\Phi u) \Rightarrow -a\Delta_{e}(\Phi u) = Q (\Phi u)^{p-1}.
\end{equation*}
The function $ \Phi $ above is defined in (\ref{sphere:eqn2}). Thus the Dirichlet problem (\ref{sphere:eqn3}) is reduced to the following Dirichlet problem
\begin{equation}\label{sphere:eqn4}
\begin{split}
-a\Delta_{e} v(x) & = Q(x) v^{p-1}(x) \; {\rm in} \; \Omega_{1, \epsilon}; \\
v(x) & \equiv 0 \; {\rm on} \; \partial \Omega_{1, \epsilon}; \\
v(x) & > 0 \; {\rm in} \; \Omega_{1, \epsilon}.
\end{split}
\end{equation}
It is equivalent to say that if there exists some $ v $ that solves (\ref{sphere:eqn4}), then $ \Phi^{-1} v $ solves (\ref{sphere:eqn3}). On the other hand, the function $ \Phi u $ solves (\ref{sphere:eqn4}) provided that $ u $ solves (\ref{sphere:eqn3}). Due to Theorem \ref{local:thm4}, we conclude that (\ref{sphere:eqn4}) admits a solution $ v $.  Hence (\ref{sphere:eqn3}) has a solution $ u : = \Phi^{-1} v $. But we can still apply the stereographic projection map $ \sigma_{2} $ to parametrize $ g_{\mathbb{S}^{n}} $. Since after both $ \sigma_{1} $ and $ \sigma_{2} $, the conformal factors are the same. It follows that
\begin{equation}\label{sphere:eqn5}
\begin{split}
-a\Delta_{e} v(y) & = Q(y) v^{p-1}(y) \; {\rm in} \; \Omega_{2, \epsilon}; \\
v(y) & \equiv 0 \; {\rm on} \; \partial \Omega_{2, \epsilon}; \\
v(y) & > 0 \; {\rm in} \; \Omega_{2, \epsilon}
\end{split}
\end{equation}
has a solution. In particular, this solution of (\ref{sphere:eqn5}) has the same expression as the solution $ v $ of (\ref{sphere:eqn4}). We denote both of them by $ v $. This solution $ v $ comes from the analysis of the functional
\begin{equation*}
J(v) = \frac{a}{2} \int_{\Omega_{1, \epsilon}} \lvert \nabla v \rvert^{2} dx - \frac{1}{p} \int_{\Omega_{1, \epsilon}} Q u^{p} dx.
\end{equation*}
For details, we refer to \cite{CFP}. We would like to mention that the solution $ v $ is just a critical point of the functional above, and cannot be the minimizer unless the domain is $ \R^{n} $.

We recall from Remark \ref{local:re5} that the solution $ v $ has a one-to-one correspondence to $ Q $ in the sense that the following two quantities
\begin{equation*}
\frac{\nabla Q(0)}{\lvert \nabla Q(0) \rvert}, \frac{Q(0)}{\lvert \nabla Q(0) \rvert}
\end{equation*}
determines the solution $ v $ and vice versa. When $ \nabla Q(0) = 0 $, this is called ``critical point at infinity" given by A. Bahri. Since $ v $ solves both (\ref{sphere:eqn4}) and (\ref{sphere:eqn5}), it determines the quantities above. We note that
\begin{equation*}
\sigma_{1}^{-1}(0) = S, \sigma_{2}^{-1}(0) = N \Rightarrow Q(x) |_{x = 0} = C Q(y) |_{y = 0}, \nabla Q(x) \bigg|_{x = 0} = C \nabla Q(x) \bigg|_{x = 0}
\end{equation*}
for some constant $ C > 0 $. Actually this constant $ C = 1 $ since we could do the same argument starting at a neighborhood of north pole with the same size and shape, and use $ \sigma_{2} $ instead of $ \sigma_{1} $. Another way to see this is to apply some $ U \in O(n + 1) $ to $ \Omega_{0} $ such that $ U(S) = N $, and note that the spherical Laplacian $ -\Delta_{g_{\mathbb{S}^{n}}} $ is $ O(n + 1) $-invariant in the sense that
\begin{equation*}
\left( -\Delta_{g_{\mathbb{S}^{n}}} f \right) \circ U = -\Delta_{g_{\mathbb{S}^{n}}} (f \circ U), f \in \calC^{\infty}(\mathbb{S}^{n}).
\end{equation*}
Recall that
\begin{equation*}
x = \sigma_{1}(\xi, \tau) = \frac{\xi}{1 - \tau}, y = \sigma_{2}(\xi, \tau) = \frac{\xi}{1 + \tau}.
\end{equation*}
It follows from the argument above that
\begin{equation}\label{sphere:eqn6}
Q(N) = Q(S), \nabla_{\xi_{i}} Q \bigg|_{N} =  \nabla_{\xi_{i}} Q \bigg|_{S}, i = 1, \dotso, n, \nabla_{\tau} Q \bigg|_{N} = -  \nabla_{\tau} Q \bigg|_{S}.
\end{equation}
Since $ \mathbb{S}^{n} $ is $ O(n + 1) $-invariant, we would like to show that (\ref{sphere:eqn6}) must hold at every pair of antipodal points at which $ \nabla Q \neq 0 $, provided that the Dirichlet problem (\ref{sphere:eqn3}) admits a solution. To show this, we need the following definition.
\begin{definition}\label{sphere:def1}
Let $ Q : \mathbb{S}^{n} \rightarrow \R $ be a nonnegative, smooth function. Let $ P $ and $ P' $ be any pair of antipodal points. We say that the function $ Q $ satisfies the {\bf{CONDITION A}} if

(i) either at antipodal points $ P $ and $ P' $ we have
\begin{equation}\label{sphere:eqn7}
\begin{split}
& \nabla^{\mathbb{S}^{n}} Q(P) \neq 0 \Rightarrow \nabla^{\mathbb{S}^{n}} Q(P') \neq 0 \\
\Rightarrow & Q(P) = Q(P'), \nabla_{\xi_{i}}^{\R^{n+1}} Q \bigg|_{P} = \nabla_{\xi_{i}}^{\R^{n+1}} Q \bigg|_{P'}, i = 1, \dotso, n, \nabla_{\tau}^{\R^{n+1}} Q \bigg|_{P} = - \nabla_{\tau}^{\R^{n+1}} Q \bigg|_{P'};
\end{split}
\end{equation}

(ii) or at antipodal points $ P $ and $ P' $ we have
\begin{equation*}
\nabla Q(P) = \nabla Q(P') = 0.
\end{equation*}

(iii) or if $ \nabla Q \equiv 0 $ on $ \mathbb{S}^{n} $, then $ Q $ is a positive constant.
\end{definition}
\begin{remark}\label{sphere:re1}
The argument above shows that if $ Q $ satisfies CONDITION A, then (\ref{sphere:eqn3}) has a solution for some neighborhood $ \Omega_{0, \epsilon} $.
\end{remark}
\begin{remark}\label{sphere:re2}
(i) We list several functions $ Q $ that satisfy the CONDITION A. One obvious candidate is any globally positive constant function. 

Denoting $ (\xi, \tau) $ be the coordinates on $ \mathbb{S}^{n} $ inherited from the ambient space $ \R^{n} $. Next, we see that the restrictions of the functions $ Q(\xi, \tau) = \xi_{i}^{2m}, i = 1, \dotso, n, m \in \mathbb{N} $ or $ Q(\xi, \tau) = \tau^{2m}, m \in \mathbb{N} $ on $ \mathbb{S}^{n} $ satisfy the CONDITION A. So are the finite linear combinations of them, as long as the coefficients are chosen so that the functions are positive on $ \mathbb{S}^{n} $. So are the infinite linear combinations of them, as long as the functions are bounded and positive.

(ii) The forbidden functions are the restrictions of the functions $ \xi_{i}^{2m + 1}, \tau^{2m + 1}, m \in \mathbb{N} $ on $ \mathbb{S}^{n} $, and linear combinations of them.
\end{remark}
\medskip

The local situation we discussed above for the Dirichlet problem (\ref{sphere:eqn3}) on a punched neighborhood of south pole leads to the following result.
\begin{proposition} Let $ Q : \mathbb{S}^{n} \rightarrow \R $ be a nonnegative, smooth function such that $ \nabla Q \not\equiv 0 $ on $ \mathbb{S}^{n} $. Assume that there exists an open subset $ \Omega_{0} \subset \mathbb{S}^{n} $ that contains a point $ P $ at which $ \nabla Q \neq 0 $, and an $ \epsilon > 0 $ such that the Dirichlet problem (\ref{sphere:eqn3}) has a solution. Assume further that $ \Omega_{0} $ does not contain the antipodal point of $ P $. Then $ Q $ must satisfy the CONDITION A.
\end{proposition}
\begin{proof}
Let $ \Omega_{0, \epsilon} $ be the region on which (\ref{sphere:eqn3}) has a solution for some $ \epsilon $. There exists an element $ U \in SO(n + 1) $ such that $ U(P) = S $. Denote new coordinates $ y = Ux $. The solution of (\ref{sphere:eqn3}) is a critical point of the functional
\begin{equation*}
J(u) = \frac{a}{2} \int_{\Omega_{0, \epsilon}} \lvert \nabla_{g_{\mathbb{S}^{n}}} u \rvert^{2} \dvol + \frac{n(n - 1)}{2} \int_{\Omega_{0, \epsilon}} u^{2} \dvol - \frac{1}{p} \int_{\Omega_{0, \epsilon}} Q u^{p-1} \dvol
\end{equation*}
The spherical Laplacian is $ O(n + 1) $-invariant. It is well known that the volume form is $ O(n + 1) $-invariant. Since $ \det U = 1 $, it follows that $ u(y) $ is also a critical point of $ J(u) $ with respect to the new domain $ U(\Omega_{0, \epsilon}) $. It means that $ u $ solves the PDE
\begin{equation*}
-a\Delta_{g_{\mathbb{S}^{n}}} u + n(n - 1) u = Q u^{p-1} \; {\rm in} \; U(\Omega_{0, \epsilon}), u > 0 \; {\rm in} \; U(\Omega_{0, \epsilon}), u = 0 \; {\rm on} \; \partial U(\Omega_{0, \epsilon}).
\end{equation*}
By the argument above, $ Q $ must satisfy (\ref{sphere:eqn6}) at $ N $ and $ S $. We can apply $ SO(n + 1) $-action on the PDE and the domain so that the antipodal pairs $ N, S $ can be mapped to any pairs at one of which $ \nabla Q $ does not vanish. Therefore (\ref{sphere:eqn7}) holds when the gradient of $ Q $ does not vanish. It follows that $ Q $ must satisfy CONDITION A. 
\end{proof}
Now we show that every function $ S \in \calC^{\infty}(\mathbb{S}^{n}) $ that satisfies the CONDITION A in Definition \ref{sphere:def1} is realized a prescribed scalar curvature function of some metric $ g \in [g_{\mathbb{S}^{n}}] $.
\begin{theorem}\label{sphere:thm1}
Let $ (\mathbb{S}^{n}, g_{\mathbb{S}^{n}}) $ be the $ n $-sphere with standard metric, $ n \geqslant 3 $. Assume that the function $ S \in \calC^{\infty}(\mathbb{S}^{n}) $ satisfies the CONDITION A in Definition \ref{sphere:def1}. Then the function $ S $ can be realized as a prescribed scalar curvature function of some metric $ \tilde{g} $, pointwise conformal to $ g_{\mathbb{S}^{n}} $.
\end{theorem}
\begin{proof}
If $ Q $ is a globally constant function, then there is nothing to proof. If $ Q $ satisfies the CONDITION A and $ \nabla Q(P) \neq 0 $, we conclude by Theorem \ref{local:thm4} that there exists a region $ \Omega_{0, \epsilon} $ such that (\ref{sphere:eqn3}) has a solution.

The rest of the proof is exactly the same as in Theorem \ref{closed:thm2} and Corollary \ref{closed:cor1}.
\end{proof}
\medskip

There have been many obstructions of prescribed scalar curvature functions on $ \mathbb{S}^{n} $ given in earlier papers. For example, there are Kazdan-Warner obstruction in \cite{KW}, the obstruction in terms of conformal Killing field given by Bourguignon and Ezin \cite{BE}, and other type of obstruction in terms prescribed Morse scalar curvature function given by Malchiodi and Mayer \cite{MaMa}. We now show that if $ S $ satisfies CONDITION A, then it does not belong to any category of obstructions mentioned above.

We know that $ \mathbb{S}^{n} \subset \R^{n+1} $. We denote coordinates on $ \R^{n + 1} $ be $ (z_{1}, \dotso, z_{n + 1}) $. We first check the Kazdan-Warner obstruction, which is
\begin{equation*}
\int_{\mathbb{S}^{n}} \left(\nabla_{g_{\mathbb{S}^{n}}}  H \cdot \nabla_{g_{\mathbb{S}^{n}}}  S \right) u^{p} \dvol = 0
\end{equation*}
if $ u $ solves the Yamabe equation $ -a\Delta_{g_{\mathbb{S}^{n}}} u + n(n - 1) u = S u^{p-1} $ on $ \mathbb{S}^{n} $. As an example, we choose $ S = z_{j}^{2k} $ for some $ k \in \mathbb{N} $ and $ H $ is a linear combination of coordinate functions. It is clear that $ \nabla_{g_{\mathbb{S}^{n}}} \cdot \nabla_{g_{\mathbb{S}^{n}}} z_{j}^{2k} $ will change sign. In general, if $ S $ is monotone along the direction of $ H $, then either $ \nabla_{g_{\mathbb{S}^{n}}} S \equiv 0 $ or there exists at least one pair of antipodal points at which $ \nabla S $ are the same in all directions by pulling back to coordinates in ambient spaces, contradicts to CONDITION A.

Now we check the obstruction in terms of conformal killing fields given by Bourguignon and Ezin \cite{BE}. A vector field $ X_{g} $ with respect to the metric $ g $ is a conformal Killing field with respect to the closed manifold $ (M, g) $ if and only if 
\begin{equation*}
\mathcal{L}_{X} g = f g
\end{equation*}
for some smooth function $ f \in \calC^{\infty}(M) $. The obstruction given by Bourguignon and Ezin \cite{BE} states that
\begin{equation*}
\int_{M} X_{g} \cdot \nabla R_{g} \dvol = 0
\end{equation*}
provided that $ R_{g} $ is the scalar curvature with respect to $ g $. Note that any conformal vector field with respect to $ g $ is still a conformal vector field after conformal change $ g \mapsto u^{2} g $. It follows from the calculation that
\begin{equation*}
\mathcal{L}_{X_{g}} (u^{2} g )= \left( \mathcal{L}_{X_{g}} u^{2} \right)g + u^{2} \mathcal{L}_{X} g = \left( X_{g}(u^{2}) + fu^{2} \right) g.
\end{equation*}
Let's check the space of conformal Killing field on $ \mathbb{S}^{n} $ with respect to the standard metric $ g_{\mathbb{S}^{n}} $. It is a fact that the space of conformal Killing field $ \mathfrak{conf}(\mathbb{S}^{n}) $ is isomorphic to $ SO(n + 1, 1) $. The dimension of $ SO(n + 1, 1) $ is $ \frac{(n + 1)(n + 2)}{2} $. On one hand, Killing fields are definitely conformal Killing fields. The isometries of $ \mathbb{S}^{n + 1} $ forms a group that is isomorphic to $ O(n + 1) $, thus the dimensions of the collection of Killing fields $ \mathfrak{isom}(\mathbb{S}^{n}) $ on $ \mathbb{S}^{n} $ is $ \frac{n (n + 1)}{2} $. On the other hand, it is straightforward to check that the vector fields
\begin{equation*}
X_{a, g_{\mathbb{S}^{n}}} = a - (a \cdot z)z
\end{equation*}
are conformal Killing fields with respect to the standard metric $ g_{\mathbb{S}^{n}} $. There are exactly $ (n + 1) $ linearly independent $ X_{a, g_{\mathbb{S}^{n}}} $ due to the choice of $ a \in \R^{n + 1} $. Furthermore, any flow generated by $ X_{a, g_{\mathbb{S}^{n}}} $ is not an isometry and thus they do not belong to $ \mathfrak{isom}(\mathbb{S}^{n}) $. Denote
\begin{equation*}
\mathcal{X} = \lbrace X_{a, g_{\mathbb{S}^{n}}} | a \in \R^{n+1}, a \neq 0 \rbrace.
\end{equation*}
According to dimensional reason, it follows that
\begin{equation*}
\mathfrak{conf}(\mathbb{S}^{n}) = \mathfrak{isom}(\mathbb{S}^{n}) \oplus \mathcal{X}.
\end{equation*}
Thus every conformal Killing field is a linear combination of a Killing field and an element of $ \mathcal{X} $. Let $ R_{g} $ be the prescribed scalar curvature with respect to $ g \in [g_{\mathbb{S}^{n}}] $, it follows from the obstruction of Bourguignon and Ezin that
\begin{equation}\label{sphere:eqn8}
\int_{\mathbb{S}^{n}} X_{a, g_{\mathbb{S}^{n}}} \cdot \nabla R_{g} \text{dvol}_{g_{\mathbb{S}^{n}}} = 0
\end{equation}
Only the $ \mathcal{X} $ part contributes nontrivially in (\ref{sphere:eqn8}) since (\ref{sphere:eqn8}) holds for any killing fields in $ \mathfrak{isom}(\mathbb{S}^{n}) $ for all functions $ R_{g} \in \calC^{\infty}(\mathbb{S}^{n}) $. 

As a motivation, assume that $ R_{g} $ is any linear combination of nonnegative even power coordinate functions in terms of $ (z_{1}, \dotso, z_{n+ 1}) $. Then $ \nabla R_{g} $ is a linear combination of positive odd power coordinate functions. According to the expression of $ X_{a, g_{\mathbb{S}^{n}}} $, it follows that the equation in (\ref{sphere:eqn8}) holds. In general, all functions satisfying the CONDITION A satisfies (\ref{sphere:eqn8}) due to the symmetry among all antipodal points on $ \mathbb{S}^{n} $.

Lastly, Machiodi and Mayer \cite{MaMa} constructed a function in terms of $ z_{n + 1} $ that has non-degenerate maximum at the North pole; in addition, all other critical points with positive Laplacian are accumulating near the South pole. We point out that if a function $ S $ satisfies the CONDITION A, then there would also be critical points with positive Laplacian near the North pole due to the symmetry of Hessian, which violates the properties of the function they constructed in \cite{MaMa}. 
\medskip

For the rest of this section, we discuss potential obstruction on prescribed scalar curvature function on all compact manifolds $ (\bar{M}, g) $ with dimensions at least 3. Without loss of generality, we may assume that $ \bar{M} $ is connected since otherwise we can discuss each connected components instead. According to the example $ n $-sphere, we see that the geometry and $ O(n + 1) $-invariance of $ \mathbb{S}^{n} $ as well as the spherical Laplacian result in the obstructions of prescribed scalar curvature. We now define the {\bf{SCENARIO A}} as follows for all compact manifolds, with or without boundary, provided that the first eigenvalues, $ \eta_{1} $ or $ \eta_{1}' $ respectively, of conformal Laplacians, are positive. We will see that when a is categorized in the SCENARIO A, it is locally conformal flat, then there will be obstructions for prescribed scalar curvature functions for some Yamabe metric in the type that generalizes the CONDITION A.
\begin{definition}\label{sphere:def2} Let $ (\bar{M}, g) $ be a connected, compact manifold, with or without smooth boundary, and with $ \dim \bar{M} \geqslant 3 $. We say that the manifold $ (\bar{M}, g) $ is of {\bf{SCENARIO A}} if there exist at least a pair of points $ \rho, \rho' \in \bar{M} $ and an open subset $ O \subset M \subset \bar{M} $ such that:

(i) The first eigenvalue of conformal Laplacian, possibly with Robin condition as mentioned in \S4, is positive; 

(ii) $ \rho \in O $ and $ \rho' \notin \bar{O} $, or vice versa;

(iii) There exists two charts $ (U_{1}, \sigma_{1}) $ and $ (U_{2}, \sigma_{2}) $ such that both maps $ \sigma_{1} $ and $ \sigma_{2} $ are conformal diffeomorphisms; furthermore, $ O \subset U_{1} \cap U_{2} $, $ \rho \in U_{1} $ and $ \rho' \in U_{2} $;

(iv) For $ g_{e} $ the Euclidean metric, we have
\begin{equation}\label{sphere:eqn9}
\left( \sigma_{1}^{-1} \right)^{*} g = \left( \sigma_{2}^{-1} \right)^{*} g = v^{2} g_{e}, \sigma_{1}(\rho) = \sigma_{2}(\rho')
\end{equation}
for some positive, smooth function $ v $.
\end{definition}
\begin{remark}
The $ n $-sphere $ (\mathbb{S}^{n}, g_{\mathbb{S}^{n}} ) $ satisfies the SCENARIO A. We would also like to mention that there might be more than one pair of points $ \rho, \rho' $ that satisfies the conditions (i), (ii), (iii) in Definition \ref{sphere:def2}.
\end{remark}
Immediately, we see that any manifold that is globally conformally flat does not satisfy the SCENARIO A, since in this case, the first eigenvalue, which is a conformal invariance, should be zero. It follows that the two charts are not the same and neither of them covers the whole manifold.

We want to classify the manifolds that satisfy the SCENARIO A. A beautiful and fundamental result, due to Schoen and Yau \cite{SY}, plays an important role here.
\begin{theorem}\label{sphere:thm2}\cite{SY}
Let $ (\bar{M}, g) $ be a compact, locally conformally flat manifold, without or with boundary. Let $ \tilde{\bar{M}} $ be the universal cover of $ \bar{M} $, respectively. If the scalar curvature of $ \bar{M} $ satisfies $ R_{g} > 0 $, then the universal cover $ \tilde{\bar{M}} $ is conformally embedded in $ (\mathbb{S}^{n}, g_{\mathbb{S}^{n}}) $. Moreover, the fundamental group $ \pi_{1}(\bar{M}) $ is isomorphic to some Kleinian group $ \Gamma \subset \mathfrak{conf}(\mathbb{S}^{n}) $ such that
\begin{equation}\label{sphere:eqn10}
\tilde{\bar{M}}  \cong \mathbb{S}^{n} \slash \Lambda, \bar{M} \cong \tilde{\bar{M}} \slash \Gamma.
\end{equation}
Here $ \Lambda $ is the limit set of the Kleinian group $ \Gamma $.
\end{theorem}
We now observe that when $ \eta_{1} $ or $ \eta_{1}' $ is positive, $ \bar{M} $ is covered by at least two charts. We start with two simple facts.
\begin{proposition}\label{sphere:prop2}
Let $ (\bar{M}, g) $ be a compact manifold that satisfies the SCENARIO A, with or without boundary, with $ \dim \bar{M} \geqslant 3 $. If $ \bar{M} $ is only covered by two different charts, each is homeomorphic to $ \R^{n} $, then $ \bar{M} $ is isomorphic to $ \mathbb{S}^{n} $.
\end{proposition}
\begin{proof}
According to a result of Brown \cite{MB}, we conclude that any compact manifold with dimension $ n = \dim \bar{M} $ at least 3 that admits only two charts, both homeomorphic to $ \R^{n} $, is homeomorphic to $ \mathbb{S}^{n} $. Since $ \pi_{1} $ is a topological invariance, it follows that $ \pi_{1}(\bar{M}) = 0 $. Since $ \bar{M} $ satisfies the SCENARIO A, it is immediate that $ \bar{M} $ is locally conformally flat. Since we assume that $ \eta_{1} > 0 $, we may assume, without loss of generality, that $ R_{g} > 0 $ on $ \bar{M} $. Due to Theorem \ref{sphere:thm2}, it follows that
\begin{equation*}
\bar{M} \cong \mathbb{S}^{n}.
\end{equation*}
\end{proof}
We see that in the special case above, the fact that $ \bar{M} $ is locally conformally flat is trivial, we want to show in general that $ (\bar{M}, g) $ is locally conformally flat if the manifold lies in the category of the SCENARIO A.
\begin{proposition}\label{sphere:prop3}
Let $ (\bar{M}, g) $ be a connected, compact manifold that satisfies the SCENARIO A, with or without boundary, with $ \dim \bar{M} \geqslant 3 $. Then $ (\bar{M}, g) $ must be a locally conformally flat manifold.
\end{proposition}
\begin{proof} According to our discussion above, there are at least two charts $ (U_{1}, \sigma_{1}) $ and $ (U_{2}, \sigma_{2}) $ of $ \bar{M} $, each $ \sigma_{i}, i = 1, 2 $ is a conformal diffeomorphism.

If $ \bar{M} = U_{1} \cup U_{2} $ then we are done. If not, then there must exist a chart $ (U_{3}, \sigma_{3}) $ that is overlapping with one of $ U_{i} $'s, $ i = 1, 2 $. Assume that $ U_{1} \cap U_{3} \neq \emptyset $. We can see that
\begin{equation*}
\left(\sigma_{3}^{-1} \right)^{*} g = \left( \sigma_{3}^{-1} \circ \sigma_{1} \circ \sigma_{1}^{-1} \right)^{*}g = \left( \sigma_{1}^{-1} \right)^{*} \circ \left( \sigma_{3}^{-1} \circ \sigma_{1} \right)^{*} g.
\end{equation*}
Note that the map $ \sigma_{3}^{-1} \circ \sigma_{1} : U_{1} \rightarrow U_{3} $ is a conformal diffeomorphism due to Liouville's theorem, thus by the assumption in the SCENARIO A, we conclude that $ \left(\sigma_{3}^{-1} \right)^{*} g $ within $ U_{1} \cap U_{3} $ is also a conformal diffeomorphism, and thus can be extended to the whole chart $ U_{3} $. We can repeat this argument for all charts of $ \bar{M} $. Therefore, $ (\bar{M}, g) $ is locally conformally flat.
\end{proof}
We therefore conclude as follows
\begin{proposition}\label{sphere:prop4}
A connected, compact manifold $ (\bar{M}, g) $, possibly with boundary, with dimensions at least 3 that satisfies the SCENARIO A must be a closed manifolds, conformally diffeomorphic to $ \mathbb{S}^{n} \slash \Gamma $ for some discrete Kleinian group $ \Gamma $.
\end{proposition}
\begin{proof} Since $ (\bar{M}, g) $ is locally conformally flat due to Proposition \ref{sphere:prop3}, it follows from Theorem \ref{sphere:thm2} that the universal cover of $ \bar{M} $ is conformally embedded in $ \mathbb{S}^{n} $. Furthermore, $ \bar{M} = \mathbb{S}^{n} \slash G $ for some group $ G $ related to the Kleinian group. It follows that $ \mathbb{S}^{n} \slash G $ has positive constant sectional curvature. It is easy to make the covering map
\begin{equation*}
\pi : \tilde{\bar{M}} \rightarrow \bar{M}
\end{equation*}
to be an isometry and thus $ \tilde{\bar{M}} $ has positive Ricci curvature, i.e. $ Ric(X, X) > 0 $ for all unit vector field $ X $ on $ \tilde{\bar{M}} $. Due to Hopf-Rinow as well as the simply connectedness, the universal cover $ \tilde{\bar{M}} $ is a complete, connected manifold with positive lower bound of Ricci curvature tensor. It follows from Bonnet-Myers that $ \tilde{\bar{M}} $ must be a compact manifold. In general, any embedding from a compact manifold to a connected manifold is a surjective and hence a homeomorphism, see e.g. \cite[Cor.~2B.4]{AH}. It follows that $ \tilde{\bar{M}} $ is homeomorphic to $ \mathbb{S}^{n} $. Therefore the limiting set $ \Lambda $ in Theorem \ref{sphere:thm2} is a trivial group. Hence $ \bar{M} = \mathbb{S}^{n} \slash \Gamma $ for some discrete Kleinian group $ \Gamma $. Furthermore it follows that all manifolds which satisfy the SCENARIO A must be a closed manifold.
\end{proof}

\begin{remark}\label{sphere:re3}
(i) In Proposition \ref{sphere:prop4}, it is possible that the Kleinian group is the trivial group $ \lbrace 1 \rbrace $.

(ii) We are not saying that there would have no obstruction of prescribed scalar curvature functions on other manifolds. For example, we showed in \S2 and \S3 that the prescribed curvature functions on closed manifolds and compact manifolds with non-empty boundary, respectively, are chosen from collections $ \mathcal{A}_{1} $ and $ \mathcal{A}_{1}' $, respectively. We only have $ \calL^{r} $-density argument in terms of general smooth functions on $ M $ or $ \bar{M} $. 
\end{remark}
\medskip

Inspired by the CONDITION A on $ \mathbb{S}^{n} $, we generalize it to the following {\bf{CONDITION B}}. To set notations, let $ Q :  \mathbb{S}^{n} \slash \Gamma \rightarrow R $ be a smooth function, let $ Q' $ be the lift of $ Q $ to $ \mathbb{S}^{n} $, and let $ Q'' : \mathbb{R}^{n+1} \backslash \lbrace 0 \rbrace \rightarrow \R $ be defined by $ Q''(x) = Q' \left( \frac{x}{\lvert x \rvert} \right) $.
\begin{definition}\label{sphere:def3}
Let $ (M, g) = (\mathbb{S}^{n} \slash \Gamma, g) $ be a connected, compact manifold defined in the sense of Proposition \ref{sphere:prop4} and Theorem \ref{sphere:thm2}. Consider the smooth function $ Q : \mathbb{S}^{n} \slash \Gamma \rightarrow \R $. We say that the function $ Q $ satisfies the {\bf{CONDITION B}} if for any pairs of points $ \rho, \rho' $ that satisfies Definition \ref{sphere:def2},

(i) either at the pair of points $ \rho $ and $ \rho' $ we have
\begin{equation}\label{sphere:eqn11}
\begin{split}
& \nabla^{M} Q(\rho) \neq 0 \Rightarrow \nabla^{M} Q(\rho') \neq 0 \\
\Rightarrow & Q(\rho) = Q(\rho'), \nabla_{\xi_{i}}^{\R^{n+1}} Q'' \bigg|_{\rho} = \nabla_{\xi_{i}}^{\R^{n+1}} Q'' \bigg|_{\rho'}, i = 1, \dotso, n, \nabla_{\tau}^{\R^{n+1}} Q'' \bigg|_{\rho} = - \nabla_{\tau}^{\R^{n+1}} Q'' \bigg|_{\rho'},
\end{split}
\end{equation}
where $ \tau $ is unit vector in $ \R^{n+1} $ from $ \rho $ to $ \rho' $, and $ \lbrace \xi_{1}, \dotso, \xi_{n}, \tau \rbrace $ is a basis of $ \R^{n+1} $.

(ii) or at antipodal points $ \rho $ and $ \rho' $ we have
\begin{equation*}
\nabla^{M} Q(\rho) = \nabla^{M} Q(\rho') = 0.
\end{equation*}

(iii) or if $ \nabla Q \equiv 0 $ on $ \mathbb{S}^{n} \slash \Gamma $, then $ Q $ is a positive constant.
\end{definition}
\begin{remark}\label{sphere:re5}
The $ \tau $-direction in the definition above is well-defined since $ \R^{n + 1} $ is also the ambient space of $ \mathbb{S}^{n} \slash \Gamma $. In addition, the condition (iii) in Definition \ref{sphere:def3} is clear since the sectional curvature on $ \mathbb{S}^{n} \slash \Gamma $ is also a positive constant.
\end{remark}

\section{Prescribed Scalar Curvature on Compact Manifolds, Revisited}
In this section, we improve our results of Kazdan-Warner problem on closed manifolds $ (M, g) $ that are not homeomorphic to the $ n $-sphere or compact manifolds with non-empty smooth boundary by removing one restriction in the set $ \mathcal{A}_{1} $ or $ \mathcal{A}_{1}' $ defined in (\ref{closed:eqnr1}) and (\ref{compact:eqn2}), respectively. Throughout this section, we assume that the dimensions of manifolds are at least $ 3 $. Explicitly speaking, we no longer require that the choice of open set $ O $, on which the function $ S $ should be positive constant, covers at least one point at which the Weyl tensor does not vanish. We may also assume that the manifolds are connected, since otherwise we just discuss each connected components instead.

We recall that the set $ \mathcal{A} $ is given as follows:
\begin{equation}\label{closed1:eqn1}
\begin{split}
\mathcal{A} & : =  \lbrace S \in \calC^{\infty}(M) : S \equiv \lambda > 0 \; {\rm on} \; \bar{O}, \text{$ \lambda > 0 $ is an arbitrary constant,} \\
& \qquad \text{ $ O \subset \bar{O} \subset M $ is any open submanifold with smooth $ \partial O $. } \rbrace
\end{split}
\end{equation}

Also recall that the Kazdan-Warner problem on closed manifolds is equivalent to find a positive, smooth solution $ u $ of the following PDE
\begin{equation}\label{closed1:eqn2}
-a\Delta_{g} u + R_{g} u = S u^{p-1} \; {\rm in} \; M.
\end{equation}

To release the restriction of the choice of $ S $ as mentioned just above, we need the following local result due to Bahri and Coron \cite{BC}.
\begin{theorem}\label{closed1:thm1}\cite[Thm.~1]{BC}
Let $ \Omega $ be a bounded, regular and connected open subsets of $ \R^{n} $ with $ n \geqslant 3 $. If there exists a positive integer $ d \neq n $ such that $ H_{d}(\Omega, \mathbb{Z}_{2}) \neq 0 $, then the following PDE admits a solution for any constant $ \lambda > 0 $.
\begin{equation}\label{closed1:eqn3}
-a\Delta_{e} u = \lambda u^{p-1} \; {\rm in} \; \Omega, u \equiv 0 \; {\rm on} \; \partial \Omega, u > 0 \; {\rm in} \; \Omega.
\end{equation}
\end{theorem}
\begin{remark}\label{closed1:re1} Kazdan and Warner pointed in \cite{KW} that $ \Omega $ can be chosen as annular region in $ \R^{n} $.
\end{remark}
\medskip

With the aid of Theorem \ref{closed1:thm1}, we can state the general theorem of prescribed scalar curvature on closed manifolds, which is an improvement of Corollary \ref{closed:cor1}. Note that when $ M $ is not $ \mathbb{S}^{n} $ or some quotient of $ \mathbb{S}^{n} $ that satisfies the SCENARIO A, we have no restriction of prescribed scalar curvature function in the sense of CONDITION A or CONDITION B, since the local solution is then completely determined by the coefficient function, elliptic operator and nonlinear term locally. If $ M $ is either $ \mathbb{S}^{n} $ or some quotient of $ \mathbb{S}^{n} $ in the category of the SCENARIO A, then the functions that are in the type of the CONDITION A or CONDITION B can be realized as prescribed scalar curvature functions of some Yamabe metric.
\begin{theorem}\label{closed1:thm2}
Let $ (M, g) $ be a connected, closed Riemannian manifold, with $ n = \dim M \geqslant 3 $. Let $ \eta_{1} $, the first eigenvalue of $ \Box_{g} $, be positive.

(i) Let $ S \in \mathcal{A} $ be a smooth function satisfies $ S = \lambda > 0 $ for some positive constant $ \lambda $ on some open subset $ O \subset M $. If $ M $ does not satisfy the SCENARIO A, then there exists a positive function $ u \in \calC^{\infty}(M) $ that solves (\ref{closed1:eqn3});

(ii) Let $ S \in \calC^{\infty}(M) $ that satisfies the CONDITION B.  If $ M $ is either $ \mathbb{S}^{n} $ or some quotient of $ \mathbb{S}^{n} $ that satisfies the SCENARIO A, then there exists a positive function $ u \in \calC^{\infty}(M) $ that solves (\ref{closed1:eqn3});
\end{theorem}
\begin{proof} In both cases, any globally positive constant function $ S $ is a trivial case. So assume not.

For (i), there are several cases in terms of the choice of $ O $.

If the metric $ g $ is not locally conformally flat in $ O $, the proof is exactly the same as in \S3.

If the manifold $ (M, g) $ is not locally conformally flat, but the Weyl tensor with respect to $ g $ does vanish on $ O $. We then choose some $ \Omega \subset O $ that is small enough so that it is covered by a single chart $ (U, \psi) $; in addition, we can choose $ \Omega $ to be an annular region in $ \R^{n} $. We then apply the conformal invariance of conformal Laplacian and Theorem \ref{closed1:thm1}, which follows that
\begin{equation}\label{closed1:eqn4}
-a\Delta_{g} u + R_{g} u = \lambda u^{p-1} \; {\rm in} \; \Omega, u = 0 \; {\rm on} \; \partial \Omega
\end{equation}
has a smooth, positive solution. The rest of the proof is the same as in \S3.

If the manifold $ (M, g) $ is locally conformally flat and not $ \mathbb{S}^{n} $, we apply the same argument as above and in \S3. The claim then follows.

For (ii), the manifold is locally conformally flat, then we apply Theorem \ref{local:thm4} to obtain the local solution of (\ref{closed1:eqn4}), the rest are the same as in \S3.
\end{proof}
\medskip

Recall the set $ \mathcal{B} $ we introduced in \S3:
\begin{equation*}
\begin{split}
\mathcal{B} & : = \lbrace f \in \calC^{\infty}(M) : f > 0 \; \text{somewhere in M}, \\
& -a\Delta_{g} u + R_{g} u = f u^{p-1} \; \text{admits a real, positive, smooth solution $ u $ with $ \eta_{1} > 0 $}. \rbrace
\end{split}
\end{equation*}
With the analysis in \S5, we can improve our results of $ \calL^{r}$-dense of $ \mathcal{A}_{1} \subset \mathcal{B} $ in Theorem \ref{closed:thm3} to $ \calC^{0} $-dense.
\begin{theorem}\label{closed1:thm3}
Let $ (M, g) $ be a connected, closed manifold that does not satisfy the SCENARIO A, $ n =  \dim M \geqslant 3 $. Assume that $ \eta_{1} > 0 $. Then $ \mathcal{A} \subset \mathcal{B} $; in addition, $ \mathcal{A} $ is $ \calC^{0} $-dense in $ \mathcal{B} $.
\end{theorem}
\begin{proof}
Given any $ \epsilon > 0 $. Pick up a smooth function $ f \in \mathcal{B} $, we pick any connected, open geodesic ball $ O \subset M $ with radius $ \delta \ll 1 $ such that $ f > 0 $ on $ O $. If $ f $ is a globally positive constant function, it is trivial. Assume not. We may choose $ \delta $ small enough such that
\begin{equation*}
\lvert \sup_{O} f - \inf_{O} f \rvert < \epsilon.
\end{equation*}
Denote $ O' \subset O \subset M $ be a concentric geodesic ball with radius $ \frac{\delta}{2} $. Define
\begin{equation*}
f^{*} = \begin{cases} \frac{\sup_{O} f + \inf_{O} f}{2}, & {\rm in} \; O' \\ f, & {\rm in} \; M \backslash \bar{O}' \end{cases}.
\end{equation*}
Mollifying $ f^{*} $ with $ \phi_{\zeta} $ with $ 0 < \zeta \ll \frac{\delta}{2} $, denote
\begin{equation*}
\bar{f} = f^{*} * \phi_{\zeta}.
\end{equation*}
It is straightforward that
\begin{equation*}
\bar{f} \in \mathcal{A}, \lVert \bar{f} - f \rVert_{\calC^{0}(M)} < \epsilon.
\end{equation*}
\end{proof}
\medskip

We now show similar results as Theorem \ref{closed1:thm2} and Theorem \ref{closed1:thm3} in compact manifolds with non-empty smooth boundary, with dimensions at least 3. Recall that this is equivalent to find a positive, smooth solution of the following boundary problem.
\begin{equation}\label{closed1:eqn5}
-a\Delta_{g} u + R_{g} u = S u^{p-1} \; {\rm in} \; M, \frac{\partial u}{\partial \nu} + \frac{2}{p-2} h_{g} u = 0 \; {\rm on} \; \partial M.
\end{equation}
Except some quotients of $ \mathbb{S}^{n} $, no other compact manifolds with non-empty smooth boundary satisfies SCENARIO A and hence no restriction of prescribed scalar curvature function in the sense of the CONDITION B generically. We can improve the result of prescribed scalar curvature with minimal boundary in Corollary \ref{compact:cor2}. Recall that the set
\begin{equation}\label{closed1:eqn6}	
\begin{split}
\mathcal{A}' & : = \lbrace S \in \calC^{\infty}(\bar{M}) : S \equiv \lambda > 0 \; {\rm on} \; \bar{O}, \text{$ \lambda > 0 $ is an arbitrary constant, } \\
& \qquad \text{ $ O \subset \bar{O} \subset M $ is an arbitrary open interior submanifold with smooth $ \partial O $} \rbrace.
\end{split}
\end{equation}
\begin{theorem}\label{closed1:thm4}
Let $ (\bar{M}, g) $ be a connected, compact Riemannian manifold with non-empty smooth boundary $ \partial M $, $ n = \dim \bar{M} \geqslant 3 $. Let $ \nu $ be the outward unit normal vector field along $ \partial M $. Let $ S \in \mathcal{A}' $ be a smooth function satisfies $ S = \lambda > 0 $ for some positive constant $ \lambda $ on some open subset $ O \subset M \subset \bar{M} $. Assume $ \eta_{1}' $, the first eigenvalue of $ \Box_{g} $ with Robin condition, be positive. Then there exists a positive function $ u \in \calC^{\infty}(\bar{M}) $ that solves (\ref{closed1:eqn6}).
\end{theorem}
\begin{proof} According to Proposition \ref{sphere:prop4}, every compact manifold with non-empty smooth boundary does not satisfy the SCENARIO A. For the rest of the proof, we still classify the argument into three cases: (a) the metric is not locally conformally flat in $ O $; (b) the manifold is not locally conformally flat but the metric is locally conformally flat in $ O $; (c) the manifold is locally conformally flat. The proofs are almost the same as discussed in Theorem \ref{closed1:thm2}, Theorem \ref{compact:thm2}, Corollary \ref{compact:cor1} and Corollary \ref{compact:cor2}.
\end{proof}
\medskip

Recall the set $ \mathcal{B}' $ we introduced in \S4:
\begin{equation*}
\begin{split}
\mathcal{B}' & : = \lbrace f \in \calC^{\infty}(\bar{M}) : f > 0 \; \text{somewhere in the interior M}, \\
& -a\Delta_{g} u + R_{g} u = f u^{p-1} \; {\rm in} \; M, \frac{\partial u}{\partial \nu} + \frac{2}{p-2} h_{g} u = 0 \; {\rm on} \; \partial M, \\
& \text{admits a real, positive, smooth solution $ u $ with $ \eta_{1}' > 0 $}. \rbrace
\end{split}
\end{equation*}
Due to the same argument as in Theorem \ref{closed1:thm3}, we conclude that
\begin{theorem}\label{closed1:thm5}
Let $ (\bar{M}, g) $ be a connected, compact manifold with non-empty smooth boundary, $ n =  \dim \bar{M} \geqslant 3 $. Assume that $ \eta_{1}' > 0 $. Then $ \mathcal{A}' \subset \mathcal{B}' $; in addition, $ \mathcal{A}' $ is $ \calC^{0} $-dense in $ \mathcal{B}' $.
\end{theorem}
\begin{proof}
The $ \calC^{0} $-closeness appears in some open subsets of the interior of the manifold and hence is a local argument, which is then exactly the same as in Theorem \ref{closed1:thm2}.
\end{proof}
\medskip

Theorem \ref{closed1:thm2} and Theorem \ref{closed1:thm4} covers most possible choices of the prescribed scalar curvature functions, but we have a little bit more flexibility when we consider the Kazdan-Warner problem on locally conformally flat compact manifolds which are not in the category of the SCENARIO A, which means that there is no restriction in terms of the CONDITION B.
\begin{theorem}\label{closed1:thm6}
Let $ (\bar{M}, g) $ be a connected, locally conformally flat, compact manifold, with or without boundary, with $ n = \dim \bar{M} \geqslant 3 $. Let the first eigenvalue of conformal Laplacian, possibly with homogeneous Robin condition, is positive. If $ \bar{M} $ does not satisfy the SCENARIO A, then any smooth function $ S \in \calC^{\infty}(\bar{M}) $ that is positive somewhere can be realized as a prescribed scalar curvature function of some metric $ \tilde{g} \in [g] $.
\end{theorem}
\begin{proof}
We choose a small enough region $ \Omega \subset M \subset \bar{M} $ such that there exists a point $ P \in \Omega $ with $ \nabla S(P) \neq 0 $. Consider
\begin{equation*}
\Omega_{\epsilon} = \lbrace x \in \Omega : \lvert x - P \rvert > \epsilon \rbrace.
\end{equation*}
We choose $ \epsilon_{\epsilon} $ small enough so that $ B_{\epsilon}(P) \subset \Omega $. We then apply Theorem \ref{local:thm4} to construct a local solution of the Yamabe equation with Dirichlet condition on $ \Omega_{\epsilon} $ since the manifold is locally conformally flat. The rest argument follows exactly the same as in either \S3 or \S4, respectively.
\end{proof}

\bibliographystyle{plain}
\bibliography{YamabessP}

\end{document}